\documentclass{article}

\usepackage[english]{babel}
\usepackage{geometry}
\usepackage[round]{natbib}

\usepackage{amsmath,amsfonts,amsthm,bm}
\usepackage{dsfont}
\usepackage{mathtools}
\usepackage{color}
\usepackage{xcolor}
\usepackage{thm-restate}
\usepackage{hyperref}
\usepackage{enumitem}
\usepackage{booktabs}
\usepackage{multirow}
\usepackage{makecell}
\usepackage{subfig}
\usepackage{bbm}

\usepackage{algorithm,algorithmic}
\usepackage[algo2e,ruled,vlined]{algorithm2e}

\usepackage{graphicx}

\def\vxh{{\hat{\bm{x}}}}

\newtheorem{lemma}{Lemma}

\newtheorem{assumption}{Assumption}

\newtheorem{definition}{Definition}

\newcommand{\innp}[1]{\left\langle#1\right\rangle}  

\usepackage{xcolor}
\definecolor{mypurple}{rgb}{0.67, 0.12, 0.47}

\def\1{\bm{1}}

\def\eps{{\epsilon}}

\def\vzero{{\bm{0}}}

\def\va{{\bm{a}}}

\def\vg{{\bm{g}}}

\def\vl{{\bm{l}}}

\def\vu{{\bm{u}}}
\def\vv{{\bm{v}}}
\def\vw{{\bm{w}}}
\def\vx{{\bm{x}}}
\def\vy{{\bm{y}}}

\def\mA{{\bm{A}}}
\def\mB{{\bm{B}}}

\def\mE{{\bm{E}}}

\def\mI{{\bm{I}}}

\def\mLambda{{\bm{\Lambda}}}

\def\mGamma{{\bm{\Gamma}}}

\DeclareMathAlphabet{\mathsfit}{\encodingdefault}{\sfdefault}{m}{sl}
\SetMathAlphabet{\mathsfit}{bold}{\encodingdefault}{\sfdefault}{bx}{n}

\def\gB{{\mathcal{B}}}

\def\gE{{\mathcal{E}}}

\def\gI{{\mathcal{I}}}

\def\gL{{\mathcal{L}}}

\def\gO{{\mathcal{O}}}

\def\gS{{\mathcal{S}}}
\def\gT{{\mathcal{T}}}

\def\sR{{\mathbb{R}}}

\newcommand{\E}{\mathbb{E}}

\newcommand{\R}{\mathbb{R}}

\DeclareMathOperator*{\argmax}{arg\,max}
\DeclareMathOperator*{\argmin}{arg\,min}

\newcommand{\footremember}[2]{%
    \footnote{#2}
    \newcounter{#1}
    \setcounter{#1}{\value{footnote}}%
}
\newcommand{\footrecall}[1]{%
    \footnotemark[\value{#1}]%
}

\usepackage[utf8]{inputenc} 
\usepackage[T1]{fontenc}    
\usepackage{hyperref}       
\usepackage{url}            
\usepackage{booktabs}       
\usepackage{amsfonts}       
\usepackage{nicefrac}       
\usepackage{microtype}      
\usepackage{xcolor}         

\title{Empirical Risk Minimization with Shuffled SGD:\\ A Primal-Dual Perspective and Improved Bounds\footnote{This material is based upon research supported in part by the U.\ S.\ Office of Naval Research under award number N00014-22-1-2348.}}

\author{Xufeng Cai\footremember{wisc}{Department of Computer Sciences, University of Wisconsin-Madison. XC (\href{mailto:xcai74@wisc.edu}{xcai74@wisc.edu}), CYL (\href{mailto:cylin@cs.wisc.edu}{cylin@cs.wisc.edu}), JD (\href{mailto:jelena@cs.wisc.edu}{jelena@cs.wisc.edu}).}
\and Cheuk Yin Lin\footrecall{wisc}
\and Jelena Diakonikolas\footrecall{wisc}
}
\date{}

\begin{document}

\maketitle

\begin{abstract}
Stochastic gradient descent (SGD) is perhaps the most prevalent optimization method  in modern machine learning. Contrary to the empirical practice of sampling from the datasets \emph{without replacement} and with (possible) reshuffling at each epoch, the theoretical counterpart of SGD usually relies on the assumption of \emph{sampling with replacement}. It is only very recently that SGD with sampling without replacement -- shuffled SGD -- has been analyzed. For convex finite sum problems with $n$ components and under the $L$-smoothness assumption for each  component function, there are matching upper and lower bounds, under sufficiently small -- $\gO(\frac{1}{nL})$ -- step sizes. Yet those bounds appear too pessimistic -- in fact, the predicted performance is generally no better than for full gradient descent -- and do not agree with the empirical observations. 
In this work, to narrow the gap between the theory and practice of shuffled SGD, we sharpen the focus from general finite sum problems to empirical risk minimization with linear predictors. This allows us to take a primal-dual perspective and interpret shuffled SGD as a primal-dual method with cyclic coordinate updates on the dual side. Leveraging this perspective, we prove fine-grained complexity bounds that depend on the data matrix and are never worse than what is predicted by the existing bounds. Notably, our bounds  predict much faster convergence than the existing analyses -- by a factor of the order of $\sqrt{n}$ in some cases. We empirically demonstrate that on common  machine learning datasets our bounds are indeed much tighter. We further  extend our analysis to  nonsmooth convex problems {and more general finite-sum problems}, with similar improvements.
\end{abstract}

\section{Introduction}\label{sec:intro}
Originally proposed in~\citet{robbins1951stochastic}, SGD has been broadly studied in machine learning due to its effectiveness in large-scale settings, where full gradient computations are often computationally prohibitive. 
When applied to unconstrained finite-sum problems 
\begin{align}\label{eq:general-prob}
   \min_{\vx \in \R^d} f(\vx),\; \text{ where }\; f(\vx) := \frac{1}{n}\sum_{i = 1}^n f_i(\vx),
\end{align}
SGD performs the update $\vx_{t} = \vx_{t - 1} - \eta\nabla f_{i_t}(\vx_{t - 1})$ in each iteration~$t$. Traditional theoretical analysis for SGD builds upon the assumption of sampling $i_t \in [n]$ with replacement which leads to $\E_{i_t}[\nabla f_{i_t}(\vx_{t - 1})] = \nabla f(\vx_{t - 1})$ and thus much of the (deterministic) gradient descent-style analysis can be transferred to this setting. By contrast, no such connection between the component and the full gradient can be established for shuffled SGD --- which employs sampling \emph{without replacement} --- making its analysis much more challenging. As a result, despite its fundamental nature, there were no non-asymptotic convergence results for shuffled SGD until a very recent line of work~\citep{gurbuzbalaban2021random, shamir2016without, haochen2019random, nagaraj2019sgd, rajput2020closing, ahn2020sgd, mishchenko2020random, nguyen2021unified, cha2023tighter}. All these results consider general finite sum problems, with the same regularity condition constant (such as the Lipschitz constant of each $f_i$ or their gradients) being assumed to be the same for all the component functions. As a result, the obtained convergence bounds are typically no better than for (full) gradient descent, and are only better than the bounds for SGD with replacement sampling if the algorithm is run for many full passes over the data~\citep{mishchenko2020random, nguyen2021unified}.  

To study the effect of the structure of the data on the convergence of shuffled SGD, we sharpen the focus from general finite-sum problems to convex empirical risk minimization (ERM) with linear predictors: 
\begin{equation}\label{eq:prob}
    \min_{\vx \in \R^d}\Big\{f(\vx) := \frac{1}{n}\sum_{i = 1}^n\ell_i(\va_i^{\top}\vx)\Big\}, \tag{P}
\end{equation}
where $\va_i \in \R^d$ ($i \in [n]$) are data vectors and $\ell_i: \R \rightarrow \R$ are convex and either smooth or Lipschitz nonsmooth  functions associated with the linear predictors $\innp{\va_i, \vx}$ for $i \in [n]$.
In addition to their explicit dependence on the data, it is worth noting that problems of the form \eqref{eq:prob} cover most of the  standard convex ERM problems where shuffled SGD is commonly applied, such as   
support vector machines, least absolute deviation, least squares, and logistic regression. 

{Problem \eqref{eq:prob} admits an explicit primal-dual formulation using the standard Fenchel conjugacy argument (see, e.g.,~\citet{chambolle2011first,chambolle2018stochastic}),   
\begin{equation}\label{eq:PD-formulation}
\begin{aligned}
   & \min_{\vx \in \sR^d}\max_{\vy \in \sR^n} \gL(\vx, \vy),\\
    \gL(\vx, \vy) :=\; & \frac{1}{n}\innp{\mA\vx, \vy} - \frac{1}{n}\sum_{i = 1}^n\ell_i^*(\vy^{i}) \\
    = \;& \frac{1}{n}\sum_{i=1}^n \big(\va_i^{\top}\vx\vy^{i} - \ell_i^*(\vy^{i})\big), 
\end{aligned}
\tag{PD}
\end{equation}
where $\mA = [\va_1, \va_2, \dots, \va_n]^{\top} \in \R^{n \times d}$ is the data matrix and $\ell_i^*:\R\to\R$ is the convex conjugate of $\ell_i$. This observation allows us to interpret without-replacement SGD updates as cyclic coordinate updates on the dual side.

As already mentioned, our motivation for considering problems of this form, in addition to their prevalence in machine learning applications, is to transparently show how the structure of the data and the possibly different Lipschitz constants of the component functions or their gradients affect the complexity of the problem.
However, our techniques extend to the more general finite-sum problem~\eqref{eq:general-prob} considered in prior work, where the dependence on the component function smoothness constants is improved from the \emph{maximum} to the much tighter \emph{average} (see Appendix \ref{appx:general} for details). 
}

\paragraph{Background and Related Work.} 
SGD (with replacement) has been extensively studied in many settings (see e.g.,~\citet{robbins1951stochastic, bottou2018optimization, bubeck2015convex, agarwal2009information} for convex optimization). Compared to SGD, shuffled SGD usually exhibits faster convergence in practice~\citep{bottou2009curiously, recht2013parallel}, and is easier and more efficient to implement~\citep{bengio2012practical}. 
For each epoch $k$, shuffled SGD-style algorithms perform incremental gradient updates based on the sample ordering (permutation of the data points) denoted by $\pi^{(k)}$. There are three main choices of data permutations: (i) $\pi^{(k)} \equiv \pi$ for some fixed permutation of $[n]$ for all epochs, where shuffled SGD 
reduces to the incremental gradient (IG) method; (ii) $\pi^{(k)} \equiv \Tilde{\pi}$ where $\Tilde{\pi}$ is randomly chosen only once, at the beginning of the first epoch, referred to as the shuffle-once (SO) scheme; 
(iii) $\pi^{(k)}$ randomly generated at the beginning of each epoch, referred to as random reshuffling (RR). 

In terms of the theoretical analysis, sampling without replacement introduces the sampling bias at each iteration, making it difficult to approximate shuffled SGD by full gradient descent. Using empirical observations, shuffled SGD was conjectured to converge much faster than SGD with replacement, based on the \emph{noncommutative arithmetic-geometric mean
inequality} conjecture~\citep{recht2012toward}, which was later proved to be false~\citep{lai2020recht,de2020random}. As a consequence, whether or not shuffled SGD can be faster than SGD at least in some regimes remained open \citep{bottou2009curiously} until a breakthrough result in~\citet{gurbuzbalaban2021random}, where it was shown that for the class of smooth strongly convex optimization problems, the convergence of the RR variant of shuffled SGD is essentially of the order-$(1/K^2)$ for $K$ full passes of the data (also called epochs), which is faster than order-$(1/nK)$ convergence of SGD for sufficiently large $K$. This bound for the smooth strongly convex case was later improved under various regimes and additional assumptions~\citep{shamir2016without,haochen2019random,nagaraj2019sgd,mishchenko2020random, nguyen2021unified,ahn2020sgd}, while the tightest of those bounds were matched by lower bounds in \citet{rajput2020closing, yun2022minibatch, cha2023tighter,safran2020good}. 

Since our results are for the general (non-strongly) convex regimes, in the rest of this section we focus on the results that apply to those (convex, smooth or nonsmooth Lipschitz) regimes. For convex nonsmooth Lipschitz problems, we are only aware of the results in \citet{shamir2016without}. These results are only useful when the number of data passes $K$ is small and the number of component functions $n$ is large, as they contain an irreducible order-$\frac{1}{\sqrt{n}}$ error, and are not directly comparable to our results. 

For general smooth convex settings, the convergence of shuffled SGD has been established only recently. For $K$ sufficiently large, \citet{nagaraj2019sgd} proved a convergence rate $\gO(\frac{1}{\sqrt{nK}})$ for RR, which leads to the complexity matching  SGD. This result was later improved to $\gO(\frac{1}{n^{1/3}K^{2/3}})$ by \citet{mishchenko2020random, nguyen2021unified, cha2023tighter} for $K$ sufficiently large and with bounded variance assumed at the minimizer, while the same rate holds for SO~\citep{mishchenko2020random}. These results were complemented by matching lower bounds in \citet{cha2023tighter}, under sufficiently small step sizes as utilized in prior work.  
The results in~\citet{mishchenko2020random, nguyen2021unified} require restricted $\gO(\frac{1}{nL})$ step sizes and reduce to $\gO(\frac{1}{K})$ for small $K$, acquiring the same iteration complexity as full-gradient methods. Unlike in strongly convex settings, we are not aware of any follow-up work with improvements under small $K$ for smooth convex settings.

For IG variant of SGD without replacement (deterministic order), asymptotic convergence was established in~\citet{mangasarian1993serial, bertsekas2000gradient}, with further convergence results for both smooth and nonsmooth settings provided in~\citet{nedic2001incremental, li2019incremental, gurbuzbalaban2019convergence, ying2018stochastic, nguyen2021unified, mishchenko2020random}. As IG does not benefit from randomization, it is known to have a worse convergence bound  than RR under the Lipschitz Hessian assumption~\citep{gurbuzbalaban2019convergence, haochen2019random}, which was also shown in more general settings~\citep{mishchenko2020random}.

The major difficulty in analyzing shuffled SGD comes from characterizing the difference between the intermediate iterate and the iterate after one full data pass, for which current analysis (see e.g.,~\citet{mishchenko2020random} in smooth convex settings) uses the global smoothness constant with a triangle inequality. Such a bound may be too pessimistic and fail capturing the nuances of intermediate progress of shuffled SGD, which leads to a small step size and large $K$ restrictions. To provide a more fine-grained analysis that narrows the theory-practice gap for shuffled SGD, we notice that such proof difficulty is reminiscent of the analysis of cyclic block coordinate methods relating the partial gradients to the full one. This natural connection was further emphasized by studies of a variant of cyclic methods with random permutations~\citep{lee2019random, wright2020analyzing}, but their result was limited to convex quadratics. Since we consider ERM problems with linear predictors, the primal-dual reformulation~\eqref{eq:PD-formulation} allows for the primal-dual interpretation of shuffled SGD, where cyclic updates are performed on the dual side. 

In contrast to randomized methods (corresponding to standard SGD), cyclic methods are usually more challenging to analyze~\citep{nesterov2012efficiency}, basic variants exhibit much worse \emph{worst-case} complexity than even full gradient methods~\citep{li2017faster,sun2021worst, beck2013convergence, gurbuzbalaban2017cyclic, li2017faster, saha2013nonasymptotic, xu2015block,xu2017globally}, with more refined results being established only recently~\citep{song2021fast,cai2022cyclic,lin2023accelerated}. While the inspiration for our work came from these recent results~\citep{song2021fast,cai2022cyclic,lin2023accelerated}, they are completely technically disjoint. First, all these results rely on non-standard block Lipschitz assumptions, which are not present in our work. Second, all of them leverage proximal gradient-style cyclic updates to carry out the analysis, which is inapplicable in our case for the cyclic updates on the dual side, as otherwise the method would not correspond to (shuffled) SGD. Finally, \citet{song2021fast,lin2023accelerated} utilize extrapolation steps, which would 
break the connection to shuffled SGD in our setting, while \citet{cai2022cyclic} relies on a gradient descent-type descent lemma, which is 
impossible to establish in our setting. 

\paragraph{Contributions.} We provide a fine-grained analysis of  shuffled SGD through a primal-dual perspective, by studying ERM problems with linear predictors. For completeness, we further extend this perspective and the corresponding analysis to general smooth convex finite-sum problems in Appendix~\ref{appx:general}. We provide improved convergence bounds for all three popular data permutation strategies: RR, SO, and IG, in smooth convex and convex nonsmooth Lipschitz settings.  
We summarize our results and compare them to the state of the art in Table~\ref{table:complexity}. As is standard, all complexity results in Table~\ref{table:complexity} are expressed in terms of individual (component) gradient evaluations. They represent the number of gradient evaluations required to construct a solution with (expected) optimality gap $\epsilon,$ given a target error $\epsilon > 0.$

\begin{table*}[t!]
\caption{Comparison of our results with state of the art, in terms of individual gradient oracle complexity required to output $\vx_{\text{output}}$ with  $\E[f(\vx_{\text{output}}) - f(\vx_*)] \leq \epsilon$, where $\epsilon > 0$ is the target error and $\vx_*$ is the optimal solution. Here, $\sigma_*^2 = \frac{1}{n}\sum_{i = 1}^n\|\nabla f_i(\vx_*)\|_2^2$, $D = \|\vx_0 - \vx_*\|_2$. 
Parameters $\Hat{L},$ $\Tilde{L}$ and $\Bar{G}$ are defined in Section~\ref{sec:PD-shuffled-SGD} and discussed in the text of this section. {Parameters $\Hat{L}_g, \Tilde{L}_g$ are defined in Appendix~\ref{appx:general} and satisfy $\Hat{L}^g \leq \frac{1}{n}\sum_{i = 1}^n L_i$ and $\Tilde{L}^g \leq L$.} 
Quantity $\|\vy_*\|_{\mLambda^{-1}}$ is related to $\sigma_*$ and formally defined in Section~\ref{sec:prelim}.} 
\label{table:complexity}
\vskip 0.15in
\begin{center}
\begin{scriptsize}
\begin{sc}
\begin{tabular}{l c c c}
    \toprule
    PAPER & COMPLEXITY & ASSUMPTIONS & STEP SIZE \\
    \midrule
    \makecell[l]{\citet{nguyen2021unified} \\
    \citet{cha2023tighter}} \hfill (RR) & 
    $\gO\big(\frac{nLD^2}{\epsilon} + \frac{\sqrt{nL}\sigma_*D^2}{\epsilon^{3/2}}\big)$ & 
    $f_i$: $L$-smooth, convex, $b = 1$ & 
    $\gO\big(\frac{1}{nL}\big)$ \\
    \addlinespace
    \cline{2-4}
    \addlinespace
    \citet{mishchenko2020random} \hfill (RR/SO) & 
    $\gO\big(\frac{nLD^2}{\epsilon} + \frac{\sqrt{nL}\sigma_*D^2}{\epsilon^{3/2}}\big)$ & 
    $f_i$: $L$-smooth, convex, $b = 1$ & 
    $\gO\big(\frac{1}{nL}\big)$ \\
    \addlinespace
    \cline{2-4}
    \addlinespace
    {\bf [Ours, Theorem~\ref{thm:convergence}]} \hfill (RR/SO) & 
    $\gO\big(\frac{n\sqrt{\Hat{L}\Tilde{L}}D^2}{\epsilon} + \sqrt{\frac{(n - b)(n + b)}{n(n - 1)}}\frac{\sqrt{n\Tilde{L}}\sigma_*D^2}{\epsilon^{3/2}}\big)$ & 
    \makecell{$\ell_i$: $L_i$-smooth, convex \\ generalized linear model} & 
    $\gO\big(\frac{b}{n\sqrt{\Hat{L}\Tilde{L}}}\big)$ \\
    \addlinespace
    \cline{2-4}
    \addlinespace
    {\bf [Ours, Theorem~\ref{thm:convergence-general}]} \hfill (RR/SO) & 
    $\gO\big(\frac{n\sqrt{\Hat{L}^g\Tilde{L}^g}D^2}{\epsilon} + \sqrt{\frac{(n - b)(n + b)}{n(n - 1)}}\frac{\sqrt{n\Tilde{L}^g}\sigma_*D^2}{\epsilon^{3/2}}\big)$ & 
    \makecell{$f_i$: $L_i$-smooth, convex} 
    & 
    $\gO\big(\frac{b}{n\sqrt{\Hat{L}^g\Tilde{L}^g}}\big)$ \\
    \addlinespace
    \cline{2-4}
    \addlinespace
    \makecell[l]{\citet{cha2023tighter} \\ Lower bound} \hfill (RR) & $\Omega\big(\frac{\sqrt{nL}\sigma_*D^2}{\epsilon^{3/2}}\big)$ & 
    \makecell{$f_i$: $L$-smooth, convex, $b = 1$ \\ Large $K$} &
    $\gO\big(\frac{1}{nL}\big)$ \\
    \addlinespace
    \cline{1-4}
    \addlinespace
    \citet{mishchenko2020random} \hfill (IG) & 
    $\gO\big(\frac{nLD^2}{\epsilon} + \frac{n\sqrt{L}\sigma_*D^2}{\epsilon^{3/2}}\big)$ & 
    $f_i$: $L$-smooth, convex, $b = 1$ & 
    $\gO\big(\frac{1}{nL}\big)$ \\
    \addlinespace
    \cline{2-4}
    \addlinespace
    {\bf [Ours, Theorem~\ref{thm:convergence-IGD}]} \hfill (IG) & 
    $\begin{array}{c}
         \gO\big(\frac{n\sqrt{\Hat{L}\Tilde{L}}D^2}{\epsilon} \\
         + \frac{\min\{\sqrt{n\Hat{L}\Tilde{L}}\|\vy_{*}\|_{\mLambda^{-1}}, (n - b)\sqrt{\Tilde{L}}\sigma_*\}D^2}{\epsilon^{3/2}}\big)
    \end{array}$ & 
    \makecell{$\ell_i$: $L_i$-smooth, convex \\ generalized linear model} & 
    $\gO\big(\frac{b}{n\sqrt{\Hat{L}\Tilde{L}}}\big)$ \\
    \addlinespace
    \cline{2-4}
    \addlinespace
    {\bf [Ours, Theorem~\ref{thm:convergence-IGD-general}]} \hfill (IG) & 
    $\begin{array}{c}
         \gO\big(\frac{n\sqrt{\Hat{L}^g\Tilde{L}^g}D^2}{\epsilon} \\
         + \frac{\min\{\sqrt{n\Hat{L}^g\Tilde{L}^g}\|\vy_{*}\|_{\mLambda^{-1}}, (n - b)\sqrt{\Tilde{L}^g}\sigma_*\}D^2}{\epsilon^{3/2}}\big)
    \end{array}$ & 
    \makecell{$f_i$: $L_i$-smooth, convex}
    & 
    $\gO\big(\frac{b}{n\sqrt{\Hat{L}^g\Tilde{L}^g}}\big)$ \\
    \addlinespace
    \midrule
    \midrule
    \addlinespace
    \citet{shamir2016without} \hfill (RR/SO) & 
    $\gO\big(\frac{\Bar{B}^2G^2}{\epsilon^2})$ ($K = 1$, $n = \Omega(1/\epsilon^2)$\big) & 
    \makecell{$\ell_i$: $G$-Lipschitz, convex, $b = 1$ \\ $\Bar{B}$-Bounded Iterate, $\|\va_i\|_2 \leq 1$\\ generalized linear model} & 
    $\gO\big(\frac{1}{\sqrt{n}}\big)$ \\
    \addlinespace
    \cline{2-4}
    \addlinespace
    {\bf [Ours, Theorem~\ref{thm:convergence-nonsmooth}]} \hfill (RR/SO/IG) & 
    $\gO\big(\frac{n\Bar{G}D^2}{\epsilon^2}\big)$ & 
    \makecell{$\ell_i$: $G_i$-Lipschitz, convex \\ generalized linear model} & 
    $\gO\big(\frac{b}{n\sqrt{\Bar{G}K}}\big)$ \\
    \addlinespace
    \bottomrule
\end{tabular}
\end{sc}
\end{scriptsize}
\end{center}
\vskip -0.1in
\end{table*}

To succinctly explain where our improvements come from, we now consider 
{\eqref{eq:prob}} where $\ell_i$ is $1$-smooth and $b = 1$, ignoring the gains from the mini-batch estimators (for large $K$) and our softer guarantee that handles individual smoothness constants. For this specific case, $\Tilde{L} = L = \max_{1\leq i\leq n}\|\va_i\|^2,$ and thus our results for the smooth case and the RR and SO variants match state of the art in the second term, which dominates when there are many ($K = \Omega(\frac{L^2D^2n}{\sigma_*^2})$)
epochs. When there are $K = O(\frac{L^2D^2n}{\sigma_*^2})$ epochs in the SO and RR variants or for all regimes of $K$ in the IG variant, the difference between our and state of the art bounds comes from the constant $\Hat{L}$ that replaces $L$, and our improvement is by a factor $\sqrt{L/\Hat{L}}$. Note that $\gO(\frac{nL}{\epsilon})$ from prior bounds, which is the dominating term in the small $K$ regime, is even worse than the complexity of full gradient descent, as the full gradient Lipschitz constant of $f$ is $\frac{1}{n}\|\mA\mA^\top\|_2 \leq L$.

Denoting by $\mI_{j\uparrow}$ the identity matrix with its first $j$ diagonal elements set to zero, given a worst-case permutation $\bar{\pi}$, and denoting by $\mA_{\bar{\pi}}$ the data matrix $\mA$ with its rows permuted according to $\Bar{\pi},$ our constant $\Hat{L}$ can be bounded above by $L$ using the following sequence of inequalities:
\begin{equation}\label{eq:seq-of-relaxations}
\begin{aligned}
\Hat{L} &= \textstyle\frac{1}{n^2}\|\sum_{j=1}^n\mI_{(j - 1)\uparrow} \mA_{\Bar{\pi}}\mA_{\Bar{\pi}}^{\top}\mI_{(j - 1)\uparrow}\|_2\\
&\stackrel{\tiny(i)}{\leq} \textstyle\frac{1}{n^2}\textstyle\sum_{j=1}^n\|\mI_{(j - 1)\uparrow} \mA_{\Bar{\pi}}\mA_{\Bar{\pi}}^{\top}\mI_{(j - 1)\uparrow}\|_2\\
&\stackrel{\tiny(ii)}{\leq} \textstyle\frac{1}{n^2}\textstyle\sum_{j=1}^n\| \mA_{\Bar{\pi}}\mA_{\Bar{\pi}}^{\top}\|_2\\
&\stackrel{\tiny(iii)}{\leq} \textstyle\frac{1}{n}\sum_{i=1}^n \|\va_i\|^2 \leq \max_{1\leq i \leq n}\|\va_i\|^2 = L, 
\end{aligned}
\end{equation}
where $(i)$ holds by the triangle inequality, $(ii)$ holds because the operator norm of the matrix $\mI_{(j - 1)\uparrow} \mA_\pi\mA_\pi^{\top}\mI_{(j - 1)\uparrow}$ (equal to the operator norm of the bottom right $(n- j + 1)\times (n - j + 1)$ submatrix of $\mA_\pi\mA_\pi^{\top}$) is always at most $\|\mA_\pi\mA_\pi^{\top}\| = \|\mA\mA^{\top}\|$, for any permutation $\pi$, and $(iii)$ holds by bounding above the operator norm of a symmetric matrix by its trace. Hence $\Hat{L}$ is never larger than $L$, but can generally be much smaller, due to the sequence of inequality relaxations in \eqref{eq:seq-of-relaxations}. While each of these inequalities can be loose, we emphasize that $(iii)$ is almost always loose, by a factor that can be as large as $n$. 
Even more, we empirically evaluate ${L/\Hat{L}}$ 
on 15 large-scale machine learning datasets and demonstrate that on those datasets it is of the order $n^{\alpha},$ for $\alpha \in [0.15, 0.96]$ (see Sec.~\ref{sec:exp} for more details), providing strong evidence of a tighter guarantee as a function of $n$. 

For the nonsmooth settings, by a similar sequence of inequalities, we can show that $\Bar{G} \leq G^2,$ which can be loose by a factor $1/n$ due to the operator norm to trace inequality.
Thus, our bound is never worse than what would be obtained from the full subgradient method, but can match
the bound of standard SGD, or can even improve\footnote{This is because it is possible for inequalities $(i)$ and $(ii)$ to be loose, in addition to $(iii)$.} upon it for at least some data matrices $\mA.$

\section{Notation and Preliminaries}\label{sec:prelim}
We consider a real $d$-dimensional Euclidean space $(\R^d, \|\cdot\|)$, where $d$ is finite and the norm $\|\cdot\| = \sqrt{\innp{\cdot, \cdot}}$ is induced by the standard inner product associated with the space. By our primal-dual formulation~(\ref{eq:PD-formulation}), the corresponding dual space is $(\R^n, \|\cdot\|)$. For a vector $\vx$, we let $\vx^j$ denote its $j$-th coordinate. For any positive integer $m$, we use $[m]$ to denote the set $\{1, 2, \dots, m\}$. We assume that we are given a positive integer $m \leq n$ and a partition of $[n]$ into sets $\gS^1, \gS^2, \dots, \gS^m$. For notational convenience, we assume that the partition is ordered, in the sense that for $1 \leq j < j' \leq m$, $\max_{i \in \gS^j} i < \min_{i' \in \gS^{j'}}i'$.\footnote{This is without loss of generality, as it can be achieved by reordering the rows in the data matrix.} We denote by $\vy^{(j)}$ the subvector of $\vy \in \R^n$ indexed by the elements of  $\gS^j$, and by $\mA^{(j)}$ the submatrix obtained from $\mA \in \R^{n \times d}$ by selecting the rows indexed by $\gS^j$. 

Given a matrix $\mA$,  $\|\mA\| := \sup_{\vx \in \R^d, \|\vx\| \leq 1}\|\mA\vx\|$ denotes its operator norm. For a positive definite matrix $\mLambda$, $\|\cdot\|_{\mLambda}$ denotes the Mahalanobis norm, $\|\vx\|_{\mLambda} := \sqrt{\innp{\mLambda \vx, \vx}}$. We use $\mI$ to denote the identity matrix of size $n \times n$ throughout the paper. For any $j \in [n]$, we define $\mI_{j\uparrow}$ as the matrix obtained from the identity matrix $\mI$ by setting the first $j$ diagonal elements to zero, and we let $\mI_j$ be the matrix with only $j$-th diagonal element nonzero and equal to $1$. Furthermore, for the $j$-th block we define $\mI_{(j)} := \sum_{i = b(j - 1) + 1}^{bj}\mI_i$ where only the elements of $\mI$ corresponding to the $j$-th block remains non-zero.

For each loss $\ell_i(\cdot)$, we denote its convex conjugate by $\ell_i^*(\vy^i): \R \rightarrow \R$ and thus  $\ell_i(\va_i^{\top}\vx) = \sup_{\vy^i \in \R}\{\vy^i\va_i^{\top}\vx - \ell_i^*(\vy^i)\}$. We let $\vy_{\vx} \in \R^n$ be the conjugate pair of $\vx \in \R^d$, i.e., $\vy_\vx^i = \argmax_{\vy^i \in \R}\{\vy^i\va_i^{\top}\vx - \ell_i^*(\vy^i)\} \in \partial \ell_i(\va_i^{\top}\vx)$. 

We make the following assumptions. The first one is the minimal  one made throughout the paper.
\begin{assumption}\label{assp:convex}
Each individual loss $\ell_i$ is convex, and there exists a minimizer $\vx_* \in \R^d$ for $f(\vx)$.
\end{assumption}
By Assumption~\ref{assp:convex}, $f$ is convex, and we denote $\vy_* = \vy_{\vx_*}$. 
In nonsmooth settings, we make an additional standard assumption.
\begin{assumption}\label{assp:Lipschitz}
Each $\ell_i$ is $G_i$-Lipschitz $(i \in [n])$, i.e., $|\ell_i(x) - \ell_i(y)| \leq G_i|x - y|$ for any $x, y \in \R$; thus $|g_i(x)| \leq G_i$ where $g_i(x) \in \partial \ell_i(x)$. 
\end{assumption}
If Assumption~\ref{assp:Lipschitz} holds, each $\ell_i(\va_i^\top\vx)$ is also $G$-Lipschitz w.r.t.\ $\vx$, where $G = \max_{i \in [n]}G_i\|\va_i\|_2$. We let $\mGamma = \text{diag}(G_1^2, G_2^2, \dots, G_n^2)$ be the diagonal matrix composed of the squared component Lipschitz constants $G_i^2$. 
For the smooth settings, we make the following assumption.
\begin{assumption}\label{assp:smooth}
Each $\ell_i$ is $L_i$-smooth $(i \in [n])$, i.e., $|\ell_i'(x) - \ell_i'(y)| \leq L_i|x - y|$ for any $x, y \in \R$. 
\end{assumption}
We remark that Assumption~\ref{assp:smooth} implies that both $f$ and each component function $f_i(\vx) = \ell_i(\va_i^\top\vx)$ are $L$-smooth, 
where $L = \max_{i \in [n]}L_i\|\va_i\|^2_2$. Assumption~\ref{assp:smooth} also implies that each convex conjugate $\ell_i^*$ is $\frac{1}{L_i}$-strongly convex~\citep{beck2017first}. In the following, we let $\mLambda = \text{diag}(L_1, L_2, \dots, L_n)$. 
When $\ell_i$ are smooth, we assume bounded variance at $\vx_*$, same as prior work~\citep{mishchenko2020random, nguyen2021unified,tran2021smg, tran2022nesterov}.

\begin{assumption}\label{assp:bounded-vr}
{$\sigma_*^2 := \frac{1}{n}\sum_{i = 1}^n\|\nabla f_i(\vx_*)\|^2 = \frac{1}{n}\sum_{i = 1}^n(\ell_i'(\va_i^\top\vx_*))^2\|\va_i\|_2^2$} is bounded. 
\end{assumption}

Given a primal-dual pair $(\vx, \vy) \in \R^d \times \R^n$, the primal-dual gap of \eqref{eq:PD-formulation} is $\text{Gap}(\vx, \vy) = \max_{(\vu, \vv) \in \R^d \times \R^n} \{\gL(\vx, \vv) - \gL(\vu, \vy)\}$. In particular, we consider the pair $(\vx, \vy_*)$ for  $\vx \in \R^d$, and bound $\text{Gap}^\vv(\vx, \vy_*) = \gL(\vx, \vv) - \gL(\vx_*, \vy_*)$ for an arbitrary but fixed $\vv$. To finally obtain the function value gap $f(\vx) - f(\vx_*)$ for (\ref{eq:prob}), we only need to choose $\vv = \argmax_{\vw \in \R^d}\gL(\vx, \vw) = \vy_{\vx}$.

To handle the cases with random data permutations, we use the following definitions corresponding to the data permutation $\pi = \{\pi_1, \pi_2, \dots, \pi_n\}$ of $[n]$: $\mA_\pi := \big[\va_{\pi_1}, \va_{\pi_2}, \dots, \va_{\pi_n}\big]^{\top}$, $\vv_\pi = \big(\vv^{\pi_1}, \vv^{\pi_2}, \dots, \vv^{\pi_n}\big)^\top$, $\vy_{*, \pi} = \big(\vy^{\pi_1}_*, \vy^{\pi_2}_*, \dots, \vy^{\pi_n}_*\big)^\top$, $\mGamma_\pi = \text{diag}\big(G_{\pi_1}^2, G_{\pi_2}^2, \dots, G_{\pi_n}^2\big)$ and $\mLambda_\pi = \text{diag}\big(L_{\pi_1}, L_{\pi_2}, \dots, L_{\pi_n}\big)$. For the permutation $\pi^{(k)}$ at the $k$-th epoch, we denote $\mA_k = \mA_{\pi^{(k)}}$, $\vv_{k} = \vv_{\pi^{(k)}}$, $\vy_{*, k} = \vy_{*, \pi^{(k)}}$, $\mGamma_k = \mGamma_{\pi^{(k)}}$ and $\mLambda_k = \mLambda_{\pi^{(k)}}$, for brevity.

\section{Primal-Dual View of Shuffled SGD and Improved Bounds}\label{sec:PD-shuffled-SGD}

\begin{algorithm}[t!]
\caption{Shuffled SGD (Primal-Dual View)}\label{alg:PD-shuffled-sgd}
\begin{algorithmic}[1]
    \STATE \textbf{Input:} Initial point $\vx_0 \in \R^d,$ batch size $b > 0,$ step size $\{\eta_k\} > 0,$ number of epochs $K > 0$
    \FOR {$k = 1$ to $K$}
        \STATE Generate any permutation $\pi^{(k)}$ of $[n]$ (either deterministic or random)
        \STATE $\vx_{k - 1, 1} = \vx_{k - 1}$
        \FOR {$i = 1$ to $m$}
            \STATE $\vy_k^{(i)} = \argmax_{\vy \in \R^b} \big\{\vy^{\top}\mA_k^{(i)}\vx_{k - 1, i} - \sum_{j = 1}^b\ell_{\pi_{b(i - 1) + j}^{(k)}}^*(\vy^j)\big\} \hfill \label{algo-line:dual-update}$ 
            \STATE $\vx_{k - 1, i + 1} =  \argmax_{\vx \in \R^d}\big\{\vy_k^{(i)\top}\mA_k^{(i)}\vx + \frac{b}{2\eta_k}\|\vx - \vx_{k - 1, i}\|^2\big\} \hfill \label{algo-line:primal-update}$
        \ENDFOR
        \STATE $\vx_{k} = \vx_{k - 1, m + 1}$, $\vy_k = \big(\vy_k^{(1)}, \vy_k^{(2)}, \dots, \vy_k^{(m)}\big)^{\top}$\;
    \ENDFOR
    \STATE \textbf{Return:} $\Hat{\vx}_K = \sum_{k=1}^K \eta_k \vx_{k} / \sum_{k = 1}^K \eta_k$
\end{algorithmic}
\end{algorithm}

We consider shuffled SGD with mini-batch estimators of size $b$ in convex settings. Without loss of generality, we assume $n = bm$ for some positive integer $m$, so shuffled SGD performs $m$ incremental gradient updates within each epoch (data pass).
Based on the formulation (\ref{eq:PD-formulation}), we view shuffled SGD as a primal-dual method with block coordinate updates on the dual side, as summarized in Algorithm~\ref{alg:PD-shuffled-sgd}, for completeness. To see the equivalence, in $i$-th inner iteration of $k$-th epoch, we first update the $i$-th block $\vy_k^{(i)} \in \R^b$ of the dual vector $\vy_{k - 1} \in \R^n$ based on $\vx_{k - 1, i}$ as in Line 6.
Since the dual update has a decomposable structure, this maximization step corresponds to computing the (sub)gradients $\{\ell_{\pi^{(k)}_j}'(\va_{\pi^{(k)}_j}^{\top}\vx_{k - 1, i})\}_{j = b(i - 1) + 1}^{bi}$ at $\vx_{k - 1, i}$ for the batch of individual losses indexed by $\{\pi^{(k)}_j\}_{j = b(i - 1) + 1}^{bi}$. Then in Line 7, 
we perform a minimization step using $\vy_k^{(i)}$ to compute $\vx_{k - 1, i + 1}$ on the primal side. Combining these two steps, we have $\vx_{k - 1, i + 1} = \vx_{k - 1, i} - \frac{\eta_k}{b}\sum_{j = b(i - 1) + 1}^{bi}\ell_{\pi^{(k)}_j}'(\va_{\pi^{(k)}_j}^{\top}\vx_{k - 1, i})\va_{\pi^{(k)}_j}$, which is exactly the {\it original primal shuffled SGD updating scheme}. 

\subsection{Smooth Convex Minimization}\label{sec:smooth}
{We begin by considering the case of smooth convex optimization for~\eqref{eq:prob}, 
while the analysis for general finite-sum problems~\eqref{eq:general-prob} is deferred to Appendix~\ref{appx:general}. 
To analyze the convergence of Algorithm~\ref{alg:PD-shuffled-sgd}, we bound its primal-dual gap $\mathrm{Gap}^\vv(\vx_{k}, \vy_*)$.
In particular, the dual update yields an upper bound on $\gL(\vx_{k}, \vv)$ by the strong convexity of $\ell_i^*$, while the primal update provides a lower bound on $\gL(\vx_*, \vy_*)$ via the minimization step. 
Combining these two, we can bound
$\text{Gap}^\vv(\vx_{k}, \vy_*)$ as summarized in Lemma~\ref{lemma:gap-bound}. 
The proof is deferred to Appendix~\ref{appx:smooth} for space considerations.}
\begin{restatable}{lemma}{gapBound}
\label{lemma:gap-bound}
Given $\{\vy_k^{(i)}\}_{i = 1}^m$ and $\{\vx_{k - 1, i}\}_{i = 1}^{m + 1}$ generated by Algorithm~\ref{alg:PD-shuffled-sgd} for  $k \in [K]$, let $\gE_k := \eta_k \mathrm{Gap}^\vv(\vx_k, \vy_*)
+ \frac{b}{2n}\|\vx_* - \vx_{k}\|_2^2 - \frac{b}{2n} \|\vx_* - \vx_{k - 1}\|_2^2$. If Assumptions~\ref{assp:convex}~and~\ref{assp:smooth} hold, then 
\begin{equation}\label{eq:gap-bound}
\begin{aligned}
    \gE_k \leq \;& \frac{\eta_k}{n}\sum_{i = 1}^m\vy_k^{(i) \top}\mA_k^{(i)}(\vx_{k} - \vx_{k - 1, i + 1}) \\
    & + \frac{\eta_k}{n}\sum_{i = 1}^m \big(\vv^{(i)}_k - \vy_k^{(i)}\big)^{\top}\mA_k^{(i)}(\vx_{k} - \vx_{k - 1, i}) \\
    & - \frac{\eta_k}{2n}\|\vy_k - \vv_k\|_{\mLambda_k^{-1}}^2 - \frac{\eta_k}{2n}\|\vy_k - \vy_{*, k}\|_{\mLambda_k^{-1}}^2 \\
    & - \frac{b}{2n}\sum_{i = 1}^m\|\vx_{k - 1, i} - \vx_{k - 1, i + 1}\|^2_2, 
\end{aligned}
\end{equation}
\end{restatable}
The first term $\gT_{1} := \frac{\eta_k}{n}\sum_{i = 1}^m\vy_k^{(i) \top}\mA_k^{(i)}(\vx_{k} - \vx_{k - 1, i + 1})$ in~\eqref{eq:gap-bound} may seem nontrivial to handle at a first glance, but can actually be aggregated into a retraction term capturing the primal progress within one epoch. We state the bound on $\gT_1$ in Lemma~\ref{lemma:first-inner}, with the proof provided in Appendix~\ref{appx:smooth}.
\begin{restatable}{lemma}{firstInner}
\label{lemma:first-inner}
For any $k \in [K]$, the iterates $\{\vy_k^{(i)}\}_{i = 1}^m$ and $\{\vx_{k - 1, i}\}_{i = 1}^{m + 1}$ in  Algorithm~\ref{alg:PD-shuffled-sgd} satisfy 
\begin{equation*}
    \gT_{1} = \frac{b}{2n}\sum_{i = 1}^m\|\vx_{k - 1, i} - \vx_{k - 1, i + 1}\|^2_2 - \frac{b}{2n}\|\vx_{k - 1} - \vx_{k}\|^2_2. 
\end{equation*}
\end{restatable}
The second term $\gT_{2} := \frac{\eta_k}{n}\sum_{i = 1}^m \big(\vv^{(i)}_k - \vy_k^{(i)}\big)^{\top}\mA_k^{(i)}(\vx_{k} - \vx_{k - 1, i})$ in~\eqref{eq:gap-bound} requires us to relate the intermediate iterate $\vx_{k - 1, i}$ to the iterate $\vx_{k}$ after one full data pass, which corresponds to a partial sum of the component gradients, each at different iterates $\{\vx_{k - 1, j}\}_{j = i}^{m}$. In contrast to prior analyses (e.g.,~\citet{mishchenko2020random}) using the global smoothness and triangle inequality to bound this partial sum, we provide a tighter bound on $\gT_2$ that tracks the progress of the cyclic update on the dual side, in the aggregate. 
To simplify the notation, we introduce the following definitions 
\begin{equation}\label{eq:new-smoothness-constants}
\begin{aligned}
    & \Hat{L}_\pi := \frac{1}{mn}\big\|\mLambda^{1/2}_\pi\big(\textstyle\sum_{j=1}^m\mI_{b(j - 1)\uparrow} \mA_\pi\mA_\pi^{\top}\mI_{b(j - 1)\uparrow}\big)\mLambda_\pi^{1/2}\big\|_2, \\
    & \Tilde{L}_\pi := \frac{1}{b}\big\|\mLambda_\pi^{1/2}\big(\textstyle\sum_{j=1}^m\mI_{(j)} \mA_\pi\mA_\pi^{\top}\mI_{(j)}\big)\mLambda_\pi^{1/2}\big\|_2, \\
    & \Hat{L} = \max_{\pi}\Hat{L}_{\pi}, \quad \Tilde{L} = \max_{\pi}\Tilde{L}_{\pi}.
\end{aligned}
\end{equation}
Each of these constants is never larger than $L$, but can possibly be much smaller, by a similar argument as in Section~\ref{sec:intro}. For the first quantity $\Hat{L}_\pi$, we have the following sequence of relaxations compared to $L$:
$
\Hat{L}_\pi
\leq \frac{1}{mn}\sum_{j = 1}^m\|\mI_{b(j - 1)\uparrow} \mLambda^{1/2}_\pi\mA_\pi(\mLambda^{1/2}_\pi\mA_\pi)^{\top}\mI_{b(j - 1)\uparrow}\|_2
\leq \frac{1}{n}\|\mLambda^{1/2}\mA(\mLambda^{1/2}\mA)^{\top}\|_2
\leq \frac{1}{n}\sum_{i = 1}^n L_i\|\va_i\|_2^2
\leq L,
$
while the gap can be as large as order-$n$, as discussed in the introduction.  
For the second quantity $\Tilde{L}_\pi$ in~\eqref{eq:new-smoothness-constants}, we notice that $\textstyle\sum_{j=1}^m\mI_{(j)} \mLambda^{1/2}_\pi\mA_\pi(\mLambda^{1/2}_\pi\mA_\pi)^{\top}\mI_{(j)}$ is a block-diagonal matrix, for which the operator norm is the maximum of the operator norms over the block submatrices, thus $\Tilde{L}_\pi = \frac{1}{b}\max_{j \in [m]}\|\mI_{(j)} \mLambda^{1/2}_\pi\mA_\pi(\mLambda^{1/2}_\pi\mA_\pi)^{\top}\mI_{(j)}\|_2 \leq L$, where the relaxation mainly comes from the operator norm to trace inequality. For $b = 1$, the inequality is tight, as $\textstyle\sum_{j=1}^n\mI_{(j)} \mLambda^{1/2}_\pi\mA_\pi(\mLambda^{1/2}_\pi\mA_\pi)^{\top}\mI_{(j)}$ reduces to a diagonal matrix with the elements $L_{\pi_i}\|\va_{\pi_i}\|_2^2$. 
We provide further numerical results for the gap between $\Tilde{L}_\pi$ and $L$ w.r.t.\ different batch sizes $b$ in Appendix~\ref{appx:exp}.

Permutation-dependent quantities  $\hat{L}_\pi$ and $\Tilde{L}_\pi$ defined in \eqref{eq:new-smoothness-constants} are obtained directly from our analysis. 
Unlike prior work and to arrive at these quantities, we avoid resorting to global smoothness and aggregate information over a full data pass. 
Assuming that a uniformly random data shuffling strategy is used (SO or RR), the resulting bound on $\gT_2$ is summarized in Lemma~\ref{lemma:second-inner}, and we defer its proof to Appendix~\ref{appx:smooth}. 
\begin{restatable}{lemma}{secondInner}
\label{lemma:second-inner}
Under Assumption~\ref{assp:bounded-vr}, for any $k \in [K]$, the iterates $\{\vy_k^{(i)}\}_{i = 1}^m$ and $\{\vx_{k - 1, i}\}_{i = 1}^{m + 1}$ generated by Algorithm~\ref{alg:PD-shuffled-sgd} with uniformly random shuffling satisfy 
\begin{equation*}
\begin{aligned}
    \E[\gT_{2}] \leq \E\Big[\frac{\eta_k^3 n\Hat{L}_{\pi^{(k)}}\Tilde{L}_{\pi^{(k)}}}{b^2}\|\vy_k - \vy_{*, k}\|^2_{\mLambda_k^{-1}} + \frac{\eta_k}{2n}\|\vv_k - \vy_{k}\|_{\mLambda_k^{-1}}^2\Big] + \frac{\eta_k^3\Tilde{L}(n - b)(n + b)}{6b^2(n - 1)}\sigma_*^2. 
\end{aligned}
\end{equation*}
\end{restatable}
Combining Lemmas~\ref{lemma:gap-bound}-\ref{lemma:second-inner}, we arrive at Theorem~\ref{thm:convergence}, whose proof is deferred to Appendix~\ref{appx:smooth}.
\begin{restatable}{theorem}{convergence}
\label{thm:convergence}
Under Assumptions~\ref{assp:convex},~\ref{assp:smooth}~and~\ref{assp:bounded-vr}, if $\eta_k \leq \frac{b}{n\sqrt{2\Hat{L}_{\pi^{(k)}}\Tilde{L}_{\pi^{(k)}}}}$ 
and $H_K = \sum_{k = 1}^K\eta_k$, the output $\vxh_K$ of 
Alg.~\ref{alg:PD-shuffled-sgd} with uniformly random (SO/RR) shuffling satisfies 
\begin{align*}
\E[H_K(f(\Hat{\vx}_K) - f(\vx_*))] \leq \frac{b}{2n}\|\vx_0 - \vx_*\|_2^2 + \sum_{k=1}^K\frac{\eta_k^3\Tilde{L}(n - b)(n + b)}{6b^2(n - 1)}\sigma_*^2.
\end{align*}
Hence, given $\epsilon > 0,$ there exists a constant step size $\eta_k = \eta$ such that  $\E[f(\Hat{\vx}_K) - f(\vx_*)] \leq \epsilon$ after $\gO\big(\frac{n\sqrt{\Hat{L}\Tilde{L}}\|\vx_0 - \vx_*\|_2^2}{\epsilon} + \sqrt{\frac{(n - b)(n + b)}{n(n - 1)}}\frac{\sqrt{n\Tilde{L}}\sigma_*\|\vx_0 - \vx_*\|_2^2}{\epsilon^{3/2}}\big)$ gradient queries.
\end{restatable}
A few remarks are in order here. First, the worst-case value $\hat{L}$ of $\hat{L}_\pi$ appears in our final complexity bound because of the need to select a constant, deterministic step size, as is commonly done in practice by tuning. However, we observe in our numerical experiments (see Appendix~\ref{appx:exp}) that the random values $\hat{L}_{\pi}$ are highly concentrated. Thus, in principle, $\hat{L}$ can be replaced by a ``high probability'' upper bound $\hat{L}_{\rm high prob}$ on $\hat{L}_{\pi}$ over uniformly random permutations, by conditioning on the event that $\hat{L}_k \leq \hat{L}_{\rm high prob}.$ This leads to a potentially weaker probabilistic guarantee, but one that is nevertheless better aligned with empirical experience. 

We now discuss some special cases of Theorem~\ref{thm:convergence}. When $b = n$, we recover the standard guarantee of gradient descent, which serves as a sanity check as in this case the algorithm reduces to standard gradient descent. When $\eps = \Omega(\frac{(n - b)(n + b)\sigma_*^2}{n^2(n - 1)\Hat{L}}),$ the resulting complexity is $\gO\big(\frac{n\sqrt{\Hat{L}\Tilde{L}}\|\vx_0 - \vx_*\|_2^2}{\epsilon}\big).$ Observe that this case can happen when either $\epsilon$ is large (compared to, say, $1/{n}$) or when $\sigma_*$ is small (it is, in fact, possible for $\sigma_*$ to be zero, which happens, for example, when the data rows are linearly independent). Unlike in bounds from previous work, we observe from our bounds the benefit of using shuffled SGD compared to full gradient descent, where the difference is by a factor  that can be as large as $\sqrt{n},$ as we have discussed in the introduction. When $\epsilon = \gO(\frac{(n - b)(n + b)\sigma_*^2}{n^2(n - 1)\Hat{L}}),$ the second term in our complexity bound dominates. In this case, when $b=1$, we recover the state of the art results from \citet{mishchenko2020random, nguyen2021unified,cha2023tighter}, while for $b > 1$ our bound provides the $\Omega\big(\sqrt{\frac{n(n - 1)}{(n - b)(n + b)}\cdot \frac{L}{\Tilde{L}}}\big)$-factor improvement, providing insights into benefits from  the mini-batching strategy commonly used in practice.    

\paragraph{Incremental gradient descent.}When using the deterministic data ordering without any random shuffling, shuffled SGD reduces to incremental gradient descent. We define $\Hat{L}_0 = \hat{L}_{\pi^{(0)}}$ and $\Tilde{L}_0 = \Tilde{L}_{\pi^{(0)}}$ w.r.t.\ the initial, fixed permutation $\pi^{(0)}$ of the data matrix $\mA$, and summarize the results in Theorem~\ref{thm:convergence-IGD}. The proof is provided in Appendix~\ref{appx:smooth}. 
While 
the provided complexity of IG can be order-$n^{1/2}$ worse than the complexity of RR/SO for small target error $\epsilon$, it aligns with prior results where a similar $\gO(n^{1/2})$ loss was incurred when removing random shuffling~\citep{mishchenko2020random}. 
\begin{restatable}{theorem}{convergenceIGD}
\label{thm:convergence-IGD}
Under Assumptions~\ref{assp:convex},~\ref{assp:smooth}~and~\ref{assp:bounded-vr}, if $\eta_k \leq \frac{b}{n\sqrt{2\Hat{L}_0\Tilde{L}_0}}$ and $H_K = \sum_{k = 1}^K\eta_k$, the output $\vxh_K$ of  Alg.~\ref{alg:PD-shuffled-sgd} with a fixed permutation satisfies 
\begin{align*}
    H_K\big(f(\Hat{\vx}_K) - f(\vx_*)\big) \leq \frac{b}{2n}\|\vx_0 - \vx_*\|_2^2 + \sum_{k=1}^K \min\Big\{\frac{\eta_k^3 n}{b^2}\Hat{L}_0\Tilde{L}_0\|\vy_{*}\|^2_{\mLambda^{-1}}, \frac{\eta_k^3 (n - b)^2}{b^2}\Tilde{L}_0\sigma_*^2\Big\}.
\end{align*}
{As a consequence, given $\epsilon > 0,$ there exists a constant step size $\eta_k = \eta$ such that $f(\Hat{\vx}_K) - f(\vx_*) \leq \epsilon$ after the number of gradient queries bounded by  $\gO\Big(\frac{n\sqrt{\Hat{L}_0\Tilde{L}_0}\|\vx_0 - \vx_*\|_2^2}{\epsilon} + \frac{\min\big\{\sqrt{n\Hat{L}_0\Tilde{L}_0}\|\vy_{*}\|_{\mLambda^{-1}},\, (n - b)\sqrt{\Tilde{L}_0}\sigma_*\big\}\|\vx_0 - \vx_*\|_2^2}{\epsilon^{3/2}}\Big)$.}
\end{restatable}

\subsection{Nonsmooth Minimization}\label{sec:nonsmooth}
We now extend our analysis of Algorithm~\ref{alg:PD-shuffled-sgd} to convex nonsmooth Lipschitz settings, 
where the conjugate functions $\ell_i^*(y^i)$ are only convex. Proceeding as in  Lemma~\ref{lemma:gap-bound}, we obtain a bound on the primal-dual gap similar to~\eqref{eq:gap-bound}, but lose two retraction terms induced by smoothness.
Instead of cancelling the corresponding error terms as in the smooth case, we rely on the boundedness of the subgradients to bound these terms under a sufficiently small step size, as is common in nonsmooth Lipschitz settings. 
As in Section~\ref{sec:smooth}, we introduce the following quantities to obtain a tighter guarantee w.r.t.\ the data matrix and Lipschitz constants 
\begin{align}
    \Hat{G}_\pi := \;& \frac{1}{mn}\big\|\mGamma^{1/2}_\pi\big(\textstyle\sum_{j=1}^m\mI_{b(j - 1)\uparrow} \mA_\pi\mA_\pi^{\top}\mI_{b(j - 1)\uparrow}\big)\mGamma_\pi^{1/2}\big\|_2, \notag\\
    \Tilde{G}_\pi := \;& \frac{1}{b}\big\|\mGamma_\pi^{1/2}\big(\textstyle \sum_{j=1}^m\mI_{(j)} \mA_\pi\mA_\pi^{\top}\mI_{(j)}\big)\mGamma_{\pi}^{1/2}\big\|_2. \label{eq:new-Lip-condition}
\end{align}
Theorem~\ref{thm:convergence-nonsmooth} describes the convergence of Algorithm~\ref{alg:PD-shuffled-sgd}, and the proof is deferred to Appendix~\ref{appx:nonsmooth}.
\begin{restatable}{theorem}{convergenceNonsmooth}
\label{thm:convergence-nonsmooth}
Under Assumptions~\ref{assp:convex}~and~\ref{assp:Lipschitz}, if $H_K = \sum_{k = 1}^K\eta_k$ and $\Bar{G} = \E_{\pi}[\sqrt{\Hat{G}_{\pi}\Tilde{G}_{\pi}}]$, the output $\vxh_K$ of Alg.~\ref{alg:PD-shuffled-sgd} with possible uniformly random shuffling satisfies 
\begin{equation}\notag
    \E[H_K(f(\Hat{\vx}_K) - f(\vx_*))] \leq \frac{b}{2n}\|\vx_0 - \vx_*\|_2^2 + \sum_{k = 1}^K\frac{2\eta_k^2 n \Bar{G}}{b}, 
\end{equation}
As a consequence, for any $\epsilon > 0,$ there exists a step size $\eta_k = \eta$ such that $\E[f(\Hat{\vx}_K) - f(\vx_*)] \leq \epsilon$ after $\gO\big(\frac{n\Bar{G}\|\vx_0 - \vx_*\|_2^2}{\epsilon^2}\big)$ gradient queries.
\end{restatable}
We now briefly discuss this result. The total number of individual gradient queries is $\gO\big(\frac{n\Bar{G}\|\vx_0 - \vx_*\|_2^2}{\epsilon^2}\big)$, which appears independent of the batch size, but this is actually not the case, as the parameter $\Bar{G} = \E_{\pi}[\sqrt{\Hat{G}_{\pi}\Tilde{G}_{\pi}}]$ depends on the block partitioning, due to \eqref{eq:new-Lip-condition}. When $b = n,$ as a sanity check, we recover the standard guarantee of (full) subgradient descent, which is expected, as in this case shuffled SGD reduces to subgradient descent. When $b = 1,$ however, the bound is worse than the corresponding bound for standard SGD, by a factor $\gO(n\Bar{G}/G^2)$. By a similar sequence of inequalities as in \eqref{eq:seq-of-relaxations}, this factor is never worse than $n,$ but it is typically much smaller, taking values as small as 1. We note that it is not known whether a better bound is possible for shuffled SGD in this setting, as the only seemingly tighter upper bound from 
\citet{shamir2016without} 
applies only for constant $K$, when $n = \Omega(\frac{1}{\epsilon^2})$, and under an additional boundedness assumption for the algorithm iterates. 

\section{Numerical Results and Discussion}\label{sec:exp}

\begin{figure}[t!]
    \centering
    \hspace{\fill}
    {\includegraphics[width=0.4\textwidth]{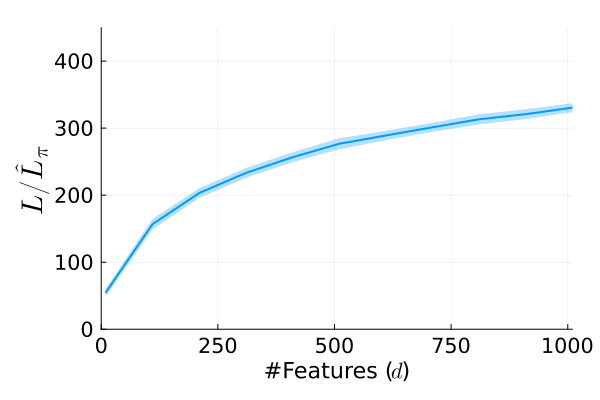}\label{fig:L-gaussian-compare-n}} \hspace{\fill}
    {\includegraphics[width=0.4\textwidth]{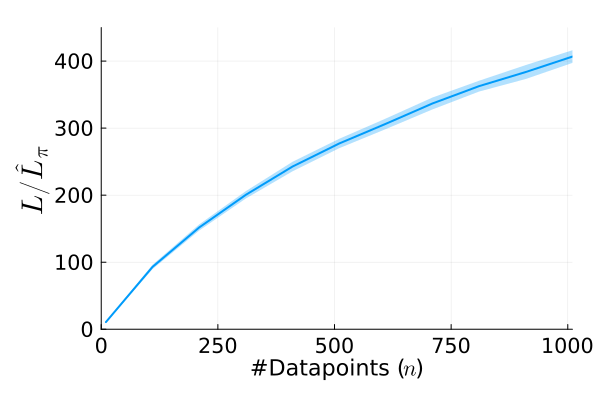}\label{fig:L-gaussian-compare-d}}\hspace*{\fill}
    \caption{Evaluation of $L / \Hat{L}_\pi$ on synthetic Gaussian datasets, where the solid lines represent the mean values and filled regions indicate the standard deviation of $100$ random permutations of $\pi$ in each randomly generated dataset. We fix the number of datapoints $n$ and the number of features $d$ to $500$ for the left and right plots, respectively. We  observe that $L / \Hat{L}_\pi$ grows with both  $d$ and $n$, thus supporting our claim that the bounds we proved are much tighter than previously known bounds.}
    \label{fig:gaussian-compare}
\end{figure}

\begin{table*}[t!]
\caption{The following table shows the computed values of $L / \hat{L}$ where $\hat{L}$ is the empirical mean of $\hat{L}_{\pi}$ over random permutations. We note that the quantity $\sqrt{L / \hat{L}}$ represents the improvement provided by the bound via our novel primal-dual perspective, compared to previous work.}
\label{table:constant-compare-dataset}
\vskip 0.15in
\begin{center}
\begin{small}
\begin{sc}
\begin{tabular}{lccccc}
    \toprule
    Dataset & \#Features ($d$) & \#Datapoints ($n$) & $L/\Hat{L}$ & $\log_n L/\Hat{L}$ & $\log_{\min(d, n)} L/\Hat{L}$ \\
    \midrule
    a1a & 123 & 1605 & 5.50 & 0.231 & 0.354 \\
    a9a & 123 & 32561 & 5.49 & 0.164 & 0.354 \\
    BBBC005 & 361920 & 19201 & 18.3 & 0.295 & 0.295 \\
    BBBC010 & 361920 & 201 & 7.04 & 0.368 & 0.368 \\
    cifar10 & 3072 & 50000 & 10.0 & 0.213 & 0.287 \\
    duke & 7129 & 44 & 38.0 & 0.962 & 0.962 \\
    e2006train & 150360 & 16087 & 5.35 & 0.173 & 0.173 \\
    gisette & 5000 & 6000 & 3.52 & 0.145 & 0.148 \\
    leu & 7129 & 38 & 32.8 & 0.960 & 0.960 \\
    mnist & 780 & 60000 & 19.1 & 0.268 & 0.443 \\
    news20 & 1355191 & 19996 & 42.1 & 0.378 & 0.378 \\
    rcv1 & 47236 & 20242 & 111 & 0.475 & 0.475 \\
    real-sim & 20958 & 72309 & 194 & 0.471 & 0.529 \\
    sonar & 60 & 208 & 6.26 & 0.344 & 0.448\\
    tmc2007 & 30438 & 21519 & 10.9 & 0.239 & 0.239 \\
    \bottomrule
\end{tabular}
\end{sc}
\end{small}
\end{center}
\vskip -0.1in
\end{table*}

In this section, we provide empirical evidence to support our claim about usefulness of the new convergence bounds obtained in our work. 
In particular, we conduct numerical evaluations to compare $\Hat{L}$ to the classical smoothness constant $L$ on synthetic datasets and on popular machine learning benchmark datasets. 
For a more streamlined comparison and to focus on the dependence on the data matrix, we assume that the loss functions $\ell_i$ all have the same smoothness constant, which leads to $L/\hat{L} = (\max_{1\leq i \leq n}\{\|\va_i\|^2\})/\big(\frac{1}{n^2}\|\sum_{j=1}^n\mI_{(j - 1)\uparrow} \mA_{\bar{\pi}}\mA_{\bar{\pi}}^{\top}\mI_{(j - 1)\uparrow}\|_2\big)$. Since the scale of the smoothness constant of the loss functions is irrelevant for the ratio $L/\hat{L}$ in this case, for simplicity, we take it to be equal to one. Note that assuming different smoothness constants over component loss functions would only make our bound better compared to related work (see Eq.~\eqref{eq:new-smoothness-constants} and the discussion following it).  

We implement the computation of $\Hat{L}$ and $L$ in \href{https://julialang.org}{Julia}, a high-performance scientific computation programming language, and compute matrix operator norms using the default settings in the Julia \texttt{Arpack} Package. However, limited by computational memory and time constraint, our selection of datasets is focused on moderately large-scale datasets of $n$ in the order of $O(10^5)$. We also include comparisons of small datasets such as \texttt{a1a} and \texttt{sonar}. Further results are provided in Appendix~\ref{appx:exp}.

We first illustrate the behavior of $L/\hat{L}$ on standard Gaussian data, where we 
fix
one of the parameters $n, d$ 
at 500 and we vary the other parameter from 1 to 1000, as shown in Fig.~\ref{fig:gaussian-compare}. The plots exhibit $\min\{n, d\}^\alpha$ growth for $L / \Hat{L}$, where $\alpha$ is around $0.3$ and $0.6$ for fixed $n$ and fixed $d,$ respectively. 

We also compare $\hat{L}$ and $L$ on 
a number of benchmarking datasets from LIBSVM~\citep{chang2011libsvm}, MNIST~\citep{deng2012mnist}, CIFAR10~\citep{Krizhevsky2009LearningML}, and Broad Bioimage Benchmark Collection~\citep{broadbioimage}. For each dataset, we generate a uniformly random permutation $\pi$ for the data matrix $\mA$ and compute $\Hat{L}_\pi$. We repeat this procedure $1000$ times for all datasets and display the average $L / \Hat{L}_\pi$ in Table~\ref{table:constant-compare-dataset}, except for \texttt{e2006train}, \texttt{CIFAR10}, \texttt{MNIST}, and $\texttt{BBBC005}$ where we do $20$ repetitions due to limitations of computation resources required for each calculation. 
We observe that among the datasets that we consider, which contain all three data matrix ``shapes'' 
$d >> n$, $d << n$, and $d \approx n$, our novel bound dependent on $\hat{L}$ is much tighter. 
For instance, for \texttt{rcv1} and \texttt{real-sim} datasets, where $d$ and $n$ are of the same order, we observe that $L / \hat{L}$ are approximately $111$ and $194$, respectively. 
For \texttt{news20} dataset where $d >> n$, $L / \hat{L} \approx 42.1$. For \texttt{MNIST}, where $d <<n,$ $L / \hat{L} \approx 19.1$. 

Finally, as a justification for using the empirical mean of $\Hat{L}_\pi$ over random permutations $\pi$ in the results displayed in Table \ref{table:constant-compare-dataset}, 
 we observe in our evaluations that the values of $L / \Hat{L}_\pi$  are fairly concentrated around their empirical mean values. Histogram plots showing the empirical distributions of $L / \Hat{L}_\pi$ for each of the datasets are provided in Appendix~\ref{appx:exp}.

\bibliography{ref}
\bibliographystyle{icml2024}

\newpage
\appendix
\onecolumn

\begin{center}{\LARGE\bfseries Supplementary Material}
\end{center}

\paragraph{Outline.} The appendix of the paper is organized as follows:
\begin{itemize}[leftmargin=*]
    \item Section~\ref{appx:smooth} presents the proofs related to the smooth convex setting from Section~\ref{sec:smooth}.
    \item Section~\ref{appx:nonsmooth} presents the proofs related to the nonsmooth convex Lipschitz setting from Section~\ref{sec:nonsmooth}.
    \item In Section~\ref{appx:general}, we present our refined analysis and results in a more general finite-sum setting considered in prior work where we only assume each component function $f_i$ to be convex and $L_i$-smooth.
    \item Section~\ref{appx:exp} presents the full details of the computational experiments performed in the paper.
\end{itemize}

\vspace{1em}

\section{Omitted Proofs From Section~\ref{sec:smooth}}\label{appx:smooth}
Before proceeding to the omitted proofs for the smooth convex settings, we first state the following standard definition and first-order characterization of strong convexity, for completeness. 
\begin{definition}
A function $f: \R^d \rightarrow \R$ is said to be $\mu$-strongly convex with parameter $\mu > 0$, if for any $\vx, \vy \in \R^d$ and any $\lambda \in (0, 1)$: 
\begin{equation*}
    f(\lambda\vx + (1 - \lambda)\vy) \leq \lambda f(\vx) + (1 - \lambda)f(\vy) - \frac{\mu}{2}\lambda(1 - \lambda)\|\vx - \vy\|_2^2.
\end{equation*}
\end{definition}
\begin{lemma}\label{lemma:strong-convex}
Let $f: \R^d \rightarrow \R$ be a continuous  $\mu$-strongly convex function with $\mu > 0$. Then, for any $\vx, \vy \in \R^d$: 
\begin{equation*}
    f(\vy) \geq f(\vx) + \innp{\vg_\vx, \vy - \vx} + \frac{\mu}{2}\|\vx - \vy\|^2_2, 
\end{equation*}
where $\vg_\vx \in \partial f(\vx)$, and $\partial f(\vx)$ is the subdifferential of $f$ at $\vx$.
\end{lemma}
We also include the following lemma on the variance bound under without-replacement sampling, which is useful for our proof of Lemma~\ref{lemma:second-inner}.
\begin{lemma}\label{lem:batch-vr}
Let $\gB$ be the set of $|\gB| = b$ samples from $[n]$, drawn without replacement and uniformly at random. Then, $\forall \vx \in \R^d$,
\begin{equation*}
\E_{\mathcal{B}}\Big[\big\|\frac{1}{b}\sum_{i\in\gB}\nabla f_i(\vx) - \nabla f(\vx)\big\|_2^2\Big] = \frac{n-b}{b(n-1)}\E_i\big[ \| \nabla f_i(\vx) - \nabla f(\vx)\|_2^2 \big].
\end{equation*}
\end{lemma}
\begin{proof}
We first expand the square on the left-hand side, as follows
\begin{equation*}
    \begin{aligned}
        \;& \E_{\mathcal{B}}\Big[\big\|\frac{1}{b}\sum_{i\in\gB}\nabla f_i(\vx) - \nabla f(\vx)\big\|_2^2\Big] \\
        = \;& \frac{1}{b^2}\E_{\mathcal{B}}\Big[\sum_{i, i' \in \mathcal{B}}\innp{\nabla f_i(\vx) - \nabla f(\vx), \nabla f_{i'}(\vx) - \nabla f(\vx)}\Big] \\
        = \;& \frac{1}{b^2}\E_{\mathcal{B}}\Big[\sum_{i, i' \in \mathcal{B}, i \neq i'}\innp{\nabla f_i(\vx) - \nabla f(\vx), \nabla f_{i'}(\vx) - \nabla f(\vx)}\Big] + \frac{1}{b}\E_i\Big[\|\nabla f_i(\vx) - \nabla f(\vx)\|_2^2\Big]. 
    \end{aligned}
\end{equation*}
Since the batch $\mathcal{B}$ is sampled uniformly and without replacement from $[n]$, the probability that any pair $(i, i')$ from $[n]$ with $i \neq i'$ is in $\mathcal{B}$ is $\frac{b(b - 1)}{n(n - 1)}$. By the linearity of expectation, we have 
\begin{equation*}
    \begin{aligned}
        \;& \E_{\mathcal{B}}\Big[\sum_{i, i' \in \mathcal{B}, i \neq i'}\innp{\nabla f_i(\vx) - \nabla f(\vx), \nabla f_{i'}(\vx) - \nabla f(\vx)}\Big] \\
        = \;& \E_{\mathcal{B}}\Big[\sum_{i, i' \in [n], i\neq i'}\mathbbm{1}_{i, i' \in \mathcal{B}}\innp{\nabla f_i(\vx) - \nabla f(\vx), \nabla f_{i'}(\vx) - \nabla f(\vx)}\Big]\\
        = \;& \sum_{i, i' \in [n], i \neq i'}\E_{\mathcal{B}}\Big[\mathbbm{1}_{i,  i' \in \mathcal{B}}\innp{\nabla f_i(\vx) - \nabla f(\vx), \nabla f_{i'}(\vx) - \nabla f(\vx)}\Big]\\
        = \;& \frac{b(b - 1)}{n(n - 1)}\sum_{i,  i' \in [n], i \neq i'}\innp{\nabla f_i(\vx) - \nabla f(\vx), \nabla f_{i'}(\vx) - \nabla f(\vx)}, 
    \end{aligned}
\end{equation*}
where $\mathbbm{1}$ is the indicator function such that $\mathbbm{1}_{i, i' \in \mathcal{B}} = 1$ if both $i, i' \in \mathcal{B}$ and is equal to zero otherwise. 
Hence, we obtain 
\begin{equation*}
    \begin{aligned}
        \;& \E_{\mathcal{B}}\Big[\big\|\frac{1}{b}\sum_{i\in\gB}\nabla f_i(\vx) - \nabla f(\vx)\big\|^2\Big] \\
        = \;& \frac{b - 1}{bn(n - 1)}\sum_{i, i' \in [n], i \neq i'}\innp{\nabla f_i(\vx) - \nabla f(\vx), \nabla f_{i'}(\vx) - \nabla f(\vx)} + \frac{1}{b}\E_i\Big[\|\nabla f_i(\vx) - \nabla f(\vx)\|^2\Big]\\
        = \;& \frac{b - 1}{bn(n - 1)}\sum_{i, i' \in [n]}\innp{\nabla f_i(\vx) - \nabla f(\vx), \nabla f_{i'}(\vx) - \nabla f(\vx)} + \frac{n - b}{b(n - 1)}\E_i\Big[\|\nabla f_i(\vx) - \nabla f(\vx)\|^2\Big]\\
        \overset{(\romannumeral1)}{=} \;& \frac{n - b}{b(n - 1)}\E_i\Big[\|\nabla^j f_i(\vx) - \nabla^j f(\vx)\|^2\Big], 
    \end{aligned}
\end{equation*}
where $(\romannumeral1)$ is due to $f = \frac{1}{n}\sum_{i = 1}^n f_i$ having the finite sum structure. 
\end{proof}

\subsection{Omitted Proofs for the Random Reshuffling/Shuffle-Once Schemes}\label{appx:convex-smooth}
\gapBound*
\begin{proof}
By Line 6
in Alg.~\ref{alg:PD-shuffled-sgd}, we have $\vy_k^{(i)} = \argmax_{\vy \in \R^b} \Big\{\vy^{\top}\mA_k^{(i)}\vx_{k - 1, i} - \sum_{j = 1}^b\ell_{\pi_{b(i - 1) + j}^{(k)}}^*(\vy^j)\Big\}$ for $i \in [m]$. Notice that since 
\begin{equation*}
    \vy^{\top}\mA_k^{(i)}\vx_{k - 1, i} - \sum_{j = 1}^b\ell_{\pi_{b(i - 1) + j}^{(k)}}^*(\vy^j) = \sum_{j = 1}^b\Big(\vy^j\va_{\pi_{b(i - 1) + j}^{(k)}}^\top\vx_{k - 1, i} - \ell_{\pi_{b(i - 1) + j}^{(k)}}^*(\vy^j)\Big)
\end{equation*}
is separable, we have $\vy_k^j = \argmax_{y \in \R}\{y\va_{\pi_{j}^{(k)}}^\top\vx_{k - 1, i} - \ell_{\pi_{j}^{(k)}}^*(y)\}$ for $b(i - 1) + 1 \leq j \leq bi$, thus $\va_{\pi_{j}^{(k)}}^\top\vx_{k - 1, i} \in \partial \ell_{\pi_{j}^{(k)}}^*(\vy_k^j)$. Since $\ell_i^*$ is $\frac{1}{L_i}$-strongly convex by Assumption~\ref{assp:smooth}, then by Lemma~\ref{lemma:strong-convex} we obtain for $b(i - 1) + 1 \leq j \leq bi$
\begin{equation*}
    \ell_{\pi_{j}^{(k)}}^*\big(\vv_k^j\big) \geq \ell_{\pi_{j}^{(k)}}^*(\vy_k^j) + \va_{\pi_{j}^{(k)}}^\top\vx_{k - 1, i}\big(\vv_k^j - \vy_k^j\big) + \frac{1}{2L_{\pi_j^{(k)}}}\big(\vv_k^j - \vy_k^j\big)^2,
\end{equation*}
which leads to 
\begin{align}
    \gL(\vx_{k}, \vv) = \;& \frac{1}{n}\sum_{i=1}^m \Big(\vv^{(i) \top}_k\mA_k^{(i)}\vx_{k - 1, i} - \sum_{j = b(i - 1) + 1}^{bi}\ell_{\pi^{(k)}_{j}}^*(\vv^{j}_k)\Big) + \frac{1}{n}\sum_{i=1}^m \vv^{(i) \top}_k\mA_k^{(i)}(\vx_{k} - \vx_{k - 1, i}) \notag\\
    \leq \;& \frac{1}{n}\sum_{i=1}^m \Big(\vy^{(i) \top}_k\mA_k^{(i)}\vx_{k - 1, i} - \sum_{j = b(i - 1) + 1}^{bi}\ell_{\pi^{(k)}_{j}}^*(\vy^{j}_k)\Big) + \frac{1}{n}\sum_{i=1}^m \vv^{(i) \top}_k\mA_k^{(i)}(\vx_{k} - \vx_{k - 1, i}) \notag \\
    & - \frac{1}{2n}\|\vy_k - \vv_k\|_{\mLambda_k^{-1}}^2 . \label{eq:ub-perm-batch}
\end{align}
Using the same argument for $\gL(\vx_*, \vy_*)$ as $\va_j^\top\vx_* \in \partial \ell_j^*(\vy_*^j)$ for $j \in [n]$, we have 
\begin{align}
    \gL(\vx_*, \vy_*) 
    = \;& \frac{1}{n}\sum_{i = 1}^m \Big(\vy_{*, k}^{(i) \top}\mA_k^{(i)}\vx_* - \sum_{j = b(i - 1) + 1}^{bi}\ell^*_{\pi_j^{(k)}}(\vy_{*, k}^j)\Big) \notag \\
    \geq \;& \frac{1}{n}\sum_{i = 1}^m\Big(\vy_k^{(i) \top}\mA_k^{(i)}\vx_* - \sum_{j = b(i - 1) + 1}^{bi}\ell^*_{\pi_j^{(k)}}(\vy_{k}^j)\Big) + \frac{1}{2n}\|\vy_k - \vy_{*, k}\|_{\mLambda_k^{-1}}^2. \label{eq:lb-perm-batch-1}
\end{align}
Adding and substracting the term $\frac{b}{2n\eta_k}\sum_{i = 1}^m\|\vx_* - \vx_{k - 1, i}\|_2^2$ on the R.H.S.\ of Eq.~\eqref{eq:lb-perm-batch-1}, we obtain 
\begin{equation*}
\begin{aligned}
    \gL(\vx_*, \vy_*) \geq \;& \frac{1}{n}\sum_{i = 1}^m\Big(\vy_k^{(i) \top}\mA_k^{(i)}\vx_* + \frac{b}{2\eta_k}\|\vx_* - \vx_{k - 1, i}\|_2^2 - \sum_{j = b(i - 1) + 1}^{bi}\ell^*_{\pi_j^{(k)}}(\vy_{k}^j)\Big) \\
    & - \frac{b}{2n\eta_k}\sum_{i = 1}^m\|\vx_* - \vx_{k - 1, i}\|_2^2 + \frac{1}{2n}\|\vy_k - \vy_{*, k}\|_{\mLambda_k^{-1}}^2.
\end{aligned}
\end{equation*}
By Line 7
of Alg.~\ref{alg:PD-shuffled-sgd}, we have $\vx_{k - 1, i + 1} = \argmin_{\vx \in \R^d}\Big\{\vy_k^{(i)\top}\mA_k^{(i)}\vx + \frac{b}{2\eta_k}\|\vx - \vx_{k - 1, i}\|_2^2\Big\}$. Further noticing that $\phi_k^{(i)}(\vx) := \vy_k^{(i)\top}\mA_k^{(i)}\vx + \frac{b}{2\eta_k}\|\vx - \vx_{k - 1, i}\|_2^2$ is $\frac{b}{\eta_k}$-strongly convex w.r.t.\ $\vx$ and $\nabla \phi_k^{(i)}(\vx_{k - 1, i + 1}) = \vzero$, we have 
\begin{equation*}
    \phi_k^{(i)}(\vx_*) \geq \phi_k^{(i)}(\vx_{k - 1, i + 1}) + \frac{b}{2\eta_k}\|\vx_* - \vx_{k - 1, i + 1}\|^2_2, 
\end{equation*}
which leads to 
\begin{align}
    \gL(\vx_*, \vy_*) \geq \;& \frac{1}{n}\sum_{i = 1}^m\Big(\vy_k^{(i) \top}\mA_k^{(i)}\vx_{k - 1, i + 1} + \frac{b}{2\eta_k}\|\vx_{k - 1, i + 1} - \vx_{k - 1, i}\|^2_2 - \sum_{j = b(i - 1) + 1}^{bi}\ell^*_{\pi_j^{(k)}}(\vy_{k}^j)\Big) \notag \\
    & + \frac{b}{2n\eta_k}\sum_{i = 1}^m\big(\|\vx_* - \vx_{k - 1, i + 1}\|^2_2 - \|\vx_* - \vx_{k - 1, i}\|_2^2\big) + \frac{1}{2n}\|\vy_k - \vy_{*, k}\|_{\mLambda_k^{-1}}^2 \notag \\
    \overset{(\romannumeral1)}{=} \;& \frac{1}{n}\sum_{i = 1}^m\Big(\vy_k^{(i) \top}\mA_k^{(i)}\vx_{k - 1, i + 1} + \frac{b}{2\eta_k}\|\vx_{k - 1, i + 1} - \vx_{k - 1, i}\|^2_2 - \sum_{j = b(i - 1) + 1}^{bi}\ell^*_{\pi_j^{(k)}}(\vy_{k}^j)\Big) \notag \\
    & + \frac{b}{2n\eta_k}\|\vx_{k} - \vx_*\|_2^2 - \frac{b}{2n\eta_k}\|\vx_{k - 1} - \vx_*\|_2^2 + \frac{1}{2n}\|\vy_k - \vy_{*, k}\|_{\mLambda_k^{-1}}^2, \label{eq:lb-perm-batch}
\end{align}
where we telescope from $i = 1$ to $m$ for the term $\sum_{i = 1}^m\big(\|\vx_* - \vx_{k - 1, i + 1}\|^2_2 - \|\vx_* - \vx_{k - 1, i}\|_2^2\big)$, and use the definitions that $\vx_k = \vx_{k - 1, m + 1}$ and $\vx_{k - 1} = \vx_{k - 1, 1}$ for $(\romannumeral1)$.

Combining the bounds from Eq.~\eqref{eq:ub-perm-batch}~and~Eq.~\eqref{eq:lb-perm-batch}  and denoting 
\begin{equation*}
    \gE_k := \eta_k \big(\gL(\vx_{k}, \vv) - \gL(\vx_*, \vy_*)\big) + \frac{b}{2n}\|\vx_* - \vx_{k}\|_2^2 - \frac{b}{2n} \|\vx_* - \vx_{k - 1}\|_2^2, 
\end{equation*}
we obtain 
\begin{equation*}
\begin{aligned}
    \gE_k \leq \;& \frac{\eta_k}{n}\sum_{i=1}^m \vy_k^{(i) \top}\mA_k^{(i)}(\vx_{k - 1, i} - \vx_{k - 1, i + 1}) + \frac{\eta_k}{n}\sum_{i=1}^m \vv^{(i) \top}_k\mA_k^{(i)}(\vx_{k} - \vx_{k - 1, i}) \\
    &- \frac{\eta_k}{2n}\|\vy_k - \vv_k\|_{\mLambda_k^{-1}}^2 - \frac{\eta_k}{2n}\|\vy_k - \vy_{*, k}\|_{\mLambda_k^{-1}}^2 - \frac{b}{2n}\sum_{i=1}^m \|\vx_{k - 1, i} - \vx_{k - 1, i+1}\|_2^2 \\
    = \;& \frac{\eta_k}{n}\sum_{i=1}^m \vy_k^{(i) \top}\mA_k^{(i)}(\vx_{k} - \vx_{k - 1, i + 1}) + \frac{\eta_k}{n}\sum_{i=1}^m (\vv^{(i)}_k - \vy_k^{(i)})^{\top}\mA_k^{(i)}(\vx_{k} - \vx_{k - 1, i}) \\
    &- \frac{\eta_k}{2n}\|\vy_k - \vv_k\|_{\mLambda_k^{-1}}^2 - \frac{\eta_k}{2n}\|\vy_k - \vy_{*, k}\|_{\mLambda_k^{-1}}^2 - \frac{b}{2n}\sum_{i=1}^m \|\vx_{k - 1, i} - \vx_{k - 1, i+1}\|_2^2, 
\end{aligned}
\end{equation*}
thus completing the proof. 
\end{proof}

\firstInner*
\begin{proof}
By Line 7
in Alg.~\ref{alg:PD-shuffled-sgd}, we have $\mA_k^{(i) \top}\vy_k^{(i)} = \frac{b}{\eta_k}(\vx_{k - 1, i} - \vx_{k - 1, i + 1})$. 
Further noticing that $\vx_{k} - \vx_{k - 1, i + 1} = -\sum_{j = i + 1}^m(\vx_{k - 1, j} - \vx_{k - 1, j + 1})$, we obtain 
\begin{equation*}
\begin{aligned}
    \gT_1 := \;& \frac{\eta_k}{n}\sum_{i=1}^m \vy_k^{(i) \top}\mA_k^{(i)}(\vx_{k} - \vx_{k - 1, i + 1}) \\
    = \;& -\frac{b}{n}\sum_{i = 1}^{m - 1}\sum_{j = i + 1}^m\innp{\vx_{k - 1, i} - \vx_{k - 1, i + 1}, \vx_{k - 1, j} - \vx_{k - 1, j + 1}} \\
    = \;& \frac{b}{2n}\sum_{i = 1}^m\|\vx_{k - 1, i} - \vx_{k - 1, i + 1}\|^2 - \frac{b}{2n}\Big\|\sum_{i = 1}^m (\vx_{k - 1, i} - \vx_{k - 1, i + 1})\Big\|^2, 
\end{aligned}
\end{equation*}
thus completing the proof.
\end{proof}

\secondInner*
\begin{proof}
By Line 7
in Alg.~\ref{alg:PD-shuffled-sgd}, we have $\vx_{k - 1, i} - \vx_{k - 1, i + 1} = \frac{\eta_k}{b}\mA_k^{(i) \top}\vy_k^{(i)}$. Using the definition of $\mI_{j\uparrow}$ for $0 \leq j \leq n - 1$ as in Section~\ref{sec:prelim}, we obtain 
\begin{equation*}
    \vx_{k} - \vx_{k - 1, i} = -\sum_{j = i}^m(\vx_{k - 1, j} - \vx_{k - 1, j + 1}) = -\frac{\eta_k}{b}\sum_{j = i}^m \mA_k^{(j) \top}\vy_k^{(j)} = -\frac{\eta_k}{b}\mA_k\mI_{b(i - 1)\uparrow}\vy_k.
\end{equation*}
Also, we have $\mA_k^{(i)\top}(\vv_k^{(i)} - \vy_k^{(i)}) = \mA_k\mI_{(i)}(\vv_k - \vy_k)$ by the definition of $\mI_{(i)}$ in Section~\ref{sec:smooth}. Combining these two observations, we have 
\begin{align}
    \gT_2 := \;& \frac{\eta_k}{n}\sum_{i = 1}^m \big(\vv^{(i)}_k - \vy_k^{(i)}\big)^{\top}\mA_k^{(i)}(\vx_{k} - \vx_{k - 1, i}) \notag \\
    = \;& -\frac{\eta_k^2}{bn}\sum_{i = 1}^m\innp{\mA_k^{\top}\mI_{b(i - 1)\uparrow}\vy_k, \mA_k^{\top}\mI_{(i)}(\vv_k - \vy_k)} \notag \\
    \overset{(\romannumeral1)}{=} \;& -\frac{\eta_k^2}{bn}\sum_{i = 1}^m\innp{\mA_k^{\top}\mI_{b(i - 1)\uparrow}(\vy_k - \vy_{*, k}), \mA_k^{\top}\mI_{(i)}(\vv_k - \vy_k)} \label{eq:inn-prod-2-fin-bnd-1} \\
    & -\frac{\eta_k^2}{bn}\sum_{i = 1}^m\innp{\mA_k^{\top}\mI_{b(i - 1)\uparrow}\vy_{*, k}, \mA_k^{\top}\mI_{(i)}(\vv_k - \vy_k)}, \label{eq:inn-prod-2-fin-bnd-2}
\end{align}
where we make a decomposition w.r.t.\ $\vy_{*, k}$ in~$(\romannumeral1)$.
For the first term in Eq.~\eqref{eq:inn-prod-2-fin-bnd-1}, we use Young's inequality for $\alpha > 0$ and have 
\begin{equation}\label{eq:inn-prod-2-bnd-1}
\begin{aligned}
    \;& -\frac{\eta_k^2}{bn}\sum_{i = 1}^m\innp{\mA_k^{\top}\mI_{b(i - 1)\uparrow}(\vy_k - \vy_{*, k}), \mA_k^{\top}\mI_{(i)}(\vv_k - \vy_k)} \\
    \leq \;& \frac{\eta_k^2\alpha}{2bn}\sum_{i = 1}^m\|\mA_k^{\top}\mI_{b(i - 1)\uparrow}(\vy_k - \vy_{*, k})\|_2^2 + \frac{\eta_k^2}{2bn\alpha}\sum_{i = 1}^m\|\mA_k^{\top}\mI_{(i)}(\vv_k - \vy_k)\|_2^2.
\end{aligned}
\end{equation}
Expanding the squares and rearranging the terms in Eq.~\eqref{eq:inn-prod-2-bnd-1}, we have 
\begin{equation}\label{eq:cyclic}
\begin{aligned}
    \;& \frac{\eta_k^2\alpha}{2bn}\sum_{i = 1}^m\|\mA_k^{\top}\mI_{b(i - 1)\uparrow}(\vy_k - \vy_{*, k})\|_2^2 \\
    = \;& \frac{\eta_k^2\alpha}{2bn}\sum_{i = 1}^m(\vy_k - \vy_{*, k})^\top\mI_{b(i - 1)\uparrow}\mA_k\mA_k^{\top}\mI_{b(i - 1)\uparrow}(\vy_k - \vy_{*, k}) \\
    = \;& \frac{\eta_k^2\alpha}{2bn}(\vy_k - \vy_{*, k})^\top\Big(\sum_{i = 1}^m\mI_{b(i - 1)\uparrow}\mA_k\mA_k^{\top}\mI_{b(i - 1)\uparrow}\Big)(\vy_k - \vy_{*, k}) \\
    = \;& \frac{\eta_k^2\alpha}{2bn}(\vy_k - \vy_{*, k})^\top\mLambda_k^{-1/2}\mLambda_k^{1/2}\Big(\sum_{i = 1}^m\mI_{b(i - 1)\uparrow}\mA_k\mA_k^{\top}\mI_{b(i - 1)\uparrow}\Big)\mLambda_k^{1/2}\mLambda_k^{-1/2}(\vy_k - \vy_{*, k}) \\
    \overset{(\romannumeral1)}{\leq} \;& \frac{\eta_k^2\alpha}{2bn}\Big\|\mLambda_k^{1/2}\Big(\sum_{i = 1}^m\mI_{b(i - 1)\uparrow}\mA_k\mA_k^{\top}\mI_{b(i - 1)\uparrow}\Big)\mLambda_k^{1/2}\Big\|_2\|\vy_k - \vy_{*, k}\|^2_{\mLambda_k^{-1}}, 
\end{aligned}
\end{equation}
where we use Cauchy-Schwarz inequality for $(\romannumeral1)$. Using a similar argument, we also have 
\begin{equation*}
    \frac{\eta_k^2}{2bn\alpha}\sum_{i = 1}^m\|\mA_k^{\top}\mI_{(i)}(\vv_k - \vy_k)\|_2^2 \leq \frac{\eta_k^2}{2bn\alpha}\Big\|\mLambda_k^{1/2}\Big(\sum_{i = 1}^m\mI_{(i)}\mA_k\mA_k^{\top}\mI_{(i)}\Big)\mLambda_k^{1/2}\Big\|_2\|\vv_k - \vy_{k}\|^2_{\mLambda_k^{-1}}.
\end{equation*}
By the definitions of $\Hat{L}_{\pi^{(k)}}$ and $\Tilde{L}_{\pi^{(k)}}$, and choosing $\alpha = 2\eta_k\Tilde{L}_{\pi^{(k)}}$ in Eq.~\eqref{eq:inn-prod-2-bnd-1}, we obtain  
\begin{equation}\label{eq:inn-prod-2-fin-1}
\begin{aligned}
    \;& -\frac{\eta_k^2}{bn}\sum_{i = 1}^m\innp{\mA_k^{\top}\mI_{b(i - 1)\uparrow}(\vy_k - \vy_{*, k}), \mA_k^{\top}\mI_{(i)}(\vv_k - \vy_k)} \\
    \leq \;& \frac{\eta_k^3 n \Hat{L}_{\pi^{(k)}}\Tilde{L}_{\pi^{(k)}}}{b^2}\|\vy_k - \vy_{*, k}\|^2_{\mLambda_k^{-1}} + \frac{\eta_k}{4n}\|\vv_k - \vy_{k}\|_{\mLambda_k^{-1}}^2.
\end{aligned}
\end{equation}
For the second term in Eq.~\eqref{eq:inn-prod-2-fin-bnd-2}, we apply Young's inequality with $\beta > 0$ and proceed as above: 
\begin{equation*}
\begin{aligned}
    \;& -\frac{\eta_k^2}{bn}\sum_{i = 1}^m\innp{\mA_k^{\top}\mI_{b(i - 1)\uparrow}\vy_{*, k}, \mA_k^{\top}\mI_{(i)}(\vv_k - \vy_k)} \\
    \leq \;& \frac{\eta_k^2\beta}{2bn}\sum_{i = 1}^m\|\mA_k^{\top}\mI_{b(i - 1)\uparrow}\vy_{*, k}\|_2^2 + \frac{\eta_k^2}{2bn\beta}\sum_{i = 1}^m\|\mA_k^{\top}\mI_{(i)}(\vv_k - \vy_k)\|_2^2 \\
    \leq \;& \frac{\eta_k^2\beta}{2bn}\sum_{i = 1}^m\|\mA_k^{\top}\mI_{b(i - 1)\uparrow}\vy_{*, k}\|_2^2 + \frac{\eta_k^2}{2n\beta}\Tilde{L}_{\pi^{(k)}}\|\vv_k - \vy_k\|^2_{\mLambda_k^{-1}}.
\end{aligned}
\end{equation*}
Noticing that $\Tilde{L}_{\pi^{(k)}} \leq \Tilde{L}$, we choose $\beta = 2\eta_k\Tilde{L}$ and obtain 
\begin{equation}\label{eq:inn-prod-2-fin-2}
\begin{aligned}
    \;& -\frac{\eta_k^2}{bn}\sum_{i = 1}^m\innp{\mA_k^{\top}\mI_{b(i - 1)\uparrow}\vy_{*, k}, \mA_k^{\top}\mI_{(i)}(\vv_k - \vy_k)} \\
    \leq \;& \frac{\eta_k^3\Tilde{L}}{nb}\sum_{i = 1}^m\|\mA_k^{\top}\mI_{b(i - 1)\uparrow}\vy_{*, k}\|_2^2 + \frac{\eta_k}{4n}\|\vv_k - \vy_k\|^2_{\mLambda_k^{-1}}. 
\end{aligned}
\end{equation}
Combining Eq.~\eqref{eq:inn-prod-2-fin-1} and Eq.~\eqref{eq:inn-prod-2-fin-2}, we have 
\begin{equation}\label{eq:inn-prod-2-bnd}
\gT_2 \leq \frac{\eta_k^3\Tilde{L}}{nb}\sum_{i = 1}^m\|\mA_k^{\top}\mI_{b(i - 1)\uparrow}\vy_{*, k}\|_2^2 + \frac{\eta_k^3 n \Hat{L}_{\pi^{(k)}}\Tilde{L}_{\pi^{(k)}}}{b^2}\|\vy_k - \vy_{*, k}\|^2_{\mLambda_k^{-1}} + \frac{\eta_k}{2n}\|\vv_k - \vy_{k}\|_{\mLambda_k^{-1}}^2.
\end{equation}
We first assume the RR scheme. Taking conditional expectation w.r.t.\ the randomness up to but not including $k$-th epoch, we have 
\begin{equation*}
\E_k[\gT_2] \leq \frac{\eta_k^3\Tilde{L}}{nb}\E_k\Big[\sum_{i = 1}^m\|\mA_k^{\top}\mI_{b(i - 1)\uparrow}\vy_{*, k}\|_2^2\Big] + \E_k\Big[\frac{\eta_k^3 n \Hat{L}_{\pi^{(k)}}\Tilde{L}_{\pi^{(k)}}}{b^2}\|\vy_k - \vy_{*, k}\|^2_{\mLambda_k^{-1}} + \frac{\eta_k}{2n}\|\vv_k - \vy_{k}\|_{\mLambda_k^{-1}}^2\Big].
\end{equation*}
For the first term $\frac{\eta_k^3\Tilde{L}}{nb}\E_k\Big[\sum_{i = 1}^m\|\mA_k^{\top}\mI_{b(i - 1)\uparrow}\vy_{*, k}\|_2^2\Big]$, the only randomness is from the random permutation $\pi^{(k)}$. In this case, each term $\mA_k^{\top}\mI_{b(i - 1)\uparrow}\vy_{*, k}$ can be considered as a sum of a batch sampled without replacement from $\{\vy_*^j\va_j\}_{j \in [n]}$, while $\sum_{j = 1}^n\vy_*^j\va_j = 0$ as $\vx_*$ is the minimizer, we then can use Lemma~\ref{lem:batch-vr} and obtain 
\begin{equation*}
\begin{aligned}
    \frac{\eta_k^3\Tilde{L}}{nb}\E_{k}\Big[\sum_{i = 1}^m\|\mA_k^{\top}\mI_{b(i - 1)\uparrow}\vy_{*, k}\|_2^2\Big] \overset{(\romannumeral1)}{=} \;& \frac{\eta_k^3\Tilde{L}}{nb}\sum_{i = 1}^m\E_{\pi^{(k)}}[\|\mA_k^{\top}\mI_{b(i - 1)\uparrow}\vy_{*, k}\|_2^2] \\
    = \;& \frac{\eta_k^3\Tilde{L}}{nb}\sum_{i = 1}^m\big(n - b(i - 1)\big)^2\E_{\pi^{(k)}}\Big[\Big\|\frac{\mA_k^{\top}\mI_{b(i - 1)\uparrow}\vy_{*, k}}{n - b(i - 1)}\Big\|_2^2\Big] \\
    \overset{(\romannumeral2)}{\leq} \;& \frac{\eta_k^3\Tilde{L}}{nb}\sum_{i = 1}^m\big(n - b(i - 1)\big)^2\frac{b(i - 1)}{\big(n - b(i - 1)\big)(n - 1)}\sigma_*^2 \\
    = \;& \frac{\eta_k^3\Tilde{L}(n - b)(n + b)}{6b^2(n - 1)}\sigma_*^2,
\end{aligned}
\end{equation*}
where $(\romannumeral1)$ is due to the linearity of expectation, and we use our definition $\sigma_*^2 = \frac{1}{n}\sum_{j = 1}^n(\vy_*^j)^2\|\va_j\|_2^2 = \E_j\big[\|\vy_*^j\va_j\|^2_2\big]$ for $(\romannumeral2)$. 
Taking expectation w.r.t.\ all the randomness on both sides and using the law of total expectation, we obtain 
\begin{equation*}
\begin{aligned}
    \E[\gT_2] \leq \E\Big[\frac{\eta_k^3 n \Hat{L}_{\pi^{(k)}}\Tilde{L}_{\pi^{(k)}}}{b^2}\|\vy_k - \vy_{*, k}\|^2_{\mLambda_k^{-1}} + \frac{\eta_k}{2n}\|\vv_k - \vy_{k}\|_{\mLambda_k^{-1}}^2\Big] + \frac{\eta_k^3\Tilde{L}(n - b)(n + b)}{6b^2(n - 1)}\sigma_*^2.
\end{aligned}
\end{equation*}
For the SO scheme, since there is only one random permutation generated at the very beginning, we can take expectation w.r.t.\ all the randomness on both sides of~\eqref{eq:inn-prod-2-bnd}, and the randomness for the term $\frac{\eta_k^3\Tilde{L}}{nb}\E\Big[\sum_{i = 1}^m\|\mA_k^{\top}\mI_{b(i - 1)\uparrow}\vy_{*, k}\|_2^2\Big]$ is only from the initial random permutation. So the above argument still applies to this case, and we complete the proof.
\end{proof}

\convergence*
\begin{proof}
Combining the bounds in Lemma~\ref{lemma:first-inner}~and~\ref{lemma:second-inner} and plugging them into Eq.~\eqref{eq:gap-bound}, we obtain 
\begin{equation*}
    \E[\gE_k] \leq \E\Big[\Big(\frac{\eta_k^3 n \Hat{L}_{\pi^{(k)}}\Tilde{L}_{\pi^{(k)}}}{b^2} - \frac{\eta_k}{2n}\Big)\|\vy_k - \vy_{*, k}\|^2_{\mLambda_k^{-1}}\Big] + \frac{\eta_k^3\Tilde{L}(n - b)(n + b)}{6b^2(n - 1)}\sigma_*^2.
\end{equation*}
For the stepsize $\eta_k$ such that $\eta_k \leq \frac{b}{n\sqrt{2\Hat{L}_{\pi^{(k)}}\Tilde{L}_{\pi^{(k)}}}}$, we have $\frac{\eta_k^3 n \Hat{L}_{\pi^{(k)}}\Tilde{L}_{\pi^{(k)}}}{b^2} - \frac{\eta_k}{2n} \leq 0$, thus 
\begin{equation*}
    \E[\gE_k] \leq \frac{\eta_k^3\Tilde{L}(n - b)(n + b)}{6b^2(n - 1)}\sigma_*^2. 
\end{equation*}
Noticing that $\gE_k = \eta_k\text{Gap}^\vv(\vx_{k}, \vy_*) + \frac{b}{2n}\|\vx_* - \vx_{k}\|_2^2 - \frac{b}{2n} \|\vx_* - \vx_{k - 1}\|_2^2$ and telescoping from $k = 1$ to $K$, we have 
\begin{equation*}
    \E\Big[\sum_{k = 1}^K\eta_k\text{Gap}^\vv(\vx_{k}, \vy_*)\Big] \leq \frac{b}{2n}\|\vx_* - \vx_0\|_2^2 - \frac{b}{2n}\E[\|\vx_* - \vx_{K}\|_2^2] + \sum_{k = 1}^K\frac{\eta_k^3\Tilde{L}(n - b)(n + b)}{6b^2(n - 1)}\sigma_*^2.
\end{equation*}
Noticing that $\gL(\vx, \vv)$ is convex w.r.t.\ $\vx$, we have $\text{Gap}^\vv(\Hat\vx_{K}, \vy_*) \leq \sum_{k = 1}^K\eta_k\text{Gap}^\vv(\vx_{k}, \vy_*) / H_K$, where $\Hat{\vx}_K = \sum_{k=1}^K \eta_k \vx_{k} / H_K$ and $H_K = \sum_{k = 1}^K\eta_k$, which leads to  
\begin{equation*}
    \E\Big[H_K\text{Gap}^\vv(\Hat\vx_{K}, \vy_*)\Big] \leq \frac{b}{2n}\|\vx_0 - \vx_*\|_2^2 + \sum_{k=1}^K \frac{\eta_k^3\Tilde{L}(n - b)(n + b)}{6b^2(n - 1)}\sigma_*^2. 
\end{equation*}
Further choosing $\vv = \vy_{\Hat{\vx}_K}$, we obtain 
\begin{equation}\label{eq:final-bound}
    \E[H_K\big(f(\Hat{\vx}_K) - f(\vx_*)\big)] \leq \frac{b}{2n}\|\vx_0 - \vx_*\|_2^2 + \sum_{k=1}^K\frac{\eta_k^3\Tilde{L}(n - b)(n + b)}{6b^2(n - 1)}\sigma_*^2.
\end{equation}
To analyze the individual gradient oracle complexity, we choose constant stepsizes $\eta \leq \frac{b}{n\sqrt{2\Hat{L}\Tilde{L}}}$, then Eq.~\eqref{eq:final-bound} will become 
\begin{equation*}
\begin{aligned}
    \E[f(\Hat{\vx}_K) - f(\vx_*)] \leq \frac{b}{2n\eta K}\|\vx_0 - \vx_*\|_2^2 + \frac{\eta^2\Tilde{L}(n - b)(n + b)}{6b^2(n - 1)}\sigma_*^2.
\end{aligned}
\end{equation*}
Without loss of generality, we assume that $b \neq n$, otherwise the method and its analysis reduce to (full) gradient descent. We consider the following two cases: 
\begin{itemize}[leftmargin=*]
    \item ``Small $K$'' case: if $\eta = \frac{b}{n\sqrt{2\Hat{L}\Tilde{L}}} \leq \Big(\frac{3b^3(n - 1)\|\vx_0 - \vx_*\|^2_2}{n(n - b)(n + b)\Tilde{L}K\sigma_*^2}\Big)^{1/3}$, we have 
    \begin{equation*}
    \begin{aligned}
        \E[f(\Hat{\vx}_K) - f(\vx_*)] \leq \;& \frac{b}{2n\eta K}\|\vx_0 - \vx_*\|_2^2 + \frac{\eta^2\Tilde{L}(n - b)(n + b)}{6b^2(n - 1)}\sigma_*^2\\
        \leq \;& \frac{\sqrt{\Hat{L}\Tilde{L}}}{\sqrt{2}K}\|\vx_0 - \vx_*\|_2^2 + \frac{1}{2}\Big(\frac{(n - b)(n + b)}{n^2(n - 1)}\Big)^{1/3}\frac{\Tilde{L}^{1/3}\sigma_*^{2/3}\|\vx_0 - \vx_*\|^{4/3}_2}{3^{1/3}K^{2/3}}.
    \end{aligned}
    \end{equation*}
    \item ``Large $K$'' case:
    if $\eta = \Big(\frac{3b^3(n - 1)\|\vx_0 - \vx_*\|^2_2}{n(n - b)(n + b)\Tilde{L}K\sigma_*^2}\Big)^{1/3} \leq \frac{b}{n\sqrt{2\Hat{L}\Tilde{L}}}$, we have 
    \begin{equation*}
    \begin{aligned}
        \E[f(\Hat{\vx}_K) - f(\vx_*)] \leq \;& \frac{b}{2n\eta K}\|\vx_0 - \vx_*\|_2^2 + \frac{\eta^2\Tilde{L}(n - b)(n + b)}{6b^2(n - 1)}\sigma_*^2 \\
        \leq \;& \Big(\frac{(n - b)(n + b)}{n^2(n - 1)}\Big)^{1/3}\frac{\Tilde{L}^{1/3}\sigma_*^{2/3}\|\vx_0 - \vx_*\|^{4/3}_2}{3^{1/3}K^{2/3}}. 
    \end{aligned}
    \end{equation*}
\end{itemize}
Combining these two cases by setting $\eta = \min\Big\{\frac{b}{n\sqrt{2\Hat{L}\Tilde{L}}} ,\, \Big(\frac{3b^3(n - 1)\|\vx_0 - \vx_*\|^2_2}{n(n - b)(n + b)\Tilde{L}K\sigma_*^2}\Big)^{1/3}\Big\}$, we obtain  
\begin{equation*}
    \E[f(\Hat{\vx}_K) - f(\vx_*)] \leq \frac{\sqrt{\Hat{L}\Tilde{L}}}{\sqrt{2}K}\|\vx_0 - \vx_*\|_2^2 + \Big(\frac{(n - b)(n + b)}{n^2(n - 1)}\Big)^{1/3}\frac{\Tilde{L}^{1/3}\sigma_*^{2/3}\|\vx_0 - \vx_*\|^{4/3}_2}{3^{1/3}K^{2/3}}.
\end{equation*}
Hence, to guarantee $\E[f(\Hat{\vx}_K) - f(\vx_*)] \leq \epsilon$ for $\epsilon > 0$, the total number of individual gradient evaluations will be 
\begin{equation*}
    nK \geq \max\Big\{\frac{n\sqrt{2\Hat{L}\Tilde{L}}\|\vx_0 - \vx_*\|_2^2}{\epsilon}, \Big(\frac{(n - b)(n + b)}{n - 1}\Big)^{1/2}\frac{2^{3/2}\Tilde{L}^{1/2}\sigma_*\|\vx_0 - \vx_*\|_2^2}{3^{1/2}\epsilon^{3/2}}\Big\}, 
\end{equation*}
as claimed.  
\end{proof}

\subsection{Omitted Proofs for Incremental Gradient Descenet}
We now provide the proof for Theorem~\ref{thm:convergence-IGD} in Section~\ref{sec:smooth} in the smooth convex settings. We first prove the following technical lemma, which bounds the inner product term $\gT_2 := \frac{\eta_k}{n}\sum_{i = 1}^m \big(\vv^{(i)} - \vy_k^{(i)}\big)^{\top}\mA^{(i)}(\vx_{k} - \vx_{k - 1, i})$ without random permutations involved.
\begin{restatable}{lemma}{secondInnerIGD}
\label{lemma:second-inner-IGD}
For any $k \in [K]$, the iterates $\{\vy_k^{(i)}\}_{i = 1}^m$ and $\{\vx_{k - 1, i}\}_{i = 1}^{m + 1}$ generated by Algorithm~\ref{alg:PD-shuffled-sgd} with fixed data ordering satisfy 
\begin{equation}\label{eq:sec-inn-bnd-IGD}
\begin{aligned}
\gT_2 \leq \frac{\eta_k^3 n}{b^2}\Hat{L}_0\Tilde{L}_0\|\vy_k - \vy_{*}\|^2_{\mLambda^{-1}} + \frac{\eta_k}{2n}\|\vv - \vy_k\|^2_{\mLambda^{-1}} + \min\Big\{\frac{\eta_k^3 n}{b^2}\Hat{L}_0\Tilde{L}_0\|\vy_{*}\|^2_{\mLambda^{-1}}, \frac{\eta_k^3 (n - b)^2}{b^2}\Tilde{L}_0\sigma_*^2\Big\}. 
\end{aligned}
\end{equation}
\end{restatable}
\begin{proof}
Proceeding as in Lemma~\ref{lemma:second-inner}, we have 
\begin{align}
    \gT_2 := \;& \frac{\eta_k}{n}\sum_{i = 1}^m \big(\vv^{(i)} - \vy_k^{(i)}\big)^{\top}\mA^{(i)}(\vx_{k} - \vx_{k - 1, i}) \notag \\
    = \;& -\frac{\eta_k^2}{bn}\sum_{i = 1}^m\innp{\mA^{\top}\mI_{b(i - 1)\uparrow}\vy_k, \mA^{\top}\mI_{(i)}(\vv - \vy_k)} \notag \\
    = \;& -\frac{\eta_k^2}{bn}\sum_{i = 1}^m\innp{\mA^{\top}\mI_{b(i - 1)\uparrow}(\vy_k - \vy_{*}), \mA^{\top}\mI_{(i)}(\vv - \vy_k)} \label{eq:inn-prod-2-fin-bnd-1-IGD} \\
    & -\frac{\eta_k^2}{bn}\sum_{i = 1}^m\innp{\mA^{\top}\mI_{b(i - 1)\uparrow}\vy_{*}, \mA^{\top}\mI_{(i)}(\vv - \vy_k)}, \label{eq:inn-prod-2-fin-bnd-2-IGD}
\end{align}
For both terms in Eq.~\eqref{eq:inn-prod-2-fin-bnd-1-IGD} and Eq.~\eqref{eq:inn-prod-2-fin-bnd-2-IGD}, we use Young's inequality for $\alpha = 2\eta_k \Tilde{L}_0 > 0$ and proceed as in Eq.~\eqref{eq:cyclic} to obtain
\begin{align}
    \;& -\frac{\eta_k^2}{bn}\sum_{i = 1}^m\innp{\mA^{\top}\mI_{b(i - 1)\uparrow}(\vy_k - \vy_{*}), \mA^{\top}\mI_{(i)}(\vv - \vy_k)} \notag \\
    \leq \;& \frac{\eta_k^2\alpha}{2bn}\sum_{i = 1}^m\|\mA^{\top}\mI_{b(i - 1)\uparrow}(\vy_k - \vy_{*})\|_2^2 + \frac{\eta_k^2}{2bn\alpha}\sum_{i = 1}^m\|\mA^{\top}\mI_{(i)}(\vv - \vy_k)\|_2^2 \notag\\
    \leq \;& \frac{\eta_k^2 n \alpha}{2b^2}\Hat{L}_0\|\vy_k - \vy_{*}\|^2_{\mLambda^{-1}} + \frac{\eta_k^2}{2n\alpha}\Tilde{L}_0\|\vv - \vy_k\|^2_{\mLambda^{-1}} \notag \\
    = \;& \frac{\eta_k^3 n}{b^2}\Hat{L}_0\Tilde{L}_0\|\vy_k - \vy_{*}\|^2_{\mLambda^{-1}} + \frac{\eta_k}{4n}\|\vv - \vy_k\|^2_{\mLambda^{-1}} \label{eq:sec-inn-IGD-1}
\end{align}
and 
\begin{align}
    -\frac{\eta_k^2}{bn}\sum_{i = 1}^m\innp{\mA^{\top}\mI_{b(i - 1)\uparrow}\vy_{*}, \mA^{\top}\mI_{(i)}(\vv - \vy_k)}
    \leq \;& \frac{\eta_k^2\alpha}{2bn}\sum_{i = 1}^m\|\mA^{\top}\mI_{b(i - 1)\uparrow}\vy_{*}\|_2^2 + \frac{\eta_k^2}{2bn\alpha}\sum_{i = 1}^m\|\mA^{\top}\mI_{(i)}(\vv - \vy_k)\|_2^2 \notag \\
    \leq \;& \frac{\eta_k^2\alpha}{2bn}\sum_{i = 1}^m\|\mA^{\top}\mI_{b(i - 1)\uparrow}\vy_{*}\|_2^2 + \frac{\eta_k^2}{2n\alpha}\Tilde{L}_0\|\vv - \vy_k\|^2_{\mLambda^{-1}} \notag \\
    = \;& \frac{\eta_k^3\Tilde{L}_0}{nb}\sum_{i = 1}^m\|\mA^{\top}\mI_{b(i - 1)\uparrow}\vy_{*}\|_2^2 + \frac{\eta_k}{4n}\|\vv - \vy_k\|^2_{\mLambda^{-1}}, \label{eq:sec-inn-IGD-2}
\end{align}
where again we used $\alpha = 2\eta_k\Tilde{L}_0$.
We then prove the term $\frac{\eta_k^3\Tilde{L}_0}{nb}\sum_{i = 1}^m\|\mA^{\top}\mI_{b(i - 1)\uparrow}\vy_{*}\|_2^2$ in Eq.~\eqref{eq:sec-inn-IGD-2} is no larger than the minimum of $\frac{\eta_k^3 n}{b^2}\Hat{L}_0\Tilde{L}_0\|\vy_{*}\|^2_{\mLambda^{-1}}$ and 
$\frac{\eta_k^3 (n - b)^2}{b^2}\Tilde{L}_0\sigma_*^2$.
Note that when $b = n$, we have $\mA^{\top}\mI_{(0)\uparrow}\vy_{*} = 0$, so this term disappears. When $b < n$, the former one can be derived as in Eq.\eqref{eq:cyclic}, which gives 
\begin{equation*}
    \sum_{i = 1}^m\|\mA^{\top}\mI_{b(i - 1)\uparrow}\vy_{*}\|_2^2 \leq \Big\|\mLambda^{1/2}\Big(\sum_{i = 1}^m\mI_{b(i - 1)\uparrow}\mA\mA^{\top}\mI_{b(i - 1)\uparrow}\Big)\mLambda^{1/2}\Big\|_2\|\vy_*\|^2_{\mLambda^{-1}} = mn\Hat{L}_0\|\vy_*\|^2_{\mLambda^{-1}} = \frac{n^2}{b}\Hat{L}_0\|\vy_*\|^2_{\mLambda^{-1}}.
\end{equation*}
For the latter one, we notice that 
\begin{equation*}
\begin{aligned}
    \sum_{i = 1}^m\|\mA^{\top}\mI_{b(i - 1)\uparrow}\vy_{*}\|_2^2 = \;& \sum_{i = 1}^m\Big\|\sum_{j = b(i - 1) + 1}^n\vy_*^j\va_j\Big\|_2^2 = \sum_{i = 0}^{m - 1}\Big\|\sum_{j = bi + 1}^n\vy_*^j\va_j\Big\|_2^2 = \sum_{i = 1}^{m - 1}\Big\|\sum_{j = bi + 1}^n\vy_*^j\va_j\Big\|_2^2 = \sum_{i = 1}^{m - 1}\Big\|\sum_{j = 1}^{bi}\vy_*^j\va_j\Big\|_2^2, 
\end{aligned}
\end{equation*}
by using the fact that $\sum_{j = 1}^n \vy_*^j\va_j = 0$. Using Young's inequality, we have 
\begin{equation*}
\begin{aligned}
\sum_{i = 1}^{m - 1}\Big\|\sum_{j = 1}^{bi}\vy_*^j\va_j\Big\|_2^2 \leq \;& \sum_{i = 1}^{m - 1}bi\sum_{j = 1}^{bi}\|\vy_*^j\va_j\|^2_2 \\
\leq \;& b(m - 1)\sum_{i = 1}^{m - 1}\sum_{j = 1}^{bi}\|\vy_*^j\va_j\|^2_2 \\
= \;& b(m - 1)\sum_{i = 1}^{m - 1}\sum_{j = b(i - 1) + 1}^{bi}(m - i)\|\vy_*^j\va_j\|^2_2 \\
\leq \;& b(m - 1)^2\sum_{i = 1}^{(m - 1)b}\|\vy_*^j\va_j\|^2_2.
\end{aligned}
\end{equation*}
By the definition that $\sigma_*^2 = \frac{1}{n}\sum_{j = 1}^n\|\vy_*^j \va_j\|^2_2$ and $\sum_{i = i}^{(m - 1)b}\|\vy_*^j\va_j\|^2_2 \leq \sum_{j = 1}^n\|\vy_*^j \va_j\|^2_2 = n\sigma_*^2$, we obtain 
\begin{equation}\label{eq:inn-prod-2-fin-bnd-2-3-IGD-2}
    \frac{\eta_k^3\Tilde{L}_0}{nb}\sum_{i = 1}^m\|\mA^{\top}\mI_{b(i - 1)\uparrow}\vy_{*}\|_2^2 \leq \frac{\eta_k^3\Tilde{L}_0}{b}b(m - 1)^2\sigma_*^2 = \frac{\eta_k^3 (n - b)^2}{b^2}\Tilde{L}_0\sigma_*^2.
\end{equation}
Note that the bound in Eq.~\eqref{eq:inn-prod-2-fin-bnd-2-3-IGD-2} equals to zero when $b = n$, which recovers the case of full gradient descent, so we have 
\begin{equation}
    \frac{\eta_k^3\Tilde{L}_0}{nb}\sum_{i = 1}^m\|\mA^{\top}\mI_{b(i - 1)\uparrow}\vy_{*}\|_2^2 \leq \min\Big\{\frac{\eta_k^3 n}{b^2}\Hat{L}_0\Tilde{L}_0\|\vy_{*}\|^2_{\mLambda^{-1}}, \frac{\eta_k^3 (n - b)^2}{b^2}\Tilde{L}_0\sigma_*^2\Big\}. \label{eq:sec-inn-IGD-3}
\end{equation}
Combining Eq.~\eqref{eq:sec-inn-IGD-1}--\eqref{eq:sec-inn-IGD-3}, we obtain 
\begin{equation*}
    \gT_2 \leq \frac{\eta_k^3 n}{b^2}\Hat{L}_0\Tilde{L}_0\|\vy_k - \vy_{*}\|^2_{\mLambda^{-1}} + \frac{\eta_k}{2n}\|\vv - \vy_k\|^2_{\mLambda^{-1}} + \min\Big\{\frac{\eta_k^3 n}{b^2}\Hat{L}_0\Tilde{L}_0\|\vy_{*}\|^2_{\mLambda^{-1}}, \frac{\eta_k^3 (n - b)^2}{b^2}\Tilde{L}_0\sigma_*^2\Big\}, 
\end{equation*}
thus finishing the proof.
\end{proof}

\convergenceIGD*
\begin{proof}
Proceeding as in Lemmas~\ref{lemma:gap-bound}~and~\ref{lemma:first-inner}, but without random permutations, we have 
\begin{align}
    \gE_k \leq \;& \frac{\eta_k}{n}\sum_{i = 1}^m\vy_k^{(i) \top}\mA^{(i)}(\vx_{k} - \vx_{k - 1, i + 1}) + \frac{\eta_k}{n}\sum_{i = 1}^m \big(\vv^{(i)} - \vy_k^{(i)}\big)^{\top}\mA^{(i)}(\vx_{k} - \vx_{k - 1, i}) \notag \\
    & - \frac{\eta_k}{2n}\|\vy_k - \vv\|_{\mLambda^{-1}}^2 - \frac{\eta_k}{2n}\|\vy_k - \vy_{*}\|_{\mLambda^{-1}}^2 - \frac{b}{2n}\sum_{i = 1}^m\|\vx_{k - 1, i} - \vx_{k - 1, i + 1}\|^2_2 \notag \\
    \leq \;& \frac{\eta_k}{n}\sum_{i = 1}^m \big(\vv^{(i)} - \vy_k^{(i)}\big)^{\top}\mA_k^{(i)}(\vx_{k} - \vx_{k - 1, i}) - \frac{\eta_k}{2n}\|\vy_k - \vv\|_{\mLambda^{-1}}^2 - \frac{\eta_k}{2n}\|\vy_k - \vy_{*}\|_{\mLambda^{-1}}^2. \label{eq:gap-bound-IGD-1}
\end{align}
Using the bound in Lemma~\ref{lemma:second-inner-IGD} and applying Eq.~\eqref{eq:sec-inn-bnd-IGD} into Eq.~\eqref{eq:gap-bound-IGD-1}, we obtain 
\begin{equation*}
    \gE_k \leq \Big(\frac{\eta_k^3 n \Hat{L}_0\Tilde{L}_0}{b^2} - \frac{\eta_k}{2n}\Big)\|\vy_k - \vy_{*}\|^2_{\mLambda^{-1}} + \min\Big\{\frac{\eta_k^3 n}{b^2}\Hat{L}_0\Tilde{L}_0\|\vy_{*}\|^2_{\mLambda^{-1}}, \frac{\eta_k^3 (n - b)^2}{b^2}\Tilde{L}_0\sigma_*^2\Big\}.
\end{equation*}
If $\eta_k \leq \frac{b}{n\sqrt{2\Hat{L}_0\Tilde{L}_0}}$, we have $\frac{\eta_k^3 n \Hat{L}_0\Tilde{L}_0}{b^2} - \frac{\eta_k}{2n} \leq 0$, thus 
\begin{equation*}
    \gE_k \leq \min\Big\{\frac{\eta_k^3 n}{b^2}\Hat{L}_0\Tilde{L}_0\|\vy_{*}\|^2_{\mLambda^{-1}}, \frac{\eta_k^3 (n - b)^2}{b^2}\Tilde{L}_0\sigma_*^2\Big\}. 
\end{equation*}
Noticing that $\gE_k = \eta_k\text{Gap}^\vv(\vx_{k}, \vy_*) + \frac{b}{2n}\|\vx_* - \vx_{k}\|_2^2 - \frac{b}{2n} \|\vx_* - \vx_{k - 1}\|_2^2$ and telescoping from $k = 1$ to $K$, we have 
\begin{equation*}
    \sum_{k = 1}^K\eta_k\text{Gap}^\vv(\vx_{k}, \vy_*) \leq \frac{b}{2n}\|\vx_* - \vx_0\|_2^2 - \frac{b}{2n}\|\vx_* - \vx_{K}\|_2^2 + \sum_{k = 1}^K\min\Big\{\frac{\eta_k^3 n}{b^2}\Hat{L}_0\Tilde{L}_0\|\vy_{*}\|^2_{\mLambda^{-1}}, \frac{\eta_k^3 (n - b)^2}{b^2}\Tilde{L}_0\sigma_*^2\Big\}.
\end{equation*}
Noticing that $\gL(\vx, \vv)$ is convex w.r.t.\ $\vx$, we have $\text{Gap}^\vv(\Hat\vx_{K}, \vy_*) \leq \sum_{k = 1}^K\eta_k\text{Gap}^\vv(\vx_{k}, \vy_*) / H_K$, where $\Hat{\vx}_K = \sum_{k=1}^K \eta_k \vx_{k} / H_K$ and $H_K = \sum_{k = 1}^K\eta_k$, so we obtain 
\begin{equation*}
    H_K\text{Gap}^\vv(\Hat\vx_{K}, \vy_*) \leq \frac{b}{2n}\|\vx_0 - \vx_*\|_2^2 + \sum_{k=1}^K \min\Big\{\frac{\eta_k^3 n}{b^2}\Hat{L}_0\Tilde{L}_0\|\vy_{*}\|^2_{\mLambda^{-1}}, \frac{\eta_k^3 (n - b)^2}{b^2}\Tilde{L}_0\sigma_*^2\Big\}, 
\end{equation*}
Further choosing $\vv = \vy_{\Hat{\vx}_K}$, we obtain 
\begin{equation}\label{eq:final-bound-IGD}
\begin{aligned}
    H_K\big(f(\Hat{\vx}_K) - f(\vx_*)\big) \leq \frac{b}{2n}\|\vx_0 - \vx_*\|_2^2 + \sum_{k=1}^K \min\Big\{\frac{\eta_k^3 n}{b^2}\Hat{L}_0\Tilde{L}_0\|\vy_{*}\|^2_{\mLambda^{-1}}, \frac{\eta_k^3 (n - b)^2}{b^2}\Tilde{L}_0\sigma_*^2\Big\}.
\end{aligned}
\end{equation}
To analyze the individual gradient oracle complexity, we choose constant stepsizes $\eta \leq \frac{b}{n\sqrt{2\Hat{L}_0\Tilde{L}_0}}$ and assume $b < n$ without loss of generality, then Eq.~\eqref{eq:final-bound-IGD} becomes 
\begin{equation*}
\begin{aligned}
    f(\Hat{\vx}_K) - f(\vx_*) \leq \;& \frac{b}{2n\eta K}\|\vx_0 - \vx_*\|_2^2 + \min\Big\{\frac{\eta^2 n}{b^2}\Hat{L}_0\Tilde{L}_0\|\vy_{*}\|^2_{\mLambda^{-1}}, \frac{\eta^2 (n - b)^2}{b^2}\Tilde{L}_0\sigma_*^2\Big\}.
\end{aligned}
\end{equation*}
When $\Hat{L}_0\|\vy_{*}\|^2_{\mLambda^{-1}} \leq \frac{(n - b)^2}{n}\sigma_*^2$, we set $\eta = \min\Big\{\frac{b}{n\sqrt{2\Hat{L}_0\Tilde{L}}_0},\, \Big(\frac{b^3\|\vx_0 - \vx_*\|^2_2}{2n^2\Hat{L}_0\Tilde{L}_0K\|\vy_*\|_{\mLambda^{-1}}^2}\Big)^{1/3}\Big\}$ and consider the following two possible cases: 
\begin{itemize}[leftmargin=*]
    \item ``Small $K$'' case: if $\eta = \frac{b}{n\sqrt{2\Hat{L}_0\Tilde{L}}_0} \leq \Big(\frac{b^3\|\vx_0 - \vx_*\|^2_2}{2n^2\Hat{L}_0\Tilde{L}_0K\|\vy_*\|_{\mLambda^{-1}}^2}\Big)^{1/3}$, we have 
    \begin{equation*}
    \begin{aligned}
        f(\Hat{\vx}_K) - f(\vx_*) \leq \;& \frac{b}{2n\eta K}\|\vx_0 - \vx_*\|_2^2 + \frac{\eta^2 n}{b^2}\Hat{L}_0\Tilde{L}_0\|\vy_{*}\|^2_{\mLambda^{-1}} \\
        \leq \;& \frac{\sqrt{\Hat{L}_0\Tilde{L}_0}}{\sqrt{2}K}\|\vx_0 - \vx_*\|_2^2 + \frac{\Hat{L}_0^{1/3}\Tilde{L}_0^{1/3}\|\vy_{*}\|_{\mLambda^{-1}}^{2/3}\|\vx_0 - \vx_*\|^{4/3}_2}{2^{2/3}n^{1/3}K^{2/3}}.
    \end{aligned}
    \end{equation*}
    \item ``Large $K$'' case: if $\eta = \Big(\frac{b^3\|\vx_0 - \vx_*\|^2_2}{2n^2\Hat{L}_0\Tilde{L}_0K\|\vy_*\|_{\mLambda^{-1}}^2}\Big)^{1/3} \leq \frac{b}{\sqrt{2\Hat{L}_0\Tilde{L}_0}}$, we have 
    \begin{equation*}
    \begin{aligned}
        f(\Hat{\vx}_K) - f(\vx_*) \leq \;& \frac{b}{2n\eta K}\|\vx_0 - \vx_*\|_2^2 + \frac{\eta^2 n}{b^2}\Hat{L}_0\Tilde{L}_0\|\vy_{*}\|^2_{\mLambda^{-1}} \\
        \leq \;& \frac{2^{1/3}\Hat{L}_0^{1/3}\Tilde{L}_0^{1/3}\|\vy_{*}\|_{\mLambda^{-1}}^{2/3}\|\vx_0 - \vx_*\|^{4/3}_2}{n^{1/3}K^{2/3}}.
    \end{aligned}
    \end{equation*}
\end{itemize}
Combining these two cases, we have 
\begin{equation*}
    f(\Hat{\vx}_K) - f(\vx_*) \leq \frac{\sqrt{\Hat{L}_0\Tilde{L}_0}}{\sqrt{2}K}\|\vx_0 - \vx_*\|_2^2 + \frac{2^{1/3}\Hat{L}_0^{1/3}\Tilde{L}_0^{1/3}\|\vy_{*}\|_{\mLambda^{-1}}^{2/3}\|\vx_0 - \vx_*\|^{4/3}_2}{n^{1/3}K^{2/3}}. 
\end{equation*}
Hence, to guarantee $\E[f(\Hat{\vx}_K) - f(\vx_*)] \leq \epsilon$ for $\epsilon > 0$, the total number of individual gradient evaluations will be 
\begin{equation}\label{eq:complexity-IGD-1}
    nK \geq \max\Big\{\frac{n\sqrt{2\Hat{L}_0\Tilde{L}_0}\|\vx_0 - \vx_*\|_2^2}{\epsilon}, \frac{4n^{1/2}\Hat{L}_0^{1/2}\Tilde{L}_0^{1/2}\|\vy_{*}\|_{\mLambda^{-1}}\|\vx_0 - \vx_*\|_2^2}{\epsilon^{3/2}}\Big\}. 
\end{equation}
When $\frac{(n - b)^2}{n}\sigma_*^2 \leq \Hat{L}_0\|\vy_*\|^2_{\mLambda^{-1}}$, we set $\eta = \min\Big\{\frac{b}{n\sqrt{2\Hat{L}_0\Tilde{L}}_0},\, \Big(\frac{b^3\|\vx_0 - \vx_*\|^2_2}{2n(n - b)^2\Tilde{L}_0K\sigma_*^2}\Big)^{1/3}\Big\}$ and consider the two cases as below: 
\begin{itemize}
    \item ``Small $K$'' case: if $\eta = \frac{b}{n\sqrt{2\Hat{L}_0\Tilde{L}}_0} \leq \Big(\frac{b^3\|\vx_0 - \vx_*\|^2_2}{2n(n - b)^2\Tilde{L}_0K\sigma_*^2}\Big)^{1/3}$, we have 
    \begin{equation*}
    \begin{aligned}
        f(\Hat{\vx}_K) - f(\vx_*) \leq \;& \frac{b}{2n\eta K}\|\vx_0 - \vx_*\|_2^2 +  \frac{\eta^2 (n - b)^2}{b^2}\Tilde{L}_0\sigma_*^2 \\
        \leq \;& \frac{\sqrt{\Hat{L}_0\Tilde{L}_0}}{\sqrt{2}K}\|\vx_0 - \vx_*\|_2^2 + \frac{(n - b)^{2/3}\Tilde{L}_0^{1/3}\sigma_*^{2/3}\|\vx_0 - \vx_*\|_2^{4/3}}{2^{2/3}n^{2/3}K^{2/3}}.
    \end{aligned}
    \end{equation*}
    \item ``Large $K$'' case: if $\eta = \Big(\frac{b^3\|\vx_0 - \vx_*\|^2_2}{2n(n - b)^2\Tilde{L}_0K\sigma_*^2}\Big)^{1/3} \leq \frac{b}{n\sqrt{2\Hat{L}_0\Tilde{L}}_0}$, we have 
    \begin{equation*}
    \begin{aligned}
        f(\Hat{\vx}_K) - f(\vx_*) \leq \;& \frac{b}{2n\eta K}\|\vx_0 - \vx_*\|_2^2 +  \frac{\eta^2 (n - b)^2}{b^2}\Tilde{L}_0\sigma_*^2 \\
        \leq \;& \frac{2^{1/3}(n - b)^{2/3}\Tilde{L}_0^{1/3}\sigma_*^{2/3}\|\vx_0 - \vx_*\|_2^{4/3}}{n^{2/3}K^{2/3}}.
    \end{aligned}
    \end{equation*}
\end{itemize}
Combining these two cases, we obtain 
\begin{equation*}
    f(\Hat{\vx}_K) - f(\vx_*) \leq \frac{\sqrt{\Hat{L}_0\Tilde{L}_0}}{\sqrt{2}K}\|\vx_0 - \vx_*\|_2^2 + \frac{2^{1/3}(n - b)^{2/3}\Tilde{L}_0^{1/3}\sigma_*^{2/3}\|\vx_0 - \vx_*\|_2^{4/3}}{n^{2/3}K^{2/3}}.
\end{equation*}
To guarantee $\E[f(\Hat{\vx}_K) - f(\vx_*)] \leq \epsilon$ for $\epsilon > 0$, the total number of individual gradient evaluations will be 
\begin{equation}\label{eq:complexity-IGD-2}
    nK \geq \max\Big\{\frac{n\sqrt{2\Hat{L}_0\Tilde{L}_0}\|\vx_0 - \vx_*\|_2^2}{\epsilon}, \frac{4(n - b)\Tilde{L}_0^{1/2}\sigma_*\|\vx_0 - \vx_*\|_2^2}{\epsilon^{3/2}}\Big\}.
\end{equation}
Combining Eq.~\eqref{eq:complexity-IGD-1} and Eq.~\eqref{eq:complexity-IGD-2}, we finally have 
\begin{equation*}
    nK \geq \frac{n\sqrt{2\Hat{L}_0\Tilde{L}_0}\|\vx_0 - \vx_*\|_2^2}{\epsilon} + \min\Big\{\frac{4n^{1/2}\Hat{L}_0^{1/2}\Tilde{L}_0^{1/2}\|\vy_{*}\|_{\mLambda^{-1}}\|\vx_0 - \vx_*\|_2^2}{\epsilon^{3/2}}, \frac{4(n - b)\Tilde{L}_0^{1/2}\sigma_*\|\vx_0 - \vx_*\|_2^2}{\epsilon^{3/2}}\Big\}, 
\end{equation*}
thus finishing the proof.
\end{proof}

\section{Omitted Proofs from Section~\ref{sec:nonsmooth}}\label{appx:nonsmooth}
Before we prove Theorem~\ref{thm:convergence-nonsmooth} in convex Lipschitz settings, for completeness, we first recall the following standard first-order characterization of convexity.
\begin{lemma}\label{lemma:convex}
Let $f: \R^d \rightarrow \R$ be a continuous convex function. Then, for any $\vx, \vy \in \R^d$: 
\begin{equation*}
    f(\vy) \geq f(\vx) + \innp{\vg_\vx, \vy - \vx}, 
\end{equation*}
where $\vg_\vx \in \partial f(\vx)$, and $\partial f(\vx)$ is the subdifferential of $f$ at $\vx$.
\end{lemma}
The following technical lemma provides a primal-dual gap bound in convex nonsmooth settings.
\begin{lemma}
\label{lemma:gap-bound-nonsmooth}
Under Assumption~\ref{assp:convex}, for any $k \in [K]$, the iterates $\{\vy_k^{(i)}\}_{i = 1}^m$ and $\{\vx_{k - 1, i}\}_{i = 1}^{m + 1}$ generated by Algorithm~\ref{alg:PD-shuffled-sgd} satisfy 
\begin{equation}\label{eq:gap-bound-nonsmooth}
\begin{aligned}
    \gE_k \leq \;& \frac{\eta_k}{n}\sum_{i=1}^m \Big(\vy_k^{(i) \top}\mA_k^{(i)}(\vx_{k} - \vx_{k - 1, i + 1}) + (\vv^{(i)}_k - \vy_k^{(i)})^{\top}\mA_k^{(i)}(\vx_{k} - \vx_{k - 1, i})\Big) \\
    & - \frac{b}{2n}\sum_{i=1}^m \|\vx_{k - 1, i} - \vx_{k - 1, i+1}\|_2^2, 
\end{aligned}
\end{equation}
where $\gE_k := \eta_k \big(\gL(\vx_{k}, \vv) - \gL(\vx_*, \vy_*)\big) + \frac{b}{2n}\|\vx_* - \vx_{k}\|_2^2 - \frac{b}{2n} \|\vx_* - \vx_{k - 1}\|_2^2$. 
\end{lemma}
\begin{proof}
By the same argument as in the proof for Lemma~\ref{lemma:gap-bound}, we know that $\va_{\pi_{j}^{(k)}}^\top\vx_{k - 1, i} \in \partial \ell_{\pi_{j}^{(k)}}^*(\vy_k^j)$ for $b(i - 1) + 1 \leq j \leq bi$, then by Lemma~\ref{lemma:convex} we have 
\begin{equation*}
    \ell_{\pi_{j}^{(k)}}^*\big(\vv_k^j\big) \geq \ell_{\pi_{j}^{(k)}}^*(\vy_k^j) + \va_{\pi_{j}^{(k)}}^\top\vx_{k - 1, i}\big(\vv_k^j - \vy_k^j\big),
\end{equation*}
which leads to 
\begin{align}
    \gL(\vx_{k}, \vv) = \;& \frac{1}{n}\sum_{i=1}^m \Big(\vv^{(i) \top}_k\mA_k^{(i)}\vx_{k - 1, i} - \sum_{j = b(i - 1) + 1}^{bi}\ell_{\pi^{(k)}_{j}}^*(\vv^{j}_k)\Big) + \frac{1}{n}\sum_{i=1}^m \vv^{(i) \top}_k\mA_k^{(i)}(\vx_{k} - \vx_{k - 1, i}) \notag\\
    \leq \;& \frac{1}{n}\sum_{i=1}^m \Big(\vy^{(i) \top}_k\mA_k^{(i)}\vx_{k - 1, i} - \sum_{j = b(i - 1) + 1}^{bi}\ell_{\pi^{(k)}_{j}}^*(\vy^{j}_k)\Big) + \frac{1}{n}\sum_{i=1}^m \vv^{(i) \top}_k\mA_k^{(i)}(\vx_{k} - \vx_{k - 1, i}). \label{eq:gap-bound-1-IGD}
\end{align}
Using the same argument for $\gL(\vx_*, \vy_*)$ as $\va_j^\top\vx_* \in \partial \ell_j^*(\vy_*^j)$ for $j \in [n]$, we have 
\begin{align}
    \gL(\vx_*, \vy_*) 
    = \;& \frac{1}{n}\sum_{i = 1}^m \Big(\vy_{*, k}^{(i) \top}\mA_k^{(i)}\vx_* - \sum_{j = b(i - 1) + 1}^{bi}\ell^*_{\pi_j^{(k)}}(\vy_{*, k}^j)\Big) \notag\\
    \geq \;& \frac{1}{n}\sum_{i = 1}^m\Big(\vy_k^{(i) \top}\mA_k^{(i)}\vx_* - \sum_{j = b(i - 1) + 1}^{bi}\ell^*_{\pi_j^{(k)}}(\vy_{k}^j)\Big). \label{eq:gap-bound-2-IGD}
\end{align}
Adding and substracting the term $\frac{b}{2n\eta_k}\sum_{i = 1}^m\|\vx_* - \vx_{k - 1, i}\|_2^2$ on the R.H.S.\ of Eq.~\eqref{eq:gap-bound-2-IGD}, we obtain
\begin{equation*}
\begin{aligned}
    \gL(\vx_*, \vy_*) \geq \;& \frac{1}{n}\sum_{i = 1}^m\Big(\vy_k^{(i) \top}\mA_k^{(i)}\vx_* + \frac{b}{2\eta_k}\|\vx_* - \vx_{k - 1, i}\|_2^2 - \sum_{j = b(i - 1) + 1}^{bi}\ell^*_{\pi_j^{(k)}}(\vy_{k}^j)\Big) \\
    & - \frac{b}{2n\eta_k}\sum_{i = 1}^m\|\vx_* - \vx_{k - 1, i}\|_2^2 .
\end{aligned}
\end{equation*}
Denote $\phi_k^{(i)}(\vx) := \vy_k^{(i)\top}\mA_k^{(i)}\vx + \frac{b}{2\eta_k}\|\vx - \vx_{k - 1, i}\|_2^2$, which is $\frac{b}{\eta_k}$-strongly convex w.r.t.\ $\vx$. 
Noticing that $\vx_{k - 1, i + 1} = \argmin_{\vx \in \R^d}\Big\{\vy_k^{(i)\top}\mA_k^{(i)}\vx + \frac{b}{2\eta_k}\|\vx - \vx_{k - 1, i}\|^2\Big\}$ by Line 7
of Alg.~\ref{alg:PD-shuffled-sgd},  we have $\nabla \phi_k^{(i)}(\vx_{k - 1, i + 1}) = \vzero$, which leads to 
\begin{equation*}
    \phi_k^{(i)}(\vx_*) \geq \phi_k^{(i)}(\vx_{k - 1, i + 1}) + \frac{b}{2\eta_k}\|\vx_* - \vx_{k - 1, i + 1}\|^2_2.
\end{equation*}
Thus, we obtain 
\begin{align}
    \gL(\vx_*, \vy_*) \geq \;& \frac{1}{n}\sum_{i = 1}^m\Big(\vy_k^{(i) \top}\mA_k^{(i)}\vx_{k - 1, i + 1} + \frac{b}{2\eta_k}\|\vx_{k - 1, i + 1} - \vx_{k - 1, i}\|^2_2 - \sum_{j = b(i - 1) + 1}^{bi}\ell^*_{\pi_j^{(k)}}(\vy_{k}^j)\Big) \notag \\
    & + \frac{b}{2n\eta_k}\sum_{i = 1}^m\big(\|\vx_* - \vx_{k - 1, i + 1}\|^2_2 - \|\vx_* - \vx_{k - 1, i}\|_2^2\big) \notag \\
    \overset{(\romannumeral1)}{=} \;& \frac{1}{n}\sum_{i = 1}^m\Big(\vy_k^{(i) \top}\mA_k^{(i)}\vx_{k - 1, i + 1} + \frac{b}{2\eta_k}\|\vx_{k - 1, i + 1} - \vx_{k - 1, i}\|^2_2 - \sum_{j = b(i - 1) + 1}^{bi}\ell^*_{\pi_j^{(k)}}(\vy_{k}^j)\Big) \notag \\
    & + \frac{b}{2n\eta_k}\|\vx_{k} - \vx_*\|_2^2 - \frac{b}{2n\eta_k}\|\vx_{k - 1} - \vx_*\|_2^2, \label{eq:gap-bound-3-IGD}
\end{align}
where $(\romannumeral1)$ is by telescoping  $\sum_{i = 1}^m\big(\|\vx_* - \vx_{k - 1, i + 1}\|^2_2 - \|\vx_* - \vx_{k - 1, i}\|_2^2\big)$ and using $\vx_k = \vx_{k - 1, m + 1}$ and $\vx_{k - 1} = \vx_{k - 1, 1}$, which both hold by definition.

Combining the bounds from Eq.~\eqref{eq:gap-bound-1-IGD} and Eq.~\eqref{eq:gap-bound-3-IGD}, and denoting 
\begin{equation*}
    \gE_k := \eta_k \big(\gL(\vx_{k}, \vv) - \gL(\vx_*, \vy_*)\big) + \frac{b}{2n}\|\vx_* - \vx_{k}\|_2^2 - \frac{b}{2n} \|\vx_* - \vx_{k - 1}\|_2^2, 
\end{equation*}
we finally obtain 
\begin{equation*}
\begin{aligned}
    \gE_k 
    \leq \;& \frac{\eta_k}{n}\sum_{i=1}^m \vy_k^{(i) \top}\mA_k^{(i)}(\vx_{k - 1, i} - \vx_{k - 1, i + 1}) + \frac{\eta_k}{n}\sum_{i=1}^m \vv^{(i) \top}_k\mA_k^{(i)}(\vx_{k} - \vx_{k - 1, i}) \\
    & - \frac{b}{2n}\sum_{i=1}^m \|\vx_{k - 1, i} - \vx_{k - 1, i+1}\|_2^2 \\
    = \;& \frac{\eta_k}{n}\sum_{i=1}^m \vy_k^{(i) \top}\mA_k^{(i)}(\vx_{k} - \vx_{k - 1, i + 1}) + \frac{\eta_k}{n}\sum_{i=1}^m (\vv^{(i)}_k - \vy_k^{(i)})^{\top}\mA_k^{(i)}(\vx_{k} - \vx_{k - 1, i}) \\
    & - \frac{b}{2n}\sum_{i=1}^m \|\vx_{k - 1, i} - \vx_{k - 1, i+1}\|_2^2, 
\end{aligned}
\end{equation*}
thus completing the proof. 
\end{proof}
Note that we can still use Lemma~\ref{lemma:first-inner} to bound the first inner product term in Eq.~\eqref{eq:gap-bound-nonsmooth}, as we are studying the same algorithm. The following lemma provides a bound on the second inner product term $\gT_2 := \frac{\eta_k}{n}\sum_{i = 1}^m \big(\vv^{(i)}_k - \vy_k^{(i)}\big)^{\top}\mA_k^{(i)}(\vx_{k} - \vx_{k - 1, i})$ in Eq.~\eqref{eq:gap-bound-nonsmooth}. 
\begin{restatable}{lemma}{secondInnerNonsmooth}
\label{lemma:second-inner-nonsmooth}
Under Assumption~\ref{assp:Lipschitz}, for any $k \in [K]$, the iterates $\{\vy_k^{(i)}\}_{i = 1}^m$ and $\{\vx_{k - 1, i}\}_{i = 1}^{m + 1}$ generated by Algorithm~\ref{alg:PD-shuffled-sgd} satisfy 
\begin{equation}\label{eq:second-inner-nonsmooth}
\begin{aligned}
    \gT_2 \leq \;& \frac{\eta_k^2\sqrt{\Hat{G}_{\pi^{(k)}}\Tilde{G}_{\pi^{(k)}}}}{b}\|\vy_k\|^2_{\mGamma_k^{-1}} + \frac{\eta_k^2\sqrt{\Hat{G}_{\pi^{(k)}}\Tilde{G}_{\pi^{(k)}}}}{4b}\|\vv_k - \vy_k\|^2_{\mGamma_k^{-1}}.
\end{aligned}
\end{equation}
\end{restatable}
\begin{proof}
Proceeding as in Lemma~\ref{lemma:second-inner}, we have 
\begin{equation*}
\begin{aligned}
    \gT_2 := \frac{\eta_k}{n}\sum_{i = 1}^m \big(\vv^{(i)}_k - \vy_k^{(i)}\big)^{\top}\mA_k^{(i)}(\vx_{k} - \vx_{k - 1, i}) = -\frac{\eta_k^2}{bn}\sum_{i = 1}^m\innp{\mA_k^{\top}\mI_{b(i - 1)\uparrow}\vy_k, \mA_k^{\top}\mI_{(i)}(\vv_k - \vy_k)}.
\end{aligned}
\end{equation*}
Using Young's inequality for some $\alpha > 0$ and proceeding as in Eq.~\eqref{eq:cyclic}, we obtain 
\begin{equation*}
\begin{aligned}
    \gT_2 \leq \;& \frac{\eta_k^2\alpha}{2bn}\sum_{i = 1}^m\|\mA_k^{\top}\mI_{b(i - 1)\uparrow}\vy_k\|_2^2 + \frac{\eta_k^2}{2bn\alpha}\sum_{i = 1}^m\|\mA_k^{\top}\mI_{(i)}(\vv_k - \vy_k)\|_2^2 \\
    \leq \;& \frac{\eta_k^2 n\alpha}{2b^2}\Hat{G}_{\pi^{(k)}}\|\vy_k\|^2_{\mGamma_k^{-1}} + \frac{\eta_k^2}{2n\alpha}\Tilde{G}_{\pi^{(k)}}\|\vv_k - \vy_k\|^2_{\mGamma_k^{-1}}, 
\end{aligned}
\end{equation*}
where we use our definitions that $\Hat{G}_{\pi^{(k)}} := \frac{1}{mn}\big\|\mGamma_k^{1/2}\big(\textstyle\sum_{j=1}^m\mI_{b(j - 1)\uparrow} \mA_k\mA_k^{\top}\mI_{b(j - 1)\uparrow}\big)\mGamma_k^{1/2}\big\|_2$ and $\Tilde{G}_{\pi^{(k)}} := \frac{1}{b}\big\|\mGamma_k^{1/2}\big(\textstyle \sum_{j=1}^m\mI_{(j)} \mA_k\mA_k^{\top}\mI_{(j)}\big)\mGamma_k^{1/2}\big\|_2$.
It remains to choose $\alpha = \frac{2b}{n}\sqrt{\frac{\Tilde{G}_k}{\Hat{G}_k}}$ to finish the proof.
\end{proof}
We are now ready to prove Theorem~\ref{thm:convergence-nonsmooth} for the convergence of shuffled SGD in the convex nonsmooth Lipschitz settings.
\convergenceNonsmooth*
\begin{proof}
To simplify the presentation of our analysis, we first assume $\|\vv\|^2_{\mGamma^{-1}} \leq n$, which will be later verified by our choice of $\vv = \vy_{\Hat{\vx}_K}$ and Assumption~\ref{assp:Lipschitz}.

Combining the bounds in Lemma~\ref{lemma:first-inner}~and~\ref{lemma:second-inner-nonsmooth} and plugging them into Eq.~\eqref{eq:gap-bound-nonsmooth}, we have 
\begin{align}
    \gE_k \leq \;& \frac{\eta_k^2\sqrt{\Hat{G}_{\pi^{(k)}}\Tilde{G}_{\pi^{(k)}}}}{b}\|\vy_k\|^2_{\mGamma_k^{-1}} + \frac{\eta_k^2\sqrt{\Hat{G}_{\pi^{(k)}}\Tilde{G}_{\pi^{(k)}}}}{4b}\|\vv_k - \vy_k\|^2_{\mGamma_k^{-1}} \notag \\
    \overset{(\romannumeral1)}{\leq} \;& \frac{\eta_k^2\sqrt{\Hat{G}_{\pi^{(k)}}\Tilde{G}_{\pi^{(k)}}}}{b}\|\vy_k\|^2_{\mGamma_k^{-1}} + \frac{\eta_k^2\sqrt{\Hat{G}_{\pi^{(k)}}\Tilde{G}_{\pi^{(k)}}}}{2b}(\|\vv\|^2_{\mGamma^{-1}} + \|\vy_k\|^2_{\mGamma_k^{-1}}) \notag \\
    \overset{(\romannumeral2)}{\leq} \;& \frac{2\eta_k^2 n \sqrt{\Hat{G}_{\pi^{(k)}}\Tilde{G}_{\pi^{(k)}}}}{b}, \label{eq:fin-bnd-nonsmooth}
\end{align}
where we use Young's inequality for $\|\vv_k - \vy_k\|^2_{\mGamma_k^{-1}}$ and $\|\vv_k\|_{\mGamma_k^{-1}} = \|\vv\|^2_{\mGamma^{-1}}$ as $\vv$ is a fixed vector for $(\romannumeral1)$, and $(\romannumeral2)$ is due to $\|\vy_k\|^2_{\mGamma_k^{-1}} \leq n$ by Assumption~\ref{assp:Lipschitz} and assuming that $\|\vv\|^2_{\mGamma^{-1}} \leq n$. Proceeding as the proof for Theorem~\ref{thm:convergence}, we first assume the RR scheme and take conditional expectation w.r.t.\ the randomness up to but not including $k$-th epoch, then we obtain 
\begin{equation*}
    \E_k[\gE_k] \leq \frac{2\eta_k^2 n \E_k\big[\sqrt{\Hat{G}_{\pi^{(k)}}\Tilde{G}_{\pi^{(k)}}}\big]}{b}.
\end{equation*}
Since the randomness only comes from the random permutation $\pi^{(k)}$, we have
\begin{equation*}
    \E_k[\gE_k] \leq \frac{2\eta_k^2 n \E_{\pi}[\sqrt{\Hat{G}_\pi\Tilde{G}_\pi}]}{b}.
\end{equation*}
For notational convenience, we denote $\Bar{G} = \E_\pi[\sqrt{\Hat{G}_\pi\Tilde{G}_\pi}]$, and further take expectation w.r.t.\ all the randomness on both sides and use the law of total expectation to obtain 
\begin{equation}\label{eq:error-bnd-nonsmooth}
    \E[\gE_k] \leq \frac{2\eta_k^2 n \Bar{G}}{b}.
\end{equation}
For the SO scheme, there is one random permutation $\pi$ generated at the very beginning such that $\pi^{(k)} = \pi$ for all $k \in [K]$. So we can directly take expectation w.r.t.\ all the randomness on both sides of Eq.~\eqref{eq:fin-bnd-nonsmooth}, with the randomness only from $\pi$, which leads to the same bound as Eq.~\eqref{eq:error-bnd-nonsmooth} with $\E\big[\sqrt{\Hat{G}_{\pi^{(k)}}\Tilde{G}_{\pi^{(k)}}}\big] = \E_\pi\big[\sqrt{\Hat{G}_{\pi}\Tilde{G}_{\pi}}\big]$. Note that for incremental gradient (IG) descent, we can let $\Bar{G} = \sqrt{\Hat{G}_0\Tilde{G}_0}$ without randomness involved, where $\Hat{G}_0 = \hat{G}_{\pi^{(0)}}$ and $\Tilde{G}_0 = \Tilde{G}_{\pi^{(0)}}$ w.r.t.\ the initial, fixed permutation $\pi^{(0)}$ of the data matrix $\mA$.

Noticing that $\gE_k = \eta_k\text{Gap}^\vv(\vx_{k}, \vy_*) + \frac{b}{2n}\|\vx_* - \vx_{k}\|_2^2 - \frac{b}{2n} \|\vx_* - \vx_{k - 1}\|_2^2$ and telescoping from $k = 1$ to $K$, we have 
\begin{equation*}
\begin{aligned}
    \E\Big[\sum_{k = 1}^K\eta_k\text{Gap}^\vv(\vx_{k}, \vy_*)\Big] \leq \;& \frac{b}{2n}\|\vx_* - \vx_0\|_2^2 - \frac{b}{2n}\E[\|\vx_* - \vx_{K}\|_2^2] + \sum_{k = 1}^K\frac{2\eta_k^2 n \Bar{G}}{b}.
\end{aligned}
\end{equation*}
Noticing that $\gL(\vx, \vv)$ is convex wrt $\vx$, we have $\text{Gap}^\vv(\Hat\vx_{K}, \vy_*) \leq \sum_{k = 1}^K\eta_k\text{Gap}^\vv(\vx_{k}, \vy_*) / H_K$, where $\Hat{\vx}_K = \sum_{k=1}^K \eta_k \vx_{k} / H_K$ and $H_K = \sum_{k = 1}^K\eta_k$, so we obtain 
\begin{equation*}
    \E\Big[H_K\text{Gap}^\vv(\Hat\vx_{K}, \vy_*)\Big] \leq \frac{b}{2n}\|\vx_0 - \vx_*\|_2^2 + \sum_{k = 1}^K\frac{2\eta_k^2 n \Bar{G}}{b}.
\end{equation*}
Further choosing $\vv = \vy_{\Hat{\vx}_K}$, which also verifies $\|\vv\|^2_{\mGamma^{-1}} = \|\vy_{\Hat{\vx}_K}\|^2_{\mGamma^{-1}} \leq n$ by Assumption~\ref{assp:Lipschitz}, we obtain 
\begin{equation*}
    \E[H_K(f(\Hat{\vx}_K) - f(\vx_*))] \leq \frac{b}{2n}\|\vx_0 - \vx_*\|_2^2 + \sum_{k = 1}^K\frac{2\eta_k^2 n \Bar{G}}{b}. 
\end{equation*}
To analyze the individual gradient oracle complexity, we choose constant stepsize $\eta$. Then,  the above bound  becomes  
\begin{equation*}
    \E[f(\Hat{\vx}_K) - f(\vx_*)] \leq \frac{b}{2n\eta K}\|\vx_0 - \vx_*\|_2^2 + \frac{2 n \eta\Bar{G}}{b}. 
\end{equation*}
Choosing $\eta = \frac{b\|\vx_0 - \vx_*\|_2}{2n\sqrt{K\Bar{G}}}$, we have 
\begin{equation*}
    \E[f(\Hat{\vx}_K) - f(\vx_*)] \leq \frac{2\sqrt{\Bar{G}}\|\vx_0 - \vx_*\|_2}{\sqrt{K}}. 
\end{equation*}
Hence, given $\epsilon > 0$, to ensure $\E[f(\Hat{\vx}_K) - f(\vx_*)] \leq \epsilon$, the total number of individual gradient evaluations will be 
\begin{equation*}
    nK \geq \frac{4n\Bar{G}\|\vx_0 - \vx_*\|_2^2}{\epsilon^2}, 
\end{equation*}
thus completing the proof. 
\end{proof}

\section{General Smooth Convex Finite-Sum Problems}\label{appx:general}
In this section, we consider the more general setting beyond generalized linear models where we only assume that each component function $f_i$ is convex and smooth, and derive the refined analysis under this setting. {Here we focus on the smooth convex problems as prior work did~\citep{mishchenko2020random,nguyen2021unified}, since smoothness is essential to showing the advantage of shuffled SGD~\citep{nagaraj2019sgd} over SGD, otherwise the rate of SGD is optimal.} In particular, we study the general smooth convex finite-sum problem~\eqref{eq:general-prob}
\begin{align}
    \min_{\vx \in \R^d} \Big\{f(\vx) := \frac{1}{n}\sum_{i = 1}^n f_i(\vx)\Big\}, \tag{\ref{eq:general-prob}}
\end{align}
where each $f_i$ is convex and smooth.\ \eqref{eq:general-prob} is equivalent to 
\begin{align}
    \min_{\vx \in \R^d}\max_{\vy \in \R^{nd}}\Big\{\gL(\vx, \vy) := \frac{1}{n}\sum_{i = 1}^n\Big(\innp{\vy^{i}, \vx} - f_i^*(\vy^{i})\Big) = \frac{1}{n}\innp{\mE \vx, \vy} - \frac{1}{n}\sum_{i = 1}^n f_i^*(\vy^{i})\Big\}, \tag{P-General}
\end{align}
where we slightly abuse the notation in this section and use $\vy^{i} \in \R^d$ to be the $i$-th $d$ elements of the vector $\vy$ such that $\vy = (\vy^1, \dots, \vy^n)^\top \in \R^{nd}$, $\mE = [\underbrace{\mI_d, \dots, \mI_d}_n]^\top \in \R^{nd \times d}$ is the vertical concatenation of $n$ identity matrices $\mI_d \in \R^{d \times d}$ and $f_i^*$ is the convex conjugate of $f_i$ defined by $f_i(\vx) = \sup_{\vy^{i} \in \R^d}\innp{\vy^{i}, \vx} - f_i^*(\vy^{i})$.

Throughout this section, we use the following assumptions. 
\begin{assumption}\label{assp:fi}
Each individual loss $f_i$ is convex and $L_i$-smooth, and there exists a minimizer $\vx_* \in \R^d$ for $f(\vx)$.  
\end{assumption}
This assumption implies that $f$ and all component functions $f_i$ are $L$-smooth, where $L := \max_{i \in [n]}L_i$. Assumption~\ref{assp:fi} also implies that each convex conjugate $f_i^*$ is $\frac{1}{L_i}$-strongly convex~\citep{beck2017first}. Also, we let $\mLambda = \mathrm{diag}(\underbrace{L_1, \dots, L_1}_{d}, \dots, \underbrace{L_n, \dots, L_n}_{d}) \in \R^{nd \times nd}$. We further assume the bounded variance at $\vx_*$, same as prior work~\citep{mishchenko2020random, nguyen2021unified}.
\begin{assumption}\label{assp:vr-fi}
The quantity $\sigma_*^2 = \frac{1}{n}\sum_{i = 1}^n\|\nabla f_i(\vx_*)\|^2$ is bounded.
\end{assumption}
Given a primal-dual pair $(\vx, \vy) \in \R^d \times \R^{nd}$, we define the primal-dual gap as in Section~\ref{sec:prelim}:
\begin{align*}
    \mathrm{Gap}(\vx, \vy) = \max_{(\vu, \vv) \in \R^d \times \R^{nd}}\big\{\gL(\vx, \vv) - \gL(\vu, \vy)\big\}. 
\end{align*}
In particular, we consider the pair $(\vx, \vy_*)$ and focus on bounding 
\begin{align*}
    \mathrm{Gap}^{\vv}(\vx, \vy_*) = \gL(\vx, \vv) - \gL(\vx_*, \vy_*). 
\end{align*}
To reduce this quantity to the function value gap $f(\vx) - f(\vx_*)$, it suffices to take $\vv = \vy_\vx$. 

In the following, we consider the mini-bath estimator of batch size $b$, and let $\vy^{(i)} \in \R^{bd}$ denote the vector comprised of the $i$\textsuperscript{th} $bd$ elements of $\vy$. For simplicity and without loss of generality, we assume that $n = bm$ for some positive integer $m$, so that $\vy = (\vy^{(1)}, \dots, \vy^{(m)})^\top$. Note that if choosing $b = 1$, our setting is the same as the ones in~\citet{mishchenko2020random, nguyen2021unified}. 
Then we have the primal-dual view of shuffled SGD scheme for general smooth convex minimization as in Alg.~\ref{alg:PD-shuffled-sgd-general}, 
\begin{algorithm}[t!]
\caption{Shuffled SGD (Primal-Dual View, General Convex Smooth)}\label{alg:PD-shuffled-sgd-general}
\begin{algorithmic}[1]
    \STATE \textbf{Input:} Initial point $\vx_0 \in \R^d,$ batch size $b > 0,$ step size $\{\eta_k\} > 0,$ number of epochs $K > 0$
    \FOR {$k = 1$ to $K$}
        \STATE Generate any permutation $\pi^{(k)}$ of $[n]$ (either deterministic or random)
        \STATE $\vx_{k - 1, 1} = \vx_{k - 1}$
        \FOR {$i = 1$ to $m$}
            \STATE $\vy_k^{(i)} = \argmax_{\vy^{(i)} \in \R^{bd}} \Big\{\innp{\mE^\top_b\vy^{(i)}, \vx_{k - 1, i}} - \sum_{j = 1}^b f_{\pi_{b(i - 1) + j}^{(k)}}^{*}(\vy^{j})\Big\}$ \hfill \label{algo-line:dual-update-general} \;
            \STATE $\vx_{k - 1, i + 1} = \argmin_{\vx \in \R^d}\Big\{\innp{\mE^\top_b\vy_k^{(i)}, \vx} + \frac{b}{2\eta_k}\|\vx - \vx_{k - 1, i}\|^2\Big\}$
        \ENDFOR
        \STATE $\vx_{k} = \vx_{k - 1, m + 1}$, $\vy_k = \big(\vy_k^{(1)}, \vy_k^{(2)}, \dots, \vy_k^{(m)}\big)^{\top}$\;
    \ENDFOR
    \STATE \textbf{Return:} $\Hat{\vx}_K = \sum_{k=1}^K \eta_k \vx_{k} / \sum_{k = 1}^K \eta_k$
\end{algorithmic}
\end{algorithm}
where $\mE^\top_b = [\underbrace{\mI_{d}, \dots, \mI_{d}}_b]^\top \in \R^{bd \times d}$ is the vertical concatenation of $b$ identity matrices $\mI_{d} \in \R^{d \times d}$. Given the data permutation $\pi^{(k)} = \{\pi_1^{(k)}, \pi_2^{(k)}, \dots, \pi_n^{(k)}\}$ of $[n]$ at the $k$-th epoch, we use the same notation of $\vv_k = (\vv^{\pi_1^{(k)}}, \dots, \vv^{\pi_n^{(k)}})^\top \in \R^{nd}$, $\vy_{*, k} = (\vy_*^{\pi_1^{(k)}}, \dots, \vy_*^{\pi_n^{(k)}})^\top \in \R^{nd}$ as in previous sections except now each $\vv^{\pi_i^{(k)}}, \vy_*^{\pi_i^{(k)}}$ are $d$-dimensional subvectors. Further, we denote the permuted smoothness constant matrices by $\mLambda_k = \mathrm{diag}(\underbrace{L_{\pi_1^{(k)}}, \dots, L_{\pi_1^{(k)}}}_{d}, \dots, \underbrace{L_{\pi_n^{(k)}}, \dots, L_{\pi_n^{(k)}}}_{d}) \in \R^{nd \times nd}$, and we use $\mI$ for $\mI_{nd} \in \R^{nd \times nd}$ throughout this section.

\paragraph{New smoothness constants and comparisons.} We define the new smoothness constants for any permutation $\pi$ of $[n]$, similar to Eq.~\eqref{eq:new-smoothness-constants}, but now for general smooth convex minimization.
\begin{equation}\label{eq:new-smoothness-constants-general}
\begin{aligned}
    \Hat{L}_\pi^{g} &:= \frac{1}{mn}\big\|\mLambda^{1/2}_\pi\big(\textstyle\sum_{i=1}^m\mI_{bd(i - 1)\uparrow} \mE\mE^{\top}\mI_{bd(i - 1)\uparrow}\big)\mLambda_\pi^{1/2}\big\|_2,
    \quad \;& \Hat{L}^g = \max_{\pi}\Hat{L}_{\pi}^g, \\
    \Tilde{L}_\pi^g &:= \frac{1}{b}\big\|\mLambda_\pi^{1/2}\big(\textstyle\sum_{i=1}^m\mI_{(di)} \mE\mE^{\top}\mI_{(di)}\big)\mLambda_\pi^{1/2}\big\|_2, \quad \;& \Tilde{L}^g = \max_{\pi}\Tilde{L}_{\pi}^g, 
\end{aligned}
\end{equation}
where $\mI_{(di)} = \sum_{j = bd(i - 1) + 1}^{bdi}\mI_{j}$.

To compare $\hat L_\pi^g$ and $L := \max_{i \in [n]}L_i$, we make use of the Kronecker product with notation $\otimes$ defined by 
\begin{align*}
    \mA \otimes \mB = \begin{bmatrix}
        A_{11}\mB & \cdots & A_{1n}\mB \\
        \vdots & \ddots & \vdots \\
        A_{m1}\mB & \cdots & A_{nn}\mB
    \end{bmatrix}
\end{align*}
for two matrices $\mA \in \R^{m \times n}$ and $\mB \in \R^{p \times q}$. The following lemma states a useful fact for the Kronecker product. 
\begin{lemma}\label{lem:eig-kro}
For square matrices $\mA$ and $\mB$ of sizes $p$ and $q$ and with eigenvalues $\lambda_i$ ($i \in [p]$) and $\mu_j$ ($j \in [q]$) respectively, the eigenvalues of $\mA \otimes \mB$ are $\lambda_i\mu_j$ for $i \in [p], j \in [q]$. 
\end{lemma}
We now use the following chain of inequalities to compare $\hat L_\pi^g$ and $L$ for any permutation $\pi$ of $[n]$:  
\begin{align*}
    \hat L_\pi^g = \frac{1}{mn}\Big\|\mLambda_\pi^{1/2}\Big(\sum_{i = 1}^m\mI_{bd(i - 1)\uparrow}\mE\mE^\top\mI_{bd(i - 1)\uparrow}\Big)\mLambda_\pi^{1/2}\Big\|_2 \leq \frac{1}{n}\Big\|\mLambda^{1/2}\mE\mE^\top\mLambda^{1/2}\Big\|_2 = \frac{1}{n}\Big\|(\vl_\pi\vl_\pi^\top) \otimes \mI_d\Big\|_2
    \overset{(\romannumeral1)}{=} \;& \frac{1}{n}\sum_{i = 1}^n L_i \leq L, 
\end{align*}
where we define $\vl_\pi = \big(\sqrt{L_{\pi_1}}, \sqrt{L_{\pi_2}}, \dots, \sqrt{L_{\pi_n}}\big)^\top$. For $(\romannumeral1)$, we use Lemma~\ref{lem:eig-kro} and notice that the eigenvalues of $\mI_d$ all equal $1$, while the largest eigenvalue of $\vl_\pi\vl_\pi^\top = \|\vl\|_2^2 = \sum_{i = 1}^n L_i$, so the operator norm of $(\vl_k\vl_k^\top) \otimes \mI_d$ is $\sum_{i = 1}^n L_i$. 

To compare $\tilde L_\pi^g$ and $L$, we notice that $\mLambda_\pi^{1/2}\big(\textstyle\sum_{i=1}^m\mI_{(di)} \mE\mE^{\top}\mI_{(di)}\big)\mLambda_\pi^{1/2} = \textstyle\sum_{i=1}^m\mI_{(di)}\mLambda_\pi^{1/2}\mE\mE^{\top}\mLambda_\pi^{1/2}\mI_{(di)}$ is a block diagonal matrix whose operator norm is the maximum of the operator norms over its diagonal block submatrices, so we have 
\begin{align*}
    \tilde L_\pi^g = \frac{1}{b}\max_{i \in [m]}\Big\|\mI_{(di)}\mLambda_\pi^{1/2}\mE\mE^{\top}\mLambda_\pi^{1/2}\mI_{(di)}\Big\| = \frac{1}{b}\max_{i \in [m]}\Big\|\mI_{(di)}\big((\vl_\pi\vl_\pi^\top) \otimes \mI_d\big)\mI_{(di)}\Big\| \overset{(\romannumeral1)}{=} \max_{i \in [m]}\frac{1}{b}\sum_{j = 1}^b L_{\pi_{b(i - 1) + j}} \leq L, 
\end{align*}
where for $(\romannumeral1)$ we use Lemma~\ref{lem:eig-kro} for each submatrix $(\vl_\pi^{(i)}\vl_\pi^{(i) \top}) \otimes \mI_d$ and $\vl_\pi^{(i)} = (0, \dots, 0, \sqrt{L_{\pi_{b(i - 1) + 1}}}, \dots, \sqrt{L_{\pi_{bi}}}, 0, \dots, 0)^\top$. Similar to the case of generalized linear models, the inequality is tight when $b = 1$ but can be loose for other values of $b$. 

To prove the convergence rate for shuffled SGD in this general setting, we follow the proof strategy in Section~\ref{sec:smooth} and first derive an upper bound $\mathrm{Gap}^\vv(\vx_k, \vy_*)$ for some fixed $\vv$ to be set later, as summarized in the following lemma.
\begin{restatable}{lemma}{gapBound-general}
\label{lemma:gap-bound-general}
Under Assumption~\ref{assp:fi}, for any $k \in [K]$, the iterates $\{\vy_k^{(i)}\}_{i = 1}^m$ and $\{\vx_{k - 1, i}\}_{i = 1}^{m + 1}$ generated by Algorithm~\ref{alg:PD-shuffled-sgd-general} satisfy 
\begin{equation}\label{eq:gap-bound-general}
\begin{aligned}
    \gE_k \leq \;& \frac{\eta_k}{n}\sum_{i = 1}^m\innp{\mE_b^\top \vy_k^{(i)}, \vx_{k} - \vx_{k - 1, i + 1}} + \frac{\eta_k}{n}\sum_{i = 1}^m\innp{\mE_b^\top\big(\vv_k^{(i)} - \vy_k^{(i)}\big), \vx_k - \vx_{k - 1, i}} \\
    & - \frac{\eta_k}{2n}\|\vy_k - \vv_k\|_{\mLambda_k^{-1}}^2 - \frac{\eta_k}{2n}\|\vy_k - \vy_{*, k}\|_{\mLambda_k^{-1}}^2 - \frac{b}{2n}\sum_{i = 1}^m\|\vx_{k - 1, i + 1} - \vx_{k - 1, i}\|^2, 
\end{aligned}
\end{equation}
where $\gE_k := \eta_k \mathrm{Gap}^\vv(\vx_k, \vy_*)
+ \frac{b}{2n}\|\vx_* - \vx_{k}\|_2^2 - \frac{b}{2n} \|\vx_* - \vx_{k - 1}\|_2^2$. 
\end{restatable}
\begin{proof}
We first note that based on Line 6
of Alg.~\ref{alg:PD-shuffled-sgd-general}, we have 
\begin{align*}
    \innp{\mE^\top_b\vy^{(i)}, \vx_{k - 1, i}} - \sum_{j = 1}^b f_{\pi_{b(i - 1) + j}^{(k)}}^{*}(\vy^{j}) = \sum_{j = 1}^b\Big(\innp{\vy^j, \vx_{k - 1, i}} - f_{\pi_{b(i - 1) + j}^{(k)}}^{*}(\vy^{j})\Big).
\end{align*}
Since the max problem defining $\vy_k$ is separable, we have for $b(i - 1) + 1 \leq j \leq bi$ and $i \in [m]$
\begin{align*}
    \vy_k^j = \argmax_{\vy^j \in \R^d}\Big\{\innp{\vy^j, \vx_{k - 1, i}} - f_{\pi_{j}^{(k)}}^{*}(\vy^{j})\Big\}, 
\end{align*}
which leads to $\vx_{k - 1, i} \in \partial f_{\pi_j^{(k)}}^*(\vy_k^j)$. 
Further, since each component function $f_j^*$ is $\frac{1}{L_j}$-strongly convex thus for $b(i - 1) + 1 \leq j \leq bi$, we also have
\begin{align*}
    f_{\pi_j^{(k)}}^*(\vv_k^{j}) \geq f_{\pi_j^{(k)}}^*(\vy_k^{j}) + \innp{\vx_{k - 1, i}, \vv_k^{j} - \vy_k^{j}} + \frac{1}{2L_{\pi_j^{(k)}}}\|\vv_k^{j} - \vy_k^{j}\|^2, 
\end{align*}
which leads to 
\begin{align*}
    \gL(\vx_k, \vv) = \;& \frac{1}{n}\sum_{i = 1}^m\Big(\innp{\mE_b^\top\vv_k^{(i)}, \vx_{k - 1, i}} - \sum_{j = b(i - 1) + 1}^{bi} f_{\pi_{j}^{(k)}}^*(\vv_k^{j})\Big) + \frac{1}{n}\sum_{i = 1}^m\innp{\mE_b^\top\vv_k^{(i)}, \vx_k - \vx_{k - 1, i}} \\
    \leq \;& \frac{1}{n}\sum_{i = 1}^m\Big(\innp{\mE_b^\top\vy_k^{(i)}, \vx_{k - 1, i}} - \sum_{j = b(i - 1) + 1}^{bi} f_{\pi_j^{(k)}}^*(\vy_k^{j})\Big) + \frac{1}{n}\sum_{i = 1}^m\innp{\mE_b^\top\vv_k^{(i)}, \vx_k - \vx_{k - 1, i}} - \frac{1}{2n}\|\vy_k - \vv_k\|_{\mLambda_k^{-1}}^2.
\end{align*}
Using the same argument, as $\vx_* \in \partial f_i^*(\vy_*^{i})$ for $i \in [n]$, we have 
\begin{align*}
    f_{\pi_i^{(k)}}^*(\vy_k^{i}) \geq f_{\pi_i^{(k)}}^*(\vy_{*, k}^{i}) + \innp{\vx_*, \vy_k^{i} - \vy_{*, k}^{i}} + \frac{1}{2L_{\pi_i^{(k)}}}\|\vy_k^{i} - \vy_{*, k}^{i}\|^2. 
\end{align*}
Thus, 
\begin{align*}
    \;& \gL(\vx_*, \vy_*) \\
    = \;& \frac{1}{n}\sum_{i = 1}^m\Big(\innp{\mE_b^\top\vy_{*, k}^{(i)}, \vx_*} - \sum_{j = b(i - 1) + 1}^{bi}f_{\pi_j^{(k)}}^*(\vy_{*, k}^{j})\Big) \\
    \geq \;& \frac{1}{n}\sum_{i = 1}^m\Big(\innp{\mE_b^\top\vy_k^{(i)}, \vx_*} - \sum_{j = b(i - 1) + 1}^{bi}f_{\pi_j^{(k)}}^*(\vy_k^{j})\Big) + \frac{1}{2n}\|\vy_k - \vy_{*, k}\|_{\mLambda_k^{-1}}^2 \\
    = \;& \frac{1}{n}\sum_{i = 1}^m\Big(\innp{\mE_b^\top\vy_k^{(i)}, \vx_*} + \frac{b}{2\eta_k}\|\vx_* - \vx_{k - 1, i}\|^2 - \frac{b}{2\eta_k}\|\vx_* - \vx_{k - 1, i}\|^2 - \sum_{j = b(i - 1) + 1}^{bi}f_{\pi_j^{(k)}}^*(\vy_k^{j})\Big) + \frac{1}{2n}\|\vy_k - \vy_{*, k}\|_{\mLambda_k^{-1}}^2.
\end{align*}
Using the updating scheme of $\vx_{k - 1, i + 1}$ and noticing that $\phi_k^i(\vx) = \innp{\mE_b^\top\vy_k^{(i)}, \vx} + \frac{b}{2\eta_k}\|\vx - \vx_{k - 1, i}\|^2$ is $\frac{b}{\eta_k}$-strongly convex and minimized at $\vx_{k - 1, i + 1}$, we have 
\begin{align*}
    \innp{\mE_b^\top\vy_k^{(i)}, \vx_*} + \frac{b}{2\eta_k}\|\vx_* - \vx_{k - 1, i}\|^2 \geq \innp{\mE_b^\top\vy_k^{(i)}, \vx_{k - 1, i + 1}} + \frac{b}{2\eta_k}\|\vx_{k - 1, i + 1} - \vx_{k - 1, i}\|^2 + \frac{b}{2\eta_k}\|\vx_{k - 1, i + 1} - \vx_*\|^2,
\end{align*}
which leads to 
\begin{align*}
    \gL(\vx_*, \vy_*) \geq \;& \frac{1}{n}\sum_{i = 1}^m\Big(\innp{\mE_b^\top\vy_k^{(i)}, \vx_{k - 1, i + 1}} + \frac{b}{2\eta_k}\|\vx_{k - 1, i + 1} - \vx_{k - 1, i}\|^2 - \sum_{j = b(i - 1) + 1}^{bi}f_{\pi_j^{(k)}}^*(\vy_k^{i})\Big) \\
    & + \frac{b}{2n\eta_k}\sum_{i = 1}^m\Big(\|\vx_{k - 1, i + 1} - \vx_*\|^2 -\|\vx_{k - 1, i} - \vx_*\|^2\Big) + \frac{1}{2n}\|\vy_k - \vy_{*, k}\|_{\mLambda_k^{-1}}^2 \\
    = \;& \frac{1}{n}\sum_{i = 1}^m\Big(\innp{\mE_b^\top\vy_k^{(i)}, \vx_{k - 1, i + 1}} + \frac{b}{2\eta_k}\|\vx_{k - 1, i + 1} - \vx_{k - 1, i}\|^2 - \sum_{j = b(i - 1) + 1}^{bi}f_{\pi_j^{(k)}}^*(\vy_k^{j})\Big) \\
    & + \frac{b}{2n\eta_k}\Big(\|\vx_{k} - \vx_*\|^2 - \|\vx_{k - 1} - \vx_*\|^2\Big) + \frac{1}{2n}\|\vy_k - \vy_{*, k}\|_{\mLambda_k^{-1}}^2.
\end{align*}
Hence, combining the bounds on $\gL(\vx_k, \vv)$ and $\gL(\vx_*, \vy_*)$ and letting
\begin{align*}
    \gE_k := \eta_k\big(\gL(\vx_k, \vv) - \gL(\vx_*, \vy_*)\big) + \frac{b}{2n}\|\vx_k - \vx_*\|^2 - \frac{b}{2n}\|\vx_{k - 1} - \vx_*\|^2, 
\end{align*}
 we obtain 
\begin{align*}
    \gE_k \leq \;& \frac{\eta_k}{n}\sum_{i = 1}^m\innp{\mE_b^\top\vy_k^{(i)}, \vx_{k - 1, i} - \vx_{k - 1, i + 1}} + \frac{\eta_k}{n}\sum_{i = 1}^m\innp{\mE_b^\top\vv_k^{(i)}, \vx_k - \vx_{k - 1, i}} \\
    & - \frac{\eta_k}{2n}\|\vy_k - \vv_k\|_{\mLambda_k^{-1}}^2 - \frac{\eta_k}{2n}\|\vy_k - \vy_{*, k}\|_{\mLambda_k^{-1}}^2 - \frac{b}{2n}\sum_{i = 1}^m\|\vx_{k - 1, i + 1} - \vx_{k - 1, i}\|^2 \\
    = \;& \frac{\eta_k}{n}\sum_{i = 1}^m\innp{\mE_b^\top\vy_k^{(i)}, \vx_{k} - \vx_{k - 1, i + 1}} + \frac{\eta_k}{n}\sum_{i = 1}^m\innp{\mE_b^\top\big(\vv_k^{(i)} - \vy_k^{(i)}\big), \vx_k - \vx_{k - 1, i}} \\
    & - \frac{\eta_k}{2n}\|\vy_k - \vv_k\|_{\mLambda_k^{-1}}^2 - \frac{\eta_k}{2n}\|\vy_k - \vy_{*, k}\|_{\mLambda_k^{-1}}^2 - \frac{b}{2n}\sum_{i = 1}^m\|\vx_{k - 1, i + 1} - \vx_{k - 1, i}\|^2, 
\end{align*}
thus completing the proof.
\end{proof}
We note that the first inner product term $\gT_1 := \frac{\eta_k}{n}\sum_{i = 1}^m\innp{\mE_b^\top\vy_k^{(i)}, \vx_{k} - \vx_{k - 1, i + 1}}$ in Eq.~\eqref{eq:gap-bound-general} can be cancelled by the last negative term $- \frac{b}{2n}\sum_{i = 1}^m\|\vx_{k - 1, i + 1} - \vx_{k - 1, i}\|^2$ therein, using the results of Lemma~\ref{lemma:first-inner}. In the following subsections, we continue our analysis and handle the remaining terms in Eq.~\eqref{eq:gap-bound-general} according to different shuffling and derive the final complexity.

\subsection{Random Reshuffling/Shuffle-Once Schemes}
We introduce the following lemma to bound the second inner product term $\gT_2 := \frac{\eta_k}{n}\sum_{i = 1}^m\innp{\mE_b^\top\big(\vv_k^{(i)} - \vy_k^{(i)}\big), \vx_k - \vx_{k - 1, i}}$ in Lemma~\ref{lemma:gap-bound-general} when there are random permutations.
\begin{restatable}{lemma}{secondInner-general}
\label{lemma:second-inner-general}
Under Assumptions~\ref{assp:fi}~and~\ref{assp:vr-fi}, for any $k \in [K]$, the iterates $\{\vy_k^{(i)}\}_{i = 1}^m$ and $\{\vx_{k - 1, i}\}_{i = 1}^{m + 1}$ generated by Algorithm~\ref{alg:PD-shuffled-sgd} with uniformly random shuffling (RR or SO) satisfy 
\begin{equation*}
\begin{aligned}
    \E[\gT_{2}] 
    \leq \;& \E\Big[\frac{\eta_k^3 n\Hat{L}^g_{\pi^{(k)}}\Tilde{L}^g_{\pi^{(k)}}}{b^2}\|\vy_k - \vy_{*, k}\|^2_{\mLambda_k^{-1}} + \frac{\eta_k}{2n}\|\vv_k - \vy_{k}\|_{\mLambda_k^{-1}}^2\Big] + \frac{\eta_k^3\Tilde{L}^g(n - b)(n + b)}{6b^2(n - 1)}\sigma_*^2, 
\end{aligned}
\end{equation*}
where $\gT_2 := \frac{\eta_k}{n}\sum_{i = 1}^m\innp{\mE_b^\top\big(\vv_k^{(i)} - \vy_k^{(i)}\big), \vx_k - \vx_{k - 1, i}}$.
\end{restatable}
\begin{proof}
First note that $\vx_k - \vx_{k - 1, i} = \sum_{j = i}^m(\vx_{k - 1, j + 1} - \vx_{k - 1, j}) = -\frac{\eta_k}{b} \sum_{j = i}^m \mE_b^\top\vy_k^{(j)} = -\frac{\eta_k}{b} \mE^\top \mI_{bd(i - 1)\uparrow}\vy_k$, so we have 
\begin{align*}
    \frac{\eta_k}{n}\sum_{i = 1}^m\innp{\mE_b^\top\big(\vv_k^{(i)} - \vy_k^{(i)}\big), \vx_k - \vx_{k - 1, i}} = \;& \frac{\eta_k}{n}\sum_{i = 1}^m\innp{\mE_b^\top\big(\vv_k^{(i)} - \vy_k^{(i)}\big), \sum_{j = i}^m(\vx_{k - 1, j + 1} - \vx_{k - 1, j})} \\
    = \;& -\frac{\eta_k^2}{bn}\sum_{i = 1}^m\innp{\mE^\top \mI_{(di)}(\vv_k - \vy_k), \mE^\top \mI_{bd(i - 1)\uparrow}\vy_k} \\
    = \;& \underbrace{-\frac{\eta_k^2}{bn}\sum_{i = 1}^m\innp{\mE^\top \mI_{(di)}(\vv_k - \vy_k), \mE^\top \mI_{bd(i - 1)\uparrow}(\vy_k - \vy_{*, k})}}_{\gI_1} \\
    & \underbrace{- \frac{\eta_k^2}{bn}\sum_{i = 1}^m\innp{\mE^\top \mI_{(di)}(\vv_k - \vy_k), \mE^\top \mI_{bd(i - 1)\uparrow}\vy_{*, k}}}_{\gI_2}. 
\end{align*}
For the term $\gI_1$, we use Young's inequality with $\alpha > 0$ to be set later and obtain 
\begin{align}
    \gI_1 = \;& -\frac{\eta_k^2}{bn}\sum_{i = 1}^m\innp{\mE^\top \mI_{(di)}(\vv_k - \vy_k), \mE^\top \mI_{bd(i - 1)\uparrow}(\vy_k - \vy_{*, k})} \notag\\
    \leq \;& \frac{\eta_k^2}{2bn\alpha}\sum_{i = 1}^m\|\mE^\top \mI_{(di)}(\vv_k - \vy_k)\|^2 + \frac{\eta_k^2\alpha}{2bn}\sum_{i = 1}^m\|\mE^\top \mI_{bd(i - 1)\uparrow}(\vy_k - \vy_{*, k})\|^2. \label{eq:gen-I1-1}
\end{align}
Further, notice that 
\begin{align}
    \;& \frac{\eta_k^2\alpha}{2bn}\sum_{i = 1}^m\|\mE^\top \mI_{bd(i - 1)\uparrow}(\vy_k - \vy_{*, k})\|^2 \notag\\
    = \;& \frac{\eta_k^2\alpha}{2bn}\sum_{i = 1}^m(\vy_k - \vy_{*, k})^\top\mI_{bd(i - 1)\uparrow}\mE\mE^\top\mI_{bd(i - 1)\uparrow}(\vy_k - \vy_{*, k}) \notag\\
    = \;& \frac{\eta_k^2\alpha}{2bn}(\vy_k - \vy_{*, k})^\top\Big(\sum_{i = 1}^m\mI_{bd(i - 1)\uparrow}\mE\mE^\top\mI_{bd(i - 1)\uparrow}\Big)(\vy_k - \vy_{*, k}) \notag\\
    = \;& \frac{\eta_k^2\alpha}{2bn}(\vy_k - \vy_{*, k})^\top\mLambda_k^{-1/2}\mLambda_k^{1/2}\Big(\sum_{i = 1}^m\mI_{bd(i - 1)\uparrow}\mE\mE^\top\mI_{bd(i - 1)\uparrow}\Big)\mLambda_k^{1/2}\mLambda_k^{-1/2}(\vy_k - \vy_{*, k}) \notag\\
    \leq \;& \frac{\eta_k^2\alpha}{2bn}\Big\|\mLambda_k^{1/2}\Big(\sum_{i = 1}^m\mI_{bd(i - 1)\uparrow}\mE\mE^\top\mI_{bd(i - 1)\uparrow}\Big)\mLambda_k^{1/2}\Big\|_2\|\vy_k - \vy_{*, k}\|_{\mLambda_k^{-1}}^2 = \frac{\eta_k^2m\alpha}{2b}\hat L_{\pi^{(k)}}^g\|\vy_k - \vy_{*, k}\|_{\mLambda_k^{-1}}^2, \label{eq:gen-I1-2}
\end{align}
where for the last inequality we use Cauchy-Schwarz inequality. Using the same argument, we can bound 
\begin{align}
    \frac{\eta_k^2}{2bn\alpha}\sum_{i = 1}^m\|\mE^\top \mI_{(di)}(\vv_k - \vy_k)\|^2 &\leq \frac{\eta_k^2}{2bn\alpha}\Big\|\mLambda_k^{1/2}\Big(\sum_{i = 1}^m\mI_{(di)}\mE\mE^\top\mI_{(di)}\Big)\mLambda_k^{1/2}\Big\|_2\|\vv_k - \vy_{k}\|_{\mLambda_k^{-1}}^2\notag \\
    &= \frac{\eta_k^2}{2n\alpha}\tilde L_{\pi^{(k)}}^g\|\vv_k - \vy_{k}\|_{\mLambda_k^{-1}}^2. \label{eq:gen-I1-3}
\end{align}
Thus, combining \eqref{eq:gen-I1-1}--\eqref{eq:gen-I1-3} and choosing $\alpha = 2\eta_k\tilde L^g_{\pi^{(k)}}$, we obtain 
\begin{align*}
    \gI_1 \leq \frac{\eta_k^3 m \hat L^g_{\pi^{(k)}}\tilde L^g_{\pi^{(k)}}}{b}\|\vy_k - \vy_{*, k}\|_{\mLambda_k^{-1}}^2 + \frac{\eta_k}{4n}\|\vv_k - \vy_k\|_{\mLambda_k^{-1}}^2. 
\end{align*}
For the term $\gI_2$, we again apply Young's inequality with $\beta > 0$ to be set later and obtain 
\begin{align*}
    \gI_2 = \;& - \frac{\eta_k^2}{bn}\sum_{i = 1}^m\innp{\mE^\top \mI_{(di)}(\vv_k - \vy_k), \mE^\top \mI_{bd(i - 1)\uparrow}\vy_{*, k}} \\
    \leq \;& \frac{\eta_k^2\beta}{2bn}\sum_{i = 1}^m\|\mE^\top \mI_{bd(i - 1)\uparrow}\vy_{*, k}\|^2 + \frac{\eta_k^2}{2bn\beta}\sum_{i = 1}^m\|\mE^\top \mI_{(di)}(\vv_k - \vy_k)\|^2 \\
    \leq \;& \frac{\eta_k^2\beta}{2bn}\sum_{i = 1}^m\|\mE^\top \mI_{bd(i - 1)\uparrow}\vy_{*, k}\|^2 + \frac{\eta_k^2\tilde L^g_{\pi^{(k)}}}{2n\beta}\|\vv_k - \vy_k\|^2_{\mLambda_{k}^{-1}}.
\end{align*}
Choosing $\beta = 2\eta_k\tilde{L}^g$ and using the fact that $\tilde L_{\pi^{(k)}}^g \leq \tilde L^g$, we have 
\begin{align*}
    \gI_2 \leq \frac{\eta_k^3\tilde L^g}{bn}\sum_{i = 1}^m\|\mE^\top \mI_{bd(i - 1)\uparrow}\vy_{*, k}\|^2 + \frac{\eta_k}{4n}\|\vv_k - \vy_k\|^2_{\mLambda_{k}^{-1}}.
\end{align*}
Hence, combining the above two estimates with $m = n/b$, we have 
\begin{align*}
     \gT_2
     \leq \frac{\eta_k^3\tilde L^g}{bn}\sum_{i = 1}^m\|\mE^\top \mI_{bd(i - 1)\uparrow}\vy_{*, k}\|^2 + \frac{\eta_k^3 n \hat L^g_{\pi^{(k)}} \tilde L^g_{\pi^{(k)}}}{b^2}\|\vy_k - \vy_{*, k}\|_{\mLambda_k^{-1}}^2 + \frac{\eta_k}{2n}\|\vv_k - \vy_k\|^2_{\mLambda_{k}^{-1}}.
\end{align*}
First, consider the RR scheme. Taking conditional expectation on both sides w.r.t.\ the randomness up to but not including $k$-th epoch, we have 
\begin{align*}
    \E_k[\gT_2] \leq \;& \frac{\eta_k^3\tilde L^g}{bn}\E_k\Big[\sum_{i = 1}^m\|\mE^\top \mI_{bd(i - 1)\uparrow}\vy_{*, k}\|^2\Big] + \E_k\Big[\frac{\eta_k^3 n \hat L^g_{\pi^{(k)}}\tilde L^g_{\pi^{(k)}}}{b^2}\|\vy_k - \vy_{*, k}\|_{\mLambda_k^{-1}}^2 + \frac{\eta_k}{2n}\|\vv_k - \vy_k\|^2_{\mLambda_{k}^{-1}}\Big].
\end{align*}
For the first term, since the only randomness comes from the permutation $\pi^{(k)}$, we can proceed as in the proof of Lemma~\ref{lemma:second-inner} and obtain 
\begin{align*}
    \frac{\eta_k^3\tilde L^g}{bn}\E_k\Big[\sum_{i = 1}^m\|\mE^\top \mI_{bd(i - 1)\uparrow}\vy_{*, k}\|^2\Big] \overset{(\romannumeral1)}{=} \;& \frac{\eta_k^3\tilde L^g}{bn}\sum_{i = 1}^m\E_{\pi^{(k)}}\Big[\|\mE^\top \mI_{bd(i - 1)\uparrow}\vy_{*, k}\|^2\Big] \\
    = \;& \frac{\eta_k^3\tilde L^g}{bn}\sum_{i = 1}^m (n - b(i - 1))^2\E_{\pi^{(k)}}\Big[\Big\|\frac{\mE^\top \mI_{bd(i - 1)\uparrow}\vy_{*, k}}{n - b(i - 1)}\Big\|^2\Big] \\
    \overset{(\romannumeral2)}{\leq} \;& \frac{\eta_k^3\tilde L^g}{bn}\sum_{i = 1}^m (n - b(i - 1))^2\frac{b(i - 1)}{(n - b(i - 1))(n - 1)}\sigma_*^2 \\
    = \;& \frac{\eta_k^3\tilde L^g (n - b)(n + b)}{6b^2(n - 1)}\sigma_*^2, 
\end{align*}
where we use the linearity of expectation for $(\romannumeral1)$, and $(\romannumeral2)$ is due to Lemma~\ref{lem:batch-vr} and the definition $\sigma_*^2 = \frac{1}{n}\sum_{i = 1}^n\|\nabla f_i(\vx_*)\|^2 = \frac{1}{n}\sum_{i = 1}^n\|\vy_*^i\|^2$. Then taking expectation w.r.t.\ all randomness on both sides, we obtain 
\begin{align*}
    \E[\gT_2] \leq \;& \E\Big[\frac{\eta_k^3 n \hat L^g_{\pi^{(k)}}\tilde L^g_{\pi^{(k)}}}{b^2}\|\vy_k - \vy_{*, k}\|_{\mLambda_k^{-1}}^2 + \frac{\eta_k}{2n}\|\vv_k - \vy_k\|^2_{\mLambda_{k}^{-1}}\Big] + \frac{\eta_k^3\tilde L^g (n - b)(n + b)}{6b^2(n - 1)}\sigma_*^2.
\end{align*}
Finally, we remark that the above argument for bounding the term $\frac{\eta_k^3\tilde L^g}{bn}\E_k\Big[\sum_{i = 1}^m\|\mE^\top \mI_{bd(i - 1)\uparrow}\vy_{*, k}\|^2\Big]$ also applies to the SO scheme, in which case there is only one random permutation at the very beginning that induces the randomness.
\end{proof}
To derive the final convergence rate and complexity, we follow the same analysis as in the proof for Theorem~\ref{thm:convergence}. We state the results in the following theorem and provide the proof for completeness.
\begin{restatable}{theorem}{convergenceGeneral}
\label{thm:convergence-general}
Under Assumptions~\ref{assp:fi}~and~\ref{assp:vr-fi}, if $\eta_k \leq \frac{b}{n\sqrt{2\Hat{L}^g_{\pi^{(k)}}\Tilde{L}^g_{\pi^{(k)}}}}$ 
and $H_K = \sum_{k = 1}^K\eta_k$, the output $\vxh_K$ of 
Algorithm~\ref{alg:PD-shuffled-sgd-general} with uniformly random (SO or RR) shuffling satisfies 
\begin{equation*}
    \E[H_K(f(\Hat{\vx}_K) - f(\vx_*))] \leq \frac{b}{2n}\|\vx_0 - \vx_*\|_2^2 + \sum_{k=1}^K\frac{\eta_k^3\Tilde{L}^g(n - b)(n + b)}{6b^2(n - 1)}\sigma_*^2.
\end{equation*}
As a consequence, for any $\epsilon > 0,$ there exists a choice of a constant step size $\eta_k = \eta$ for which $\E[f(\Hat{\vx}_K) - f(\vx_*)] \leq \epsilon$ after $\gO\big(\frac{n\sqrt{\Hat{L}^g\Tilde{L}^g}\|\vx_0 - \vx_*\|_2^2}{\epsilon} + \sqrt{\frac{(n - b)(n + b)}{n(n - 1)}}\frac{\sqrt{n\Tilde{L}^g}\sigma_*\|\vx_0 - \vx_*\|_2^2}{\epsilon^{3/2}}\big)$ gradient queries.
\end{restatable}
\begin{proof}
Combining the bounds in Lemma~\ref{lemma:first-inner}~and~\ref{lemma:second-inner-general} and plugging them into Eq.~\eqref{eq:gap-bound-general}, we obtain 
\begin{equation*}
    \E[\gE_k] \leq \E\Big[\Big(\frac{\eta_k^3 n \Hat{L}^g_{\pi^{(k)}}\Tilde{L}^g_{\pi^{(k)}}}{b^2} - \frac{\eta_k}{2n}\Big)\|\vy_k - \vy_{*, k}\|^2_{\mLambda_k^{-1}}\Big] + \frac{\eta_k^3\Tilde{L}^g(n - b)(n + b)}{6b^2(n - 1)}\sigma_*^2.
\end{equation*}
For the stepsize $\eta_k$ such that $\eta_k \leq \frac{b}{n\sqrt{2\Hat{L}^g_{\pi^{(k)}}\Tilde{L}^g_{\pi^{(k)}}}}$, we have $\frac{\eta_k^3 n \Hat{L}^g_{\pi^{(k)}}\Tilde{L}^g_{\pi^{(k)}}}{b^2} - \frac{\eta_k}{2n} \leq 0$, thus 
\begin{equation*}
    \E[\gE_k] \leq \frac{\eta_k^3\Tilde{L}^g(n - b)(n + b)}{6b^2(n - 1)}\sigma_*^2. 
\end{equation*}
Using our definition of $\gE_k$ and telescoping from $k = 1$ to $K$, we have 
\begin{equation*}
    \E\Big[\sum_{k = 1}^K\eta_k\text{Gap}^\vv(\vx_{k}, \vy_*)\Big] \leq \frac{b}{2n}\|\vx_* - \vx_0\|_2^2 - \frac{b}{2n}\E[\|\vx_* - \vx_{K}\|_2^2] + \sum_{k = 1}^K\frac{\eta_k^3\Tilde{L}^g(n - b)(n + b)}{6b^2(n - 1)}\sigma_*^2.
\end{equation*}
Noticing that $\gL(\vx, \vv)$ is convex in $\vx$ for a fixed $\vv$, we have $\text{Gap}^\vv(\Hat\vx_{K}, \vy_*) \leq \sum_{k = 1}^K\eta_k\text{Gap}^\vv(\vx_{k}, \vy_*) / H_K$, where $\Hat{\vx}_K = \sum_{k=1}^K \eta_k \vx_{k} / H_K$ and $H_K = \sum_{k = 1}^K\eta_k$, which leads to  
\begin{equation*}
    \E\Big[H_K\text{Gap}^\vv(\Hat\vx_{K}, \vy_*)\Big] \leq \frac{b}{2n}\|\vx_0 - \vx_*\|_2^2 + \sum_{k=1}^K \frac{\eta_k^3\Tilde{L}^g(n - b)(n + b)}{6b^2(n - 1)}\sigma_*^2. 
\end{equation*}
Further choosing $\vv = \vy_{\Hat{\vx}_K}$, we obtain 
\begin{equation}\label{eq:final-bound-general}
    \E[H_K\big(f(\Hat{\vx}_K) - f(\vx_*)\big)] \leq \frac{b}{2n}\|\vx_0 - \vx_*\|_2^2 + \sum_{k=1}^K\frac{\eta_k^3\Tilde{L}^g(n - b)(n + b)}{6b^2(n - 1)}\sigma_*^2.
\end{equation}
To analyze the individual gradient oracle complexity, we choose constant stepsizes $\eta \leq \frac{b}{n\sqrt{2\Hat{L}^g\Tilde{L}^g}}$, then Eq.~\eqref{eq:final-bound-general} will become 
\begin{equation*}
\begin{aligned}
    \E[f(\Hat{\vx}_K) - f(\vx_*)] \leq \frac{b}{2n\eta K}\|\vx_0 - \vx_*\|_2^2 + \frac{\eta^2\Tilde{L}^g(n - b)(n + b)}{6b^2(n - 1)}\sigma_*^2.
\end{aligned}
\end{equation*}
Without loss of generality, we assume that $b \neq n$, otherwise the method and its analysis reduce to (full) gradient descent. We consider the following two cases: 
\begin{itemize}[leftmargin=*]
    \item ``Small $K$'' case: if $\eta = \frac{b}{n\sqrt{2\Hat{L}^g\Tilde{L}^g}} \leq \Big(\frac{3b^3(n - 1)\|\vx_0 - \vx_*\|^2_2}{n(n - b)(n + b)\Tilde{L}^g K\sigma_*^2}\Big)^{1/3}$, we have 
    \begin{equation*}
    \begin{aligned}
        \E[f(\Hat{\vx}_K) - f(\vx_*)] \leq \;& \frac{b}{2n\eta K}\|\vx_0 - \vx_*\|_2^2 + \frac{\eta^2\Tilde{L}^g(n - b)(n + b)}{6b^2(n - 1)}\sigma_*^2\\
        \leq \;& \frac{\sqrt{\Hat{L}^g\Tilde{L}^g}}{\sqrt{2}K}\|\vx_0 - \vx_*\|_2^2 + \frac{1}{2}\Big(\frac{(n - b)(n + b)}{n^2(n - 1)}\Big)^{1/3}\frac{(\Tilde{L}^g)^{1/3}\sigma_*^{2/3}\|\vx_0 - \vx_*\|^{4/3}_2}{3^{1/3}K^{2/3}}.
    \end{aligned}
    \end{equation*}
    \item ``Large $K$'' case: 
    if $\eta = \Big(\frac{3b^3(n - 1)\|\vx_0 - \vx_*\|^2_2}{n(n - b)(n + b)\Tilde{L}^g K\sigma_*^2}\Big)^{1/3} \leq \frac{b}{n\sqrt{2\Hat{L}^g\Tilde{L}^g}}$, we have 
    \begin{equation*}
    \begin{aligned}
        \E[f(\Hat{\vx}_K) - f(\vx_*)] \leq \;& \frac{b}{2n\eta K}\|\vx_0 - \vx_*\|_2^2 + \frac{\eta^2\Tilde{L}^g(n - b)(n + b)}{6b^2(n - 1)}\sigma_*^2 \\
        \leq \;& \Big(\frac{(n - b)(n + b)}{n^2(n - 1)}\Big)^{1/3}\frac{(\Tilde{L}^g)^{1/3}\sigma_*^{2/3}\|\vx_0 - \vx_*\|^{4/3}_2}{3^{1/3}K^{2/3}}. 
    \end{aligned}
    \end{equation*}
\end{itemize}
Combining these two cases by setting $\eta = \min\Big\{\frac{b}{n\sqrt{2\Hat{L}^g\Tilde{L}^g}} ,\, \Big(\frac{3b^3(n - 1)\|\vx_0 - \vx_*\|^2_2}{n(n - b)(n + b)\Tilde{L}^g K\sigma_*^2}\Big)^{1/3}\Big\}$, we obtain  
\begin{equation*}
    \E[f(\Hat{\vx}_K) - f(\vx_*)] \leq \frac{\sqrt{\Hat{L}^g\Tilde{L}^g}}{\sqrt{2}K}\|\vx_0 - \vx_*\|_2^2 + \Big(\frac{(n - b)(n + b)}{n^2(n - 1)}\Big)^{1/3}\frac{(\Tilde{L}^g)^{1/3}\sigma_*^{2/3}\|\vx_0 - \vx_*\|^{4/3}_2}{3^{1/3}K^{2/3}}.
\end{equation*}
Hence, to guarantee $\E[f(\Hat{\vx}_K) - f(\vx_*)] \leq \epsilon$ for $\epsilon > 0$, the total number of individual gradient evaluations will be 
\begin{equation*}
    nK \geq \max\Big\{\frac{n\sqrt{2\Hat{L}^g\Tilde{L}^g}\|\vx_0 - \vx_*\|_2^2}{\epsilon}, \Big(\frac{(n - b)(n + b)}{n - 1}\Big)^{1/2}\frac{2^{3/2}(\Tilde{L}^g)^{1/2}\sigma_*\|\vx_0 - \vx_*\|_2^2}{3^{1/2}\epsilon^{3/2}}\Big\}, 
\end{equation*}
as claimed.  
\end{proof}

\subsection{Incremental Gradient Descent (IG)}
In this subsection, we provide the convergence results for incremental gradient descent which does not involve random permutations. We first follow the process of Lemma~\ref{lemma:second-inner-IGD} and prove the technical lemma below to bound the term $\gT_2 := \frac{\eta_k}{n}\sum_{i = 1}^m\innp{\mE_b^\top\big(\vv_k^{(i)} - \vy_k^{(i)}\big), \vx_k - \vx_{k - 1, i}}$ in Eq.~\eqref{eq:gap-bound-general} of Lemma~\ref{lemma:gap-bound-general}. 
\begin{restatable}{lemma}{secondInnerIGD-general}
\label{lemma:second-inner-IGD-general}
For any $k \in [K]$, the iterates $\{\vy_k^{(i)}\}_{i = 1}^m$ and $\{\vx_{k - 1, i}\}_{i = 1}^{m + 1}$ generated by Algorithm~\ref{alg:PD-shuffled-sgd} with fixed data ordering satisfy 
\begin{equation}\label{eq:sec-inn-bnd-IGD-general}
\begin{aligned}
\gT_2 \leq \frac{\eta_k^3 n}{b^2}\Hat{L}^g_0\Tilde{L}^g_0\|\vy_k - \vy_{*}\|^2_{\mLambda^{-1}} + \frac{\eta_k}{2n}\|\vv - \vy_k\|^2_{\mLambda^{-1}} + \min\Big\{\frac{\eta_k^3 n}{b^2}\Hat{L}^g_0\Tilde{L}^g_0\|\vy_{*}\|^2_{\mLambda^{-1}}, \frac{\eta_k^3 (n - b)^2}{b^2}\Tilde{L}^g_0\sigma_*^2\Big\}. 
\end{aligned}
\end{equation}
\end{restatable}
\begin{proof}
Proceeding as in the proof of Lemma~\ref{lemma:second-inner-general}, we have 
\begin{align*}
    \frac{\eta_k}{n}\sum_{i = 1}^m\innp{\mE_b^\top\big(\vv^{(i)} - \vy_k^{(i)}\big), \vx_k - \vx_{k - 1, i}} = \;& \frac{\eta_k}{n}\sum_{i = 1}^m\innp{\mE_b^\top\big(\vv^{(i)} - \vy_k^{(i)}\big), \sum_{j = i}^m(\vx_{k - 1, j + 1} - \vx_{k - 1, j})} \\
    = \;& -\frac{\eta_k^2}{bn}\sum_{i = 1}^m\innp{\mE^\top \mI_{(di)}(\vv - \vy_k), \mE^\top \mI_{bd(i - 1)\uparrow}\vy_k} \\
    = \;& \underbrace{-\frac{\eta_k^2}{bn}\sum_{i = 1}^m\innp{\mE^\top \mI_{(di)}(\vv - \vy_k), \mE^\top \mI_{bd(i - 1)\uparrow}(\vy_k - \vy_{*})}}_{\gI_1} \\
    & \underbrace{- \frac{\eta_k^2}{bn}\sum_{i = 1}^m\innp{\mE^\top \mI_{(di)}(\vv - \vy_k), \mE^\top \mI_{bd(i - 1)\uparrow}\vy_{*}}}_{\gI_2}. 
\end{align*}
For both terms $\gI_1$ and $\gI_2$, we apply Young's inequality with $\alpha = 2\eta_k\tilde L_0^g$ and obtain 
\begin{align}
    \gI_1
    \leq \;& \frac{\eta_k^2\alpha}{2bn}\sum_{i = 1}^m\|\mE^{\top}\mI_{bd(i - 1)\uparrow}(\vy_k - \vy_{*})\|_2^2 + \frac{\eta_k^2}{2bn\alpha}\sum_{i = 1}^m\|\mE^{\top}\mI_{(di)}(\vv - \vy_k)\|_2^2 \notag\\
    \leq \;& \frac{\eta_k^2 n \alpha}{2b^2}\Hat{L}^g_0\|\vy_k - \vy_{*}\|^2_{\mLambda^{-1}} + \frac{\eta_k^2}{2n\alpha}\Tilde{L}^g_0\|\vv - \vy_k\|^2_{\mLambda^{-1}} \notag \\
    = \;& \frac{\eta_k^3 n}{b^2}\Hat{L}^g_0\Tilde{L}^g_0\|\vy_k - \vy_{*}\|^2_{\mLambda^{-1}} + \frac{\eta_k}{4n}\|\vv - \vy_k\|^2_{\mLambda^{-1}}, \label{eq:sec-inn-IGD-1-general}
\end{align}
and 
\begin{align}
    \gI_2
    \leq \;& \frac{\eta_k^2\alpha}{2bn}\sum_{i = 1}^m\|\mE^{\top}\mI_{bd(i - 1)\uparrow}\vy_{*}\|_2^2 + \frac{\eta_k^2}{2bn\alpha}\sum_{i = 1}^m\|\mE^{\top}\mI_{(di)}(\vv - \vy_k)\|_2^2 \notag \\
    \leq \;& \frac{\eta_k^2\alpha}{2bn}\sum_{i = 1}^m\|\mE^{\top}\mI_{bd(i - 1)\uparrow}\vy_{*}\|_2^2 + \frac{\eta_k^2}{2n\alpha}\Tilde{L}^g_0\|\vv - \vy_k\|^2_{\mLambda^{-1}} \notag \\
    = \;& \frac{\eta_k^3\Tilde{L}^g_0}{nb}\sum_{i = 1}^m\|\mE^{\top}\mI_{bd(i - 1)\uparrow}\vy_{*}\|_2^2 + \frac{\eta_k}{4n}\|\vv - \vy_k\|^2_{\mLambda^{-1}}. \label{eq:sec-inn-IGD-2-general}
\end{align}
We now show that the term $\frac{\eta_k^3\Tilde{L}^g_0}{nb}\sum_{i = 1}^m\|\mE^{\top}\mI_{bd(i - 1)\uparrow}\vy_{*}\|_2^2$ is no larger than  either $\frac{\eta_k^3 n}{b^2}\hat{L}^g_0\tilde{L}^g_0\|\vy_*\|^2_{\mLambda^{-1}}$ or $\frac{\eta_k^3(n - b)^2}{b^2}\tilde{L}^g_0\sigma_*^2$. This is trivial when $b = n$ as $\mE^\top \mI_{0\uparrow}\vy_* = \sum_{i = 1}^n \vy_*^i = \mathbf{0}$. When $b < n$, to show the former one, we have 
\begin{align*}
    \sum_{i = 1}^m\|\mE^{\top}\mI_{bd(i - 1)\uparrow}\vy_{*}\|_2^2 \leq \Big\|\mLambda^{1/2}\Big(\sum_{i = 1}^m\mI_{bd(i - 1)\uparrow}\mE\mE^{\top}\mI_{bd(i - 1)\uparrow}\Big)\mLambda^{1/2}\Big\|_2\|\vy_*\|^2_{\mLambda^{-1}} = mn\Hat{L}^g_0\|\vy_*\|^2_{\mLambda^{-1}} = \frac{n^2}{b}\Hat{L}^g_0\|\vy_*\|^2_{\mLambda^{-1}}.
\end{align*}
To prove the latter one, we notice that 
\begin{align*}
    \sum_{i = 1}^m\|\mE^{\top}\mI_{bd(i - 1)\uparrow}\vy_{*}\|_2^2 = \;& \sum_{i = 1}^m\Big\|\sum_{j = b(i - 1) + 1}^n\vy_*^j\Big\|_2^2 = \sum_{i = 0}^{m - 1}\Big\|\sum_{j = bi + 1}^n\vy_*^j\Big\|_2^2 = \sum_{i = 1}^{m - 1}\Big\|\sum_{j = bi + 1}^n\vy_*^j\Big\|_2^2 = \sum_{i = 1}^{m - 1}\Big\|\sum_{j = 1}^{bi}\vy_*^j\Big\|_2^2, 
\end{align*}
using the fact that $\sum_{i = 1}^n\vy_*^i = \mathbf{0}$. Then using Young's inequality we obtain 
\begin{align*}
\sum_{i = 1}^{m - 1}\Big\|\sum_{j = 1}^{bi}\vy_*^j\Big\|_2^2 \leq \;& \sum_{i = 1}^{m - 1}bi\sum_{j = 1}^{bi}\|\vy_*^j\|^2_2 \\
\leq \;& b(m - 1)\sum_{i = 1}^{m - 1}\sum_{j = 1}^{bi}\|\vy_*^j\|^2_2 \\
= \;& b(m - 1)\sum_{i = 1}^{m - 1}\sum_{j = b(i - 1) + 1}^{bi}(m - i)\|\vy_*^j\|^2_2 \\
\leq \;& b(m - 1)^2\sum_{i = 1}^{(m - 1)b}\|\vy_*^i\|^2_2.
\end{align*}
Further noticing that $\sum_{i = 1}^{(m - 1)b}\|\vy_*^i\|^2_2 \leq \sum_{i = 1}^n\|\vy_*^i\|^2 = n\sigma_*^2$, we have 
\begin{align*}
    \frac{\eta_k^3\Tilde{L}^g_0}{nb}\sum_{i = 1}^m\|\mE^{\top}\mI_{bd(i - 1)\uparrow}\vy_{*}\|_2^2 \leq \frac{\eta_k^3\Tilde{L}^g_0}{nb}b(m - 1)^2n\sigma_*^2 = \frac{\eta_k^3\Tilde{L}^g_0(n - b)^2}{b^2}\sigma_*^2. 
\end{align*}
The same bound also captures the case $b = n$ and leads to 
\begin{align}
    \frac{\eta_k^3\Tilde{L}^g_0}{nb}\sum_{i = 1}^m\|\mE^{\top}\mI_{bd(i - 1)\uparrow}\vy_{*}\|_2^2 \leq \min\Big\{\frac{\eta_k^3 n}{b^2}\Hat{L}^g_0\Tilde{L}^g_0\|\vy_{*}\|^2_{\mLambda^{-1}}, \frac{\eta_k^3 (n - b)^2}{b^2}\Tilde{L}^g_0\sigma_*^2\Big\}. \label{eq:sec-inn-IGD-3-general}
\end{align}
Hence, combining Eq.~\eqref{eq:sec-inn-IGD-1-general}--\eqref{eq:sec-inn-IGD-3-general}, we obtain 
\begin{align*}
    \gI_2 \leq \frac{\eta_k^3 n}{b^2}\Hat{L}^g_0\Tilde{L}^g_0\|\vy_k - \vy_{*}\|^2_{\mLambda^{-1}} + \frac{\eta_k}{2n}\|\vv - \vy_k\|^2_{\mLambda^{-1}} + \min\Big\{\frac{\eta_k^3 n}{b^2}\Hat{L}^g_0\Tilde{L}^g_0\|\vy_{*}\|^2_{\mLambda^{-1}}, \frac{\eta_k^3 (n - b)^2}{b^2}\Tilde{L}^g_0\sigma_*^2\Big\}, 
\end{align*}
which finishes the proof.
\end{proof}
We are now ready to state our convergence results for IGD in the following theorem, with its proof provided for completeness.
\begin{restatable}{theorem}{convergenceIGD-general}
\label{thm:convergence-IGD-general}
Under Assumptions~\ref{assp:fi}~and~\ref{assp:vr-fi}, if $\eta_k \leq \frac{b}{n\sqrt{2\Hat{L}^g_0\Tilde{L}^g_0}}$ and $H_K = \sum_{k = 1}^K\eta_k$, the output $\vxh_K$ of  Algorithm~\ref{alg:PD-shuffled-sgd-general} with a fixed permutation satisfies 
\begin{equation}\notag
    H_K\big(f(\Hat{\vx}_K) - f(\vx_*)\big) \leq \frac{b}{2n}\|\vx_0 - \vx_*\|_2^2 + \sum_{k=1}^K \min\Big\{\frac{\eta_k^3 n}{b^2}\Hat{L}^g_0\Tilde{L}^g_0\|\vy_{*}\|^2_{\mLambda^{-1}}, \frac{\eta_k^3 (n - b)^2}{b^2}\Tilde{L}^g_0\sigma_*^2\Big\}. 
\end{equation}
As a consequence, for any $\epsilon > 0,$ there exists a choice of a constant step size $\eta_k = \eta$ such that $f(\Hat{\vx}_K) - f(\vx_*) \leq \epsilon$ after $\gO\Big(\frac{n\sqrt{\Hat{L}^g_0\Tilde{L}^g_0}\|\vx_0 - \vx_*\|_2^2}{\epsilon} + \frac{\min\big\{\sqrt{n\Hat{L}^g_0\Tilde{L}^g_0}\|\vy_{*}\|_{\mLambda^{-1}},\, (n - b)\sqrt{\Tilde{L}^g_0}\sigma_*\big\}\|\vx_0 - \vx_*\|_2^2}{\epsilon^{3/2}}\Big)$ gradient queries. 
\end{restatable}
\begin{proof}
Combining the bounds in Lemma~\ref{lemma:first-inner}~and~\ref{lemma:second-inner-IGD-general} and plugging them into Eq.~\eqref{eq:gap-bound-general} in Lemma~\ref{lemma:gap-bound-general} without random permutations, we have 
\begin{equation*}
    \gE_k \leq \Big(\frac{\eta_k^3 n \Hat{L}^g_0\Tilde{L}^g_0}{b^2} - \frac{\eta_k}{2n}\Big)\|\vy_k - \vy_{*}\|^2_{\mLambda^{-1}} + \min\Big\{\frac{\eta_k^3 n}{b^2}\Hat{L}^g_0\Tilde{L}^g_0\|\vy_{*}\|^2_{\mLambda^{-1}}, \frac{\eta_k^3 (n - b)^2}{b^2}\Tilde{L}^g_0\sigma_*^2\Big\}.
\end{equation*}
If $\eta_k \leq \frac{b}{n\sqrt{2\Hat{L}^g_0\Tilde{L}^g_0}}$, we have $\frac{\eta_k^3 n \Hat{L}^g_0\Tilde{L}^g_0}{b^2} - \frac{\eta_k}{2n} \leq 0$, thus 
\begin{equation*}
    \gE_k \leq \min\Big\{\frac{\eta_k^3 n}{b^2}\Hat{L}^g_0\Tilde{L}^g_0\|\vy_{*}\|^2_{\mLambda^{-1}}, \frac{\eta_k^3 (n - b)^2}{b^2}\Tilde{L}^g_0\sigma_*^2\Big\}. 
\end{equation*}
Using the definition of $\gE_k$ and telescoping from $k = 1$ to $K$, we obtain  
\begin{equation*}
    \sum_{k = 1}^K\eta_k\text{Gap}^\vv(\vx_{k}, \vy_*) \leq \frac{b}{2n}\|\vx_* - \vx_0\|_2^2 - \frac{b}{2n}\|\vx_* - \vx_{K}\|_2^2 + \sum_{k = 1}^K\min\Big\{\frac{\eta_k^3 n}{b^2}\Hat{L}^g_0\Tilde{L}^g_0\|\vy_{*}\|^2_{\mLambda^{-1}}, \frac{\eta_k^3 (n - b)^2}{b^2}\Tilde{L}^g_0\sigma_*^2\Big\}.
\end{equation*}
Noticing that $\gL(\vx, \vv)$ is convex w.r.t.\ $\vx$, we have $\text{Gap}^\vv(\Hat\vx_{K}, \vy_*) \leq \sum_{k = 1}^K\eta_k\text{Gap}^\vv(\vx_{k}, \vy_*) / H_K$, where $\Hat{\vx}_K = \sum_{k=1}^K \eta_k \vx_{k} / H_K$ and $H_K = \sum_{k = 1}^K\eta_k$, so we obtain 
\begin{equation*}
    H_K\text{Gap}^\vv(\Hat\vx_{K}, \vy_*) \leq \frac{b}{2n}\|\vx_0 - \vx_*\|_2^2 + \sum_{k=1}^K \min\Big\{\frac{\eta_k^3 n}{b^2}\Hat{L}^g_0\Tilde{L}^g_0\|\vy_{*}\|^2_{\mLambda^{-1}}, \frac{\eta_k^3 (n - b)^2}{b^2}\Tilde{L}^g_0\sigma_*^2\Big\}, 
\end{equation*}
Further choosing $\vv = \vy_{\Hat{\vx}_K}$, we obtain 
\begin{equation}\label{eq:final-bound-IGD-general}
\begin{aligned}
    H_K\big(f(\Hat{\vx}_K) - f(\vx_*)\big) \leq \frac{b}{2n}\|\vx_0 - \vx_*\|_2^2 + \sum_{k=1}^K \min\Big\{\frac{\eta_k^3 n}{b^2}\Hat{L}^g_0\Tilde{L}^g_0\|\vy_{*}\|^2_{\mLambda^{-1}}, \frac{\eta_k^3 (n - b)^2}{b^2}\Tilde{L}^g_0\sigma_*^2\Big\}.
\end{aligned}
\end{equation}

To analyze the individual gradient oracle complexity, we choose constant stepsizes $\eta \leq \frac{b}{n\sqrt{2\Hat{L}^g_0\Tilde{L}^g_0}}$ and assume $b < n$ without loss of generality, then Eq.~\eqref{eq:final-bound-IGD-general} becomes 
\begin{equation*}
\begin{aligned}
    f(\Hat{\vx}_K) - f(\vx_*) \leq \;& \frac{b}{2n\eta K}\|\vx_0 - \vx_*\|_2^2 + \min\Big\{\frac{\eta^2 n}{b^2}\Hat{L}^g_0\Tilde{L}^g_0\|\vy_{*}\|^2_{\mLambda^{-1}}, \frac{\eta^2 (n - b)^2}{b^2}\Tilde{L}^g_0\sigma_*^2\Big\}.
\end{aligned}
\end{equation*}
When $\Hat{L}^g_0\|\vy_{*}\|^2_{\mLambda^{-1}} \leq \frac{(n - b)^2}{n}\sigma_*^2$, we set $\eta = \min\Big\{\frac{b}{n\sqrt{2\Hat{L}^g_0\Tilde{L}^g_0}},\, \Big(\frac{b^3\|\vx_0 - \vx_*\|^2_2}{2n^2\Hat{L}^g_0\Tilde{L}^g_0K\|\vy_*\|_{\mLambda^{-1}}^2}\Big)^{1/3}\Big\}$ and consider the following two possible cases: 
\begin{itemize}[leftmargin=*]
    \item ``Small $K$'' case: if $\eta = \frac{b}{n\sqrt{2\Hat{L}^g_0\Tilde{L}^g_0}} \leq \Big(\frac{b^3\|\vx_0 - \vx_*\|^2_2}{2n^2\Hat{L}^g_0\Tilde{L}^g_0K\|\vy_*\|_{\mLambda^{-1}}^2}\Big)^{1/3}$, we have 
    \begin{equation*}
    \begin{aligned}
        f(\Hat{\vx}_K) - f(\vx_*) \leq \;& \frac{b}{2n\eta K}\|\vx_0 - \vx_*\|_2^2 + \frac{\eta^2 n}{b^2}\Hat{L}^g_0\Tilde{L}^g_0\|\vy_{*}\|^2_{\mLambda^{-1}} \\
        \leq \;& \frac{\sqrt{\Hat{L}^g_0\Tilde{L}^g_0}}{\sqrt{2}K}\|\vx_0 - \vx_*\|_2^2 + \frac{\big(\Hat{L}^g_0\Tilde{L}^g_0\big)^{1/3}\|\vy_{*}\|_{\mLambda^{-1}}^{2/3}\|\vx_0 - \vx_*\|^{4/3}_2}{2^{2/3}n^{1/3}K^{2/3}}.
    \end{aligned}
    \end{equation*}
    \item ``Large $K$'' case: if $\eta = \Big(\frac{b^3\|\vx_0 - \vx_*\|^2_2}{2n^2\Hat{L}^g_0\Tilde{L}^g_0K\|\vy_*\|_{\mLambda^{-1}}^2}\Big)^{1/3} \leq \frac{b}{\sqrt{2\Hat{L}^g_0\Tilde{L}^g_0}}$, we have 
    \begin{equation*}
    \begin{aligned}
        f(\Hat{\vx}_K) - f(\vx_*) \leq \;& \frac{b}{2n\eta K}\|\vx_0 - \vx_*\|_2^2 + \frac{\eta^2 n}{b^2}\Hat{L}^g_0\Tilde{L}^g_0\|\vy_{*}\|^2_{\mLambda^{-1}} \\
        \leq \;& \frac{2^{1/3}\big(\Hat{L}^g_0\Tilde{L}^g_0\big)^{1/3}\|\vy_{*}\|_{\mLambda^{-1}}^{2/3}\|\vx_0 - \vx_*\|^{4/3}_2}{n^{1/3}K^{2/3}}.
    \end{aligned}
    \end{equation*}
\end{itemize}
Combining these two cases, we have 
\begin{equation*}
    f(\Hat{\vx}_K) - f(\vx_*) \leq \frac{\sqrt{\Hat{L}^g_0\Tilde{L}^g_0}}{\sqrt{2}K}\|\vx_0 - \vx_*\|_2^2 + \frac{2^{1/3}\big(\Hat{L}^g_0\Tilde{L}^g_0\big)^{1/3}\|\vy_{*}\|_{\mLambda^{-1}}^{2/3}\|\vx_0 - \vx_*\|^{4/3}_2}{n^{1/3}K^{2/3}}. 
\end{equation*}
Hence, to guarantee $\E[f(\Hat{\vx}_K) - f(\vx_*)] \leq \epsilon$ for $\epsilon > 0$, the total number of required individual gradient evaluations will be 
\begin{equation}\label{eq:complexity-IGD-1-general}
    nK \geq \max\Big\{\frac{n\sqrt{2\Hat{L}^g_0\Tilde{L}^g_0}\|\vx_0 - \vx_*\|_2^2}{\epsilon}, \frac{4n^{1/2}\big(\Hat{L}^g_0\Tilde{L}^g_0\big)^{1/2}\|\vy_{*}\|_{\mLambda^{-1}}\|\vx_0 - \vx_*\|_2^2}{\epsilon^{3/2}}\Big\}. 
\end{equation}
When $\frac{(n - b)^2}{n}\sigma_*^2 \leq \Hat{L}^g_0\|\vy_*\|^2_{\mLambda^{-1}}$, we set $\eta = \min\Big\{\frac{b}{n\sqrt{2\Hat{L}^g_0\Tilde{L}^g_0}},\, \Big(\frac{b^3\|\vx_0 - \vx_*\|^2_2}{2n(n - b)^2\Tilde{L}^g_0K\sigma_*^2}\Big)^{1/3}\Big\}$ and consider the two cases as below: 
\begin{itemize}
    \item ``Small $K$'' case: if $\eta = \frac{b}{n\sqrt{2\Hat{L}^g_0\Tilde{L}^g_0}} \leq \Big(\frac{b^3\|\vx_0 - \vx_*\|^2_2}{2n(n - b)^2\Tilde{L}^g_0K\sigma_*^2}\Big)^{1/3}$, we have 
    \begin{equation*}
    \begin{aligned}
        f(\Hat{\vx}_K) - f(\vx_*) \leq \;& \frac{b}{2n\eta K}\|\vx_0 - \vx_*\|_2^2 +  \frac{\eta^2 (n - b)^2}{b^2}\Tilde{L}^g_0\sigma_*^2 \\
        \leq \;& \frac{\sqrt{\Hat{L}^g_0\Tilde{L}^g_0}}{\sqrt{2}K}\|\vx_0 - \vx_*\|_2^2 + \frac{(n - b)^{2/3}(\Tilde{L}^g)^{1/3}\sigma_*^{2/3}\|\vx_0 - \vx_*\|_2^{4/3}}{2^{2/3}n^{2/3}K^{2/3}}.
    \end{aligned}
    \end{equation*}
    \item ``Large $K$'' case: if $\eta = \Big(\frac{b^3\|\vx_0 - \vx_*\|^2_2}{2n(n - b)^2\Tilde{L}^g_0K\sigma_*^2}\Big)^{1/3} \leq \frac{b}{n\sqrt{2\Hat{L}^g_0\Tilde{L}^g_0}}$, we have 
    \begin{equation*}
    \begin{aligned}
        f(\Hat{\vx}_K) - f(\vx_*) \leq \;& \frac{b}{2n\eta K}\|\vx_0 - \vx_*\|_2^2 +  \frac{\eta^2 (n - b)^2}{b^2}\Tilde{L}^g_0\sigma_*^2 \\
        \leq \;& \frac{2^{1/3}(n - b)^{2/3}(\Tilde{L}^g)^{1/3}\sigma_*^{2/3}\|\vx_0 - \vx_*\|_2^{4/3}}{n^{2/3}K^{2/3}}.
    \end{aligned}
    \end{equation*}
\end{itemize}
Combining these two cases, we obtain 
\begin{equation*}
    f(\Hat{\vx}_K) - f(\vx_*) \leq \frac{\sqrt{\Hat{L}^g_0\Tilde{L}^g_0}}{\sqrt{2}K}\|\vx_0 - \vx_*\|_2^2 + \frac{2^{1/3}(n - b)^{2/3}(\Tilde{L}^g)^{1/3}\sigma_*^{2/3}\|\vx_0 - \vx_*\|_2^{4/3}}{n^{2/3}K^{2/3}}.
\end{equation*}
To guarantee $\E[f(\Hat{\vx}_K) - f(\vx_*)] \leq \epsilon$ for $\epsilon > 0$, the total number of required individual gradient evaluations will be 
\begin{equation}\label{eq:complexity-IGD-2-general}
    nK \geq \max\Big\{\frac{n\sqrt{2\Hat{L}^g_0\Tilde{L}^g_0}\|\vx_0 - \vx_*\|_2^2}{\epsilon}, \frac{4(n - b)(\Tilde{L}^g_0)^{1/2}\sigma_*\|\vx_0 - \vx_*\|_2^2}{\epsilon^{3/2}}\Big\}.
\end{equation}
Combining Eq.~\eqref{eq:complexity-IGD-1-general} and Eq.~\eqref{eq:complexity-IGD-2-general}, we finally have 
\begin{equation*}
    nK \geq \frac{n\sqrt{2\Hat{L}^g_0\Tilde{L}^g_0}\|\vx_0 - \vx_*\|_2^2}{\epsilon} + \min\Big\{\frac{4n^{1/2}\big(\Hat{L}^g_0\Tilde{L}^g_0\big)^{1/2}\|\vy_{*}\|_{\mLambda^{-1}}\|\vx_0 - \vx_*\|_2^2}{\epsilon^{3/2}}, \frac{4(n - b)(\Tilde{L}^g_0)^{1/2}\sigma_*\|\vx_0 - \vx_*\|_2^2}{\epsilon^{3/2}}\Big\}, 
\end{equation*}
thus finishing the proof.
\end{proof}

\section{Additional Experiments}\label{appx:exp}
\subsection{Evaluations of $L / \Tilde{L}_\pi$ on Synthetic Gaussian Datasets}
We first study the gap between $\Tilde{L}_\pi$ and $L$ for different batch sizes $b$, as shown in Figure~\ref{fig:gaussian-compare-tilde}. As in Section~\ref{sec:exp}, we focus on their dependence on the data matrix, and assume that the loss functions $\ell_i$ all have the same smoothness constant. In this case, the ratio $L/\Tilde{L}_\pi$ that characterizes the gap between $\Tilde{L}_\pi$ and $L$ will become $L / \Tilde{L}_\pi = (\max_{1\leq i \leq n}\{\|\va_i\|_2^2\})/\big(\frac{1}{b}\|\sum_{j=1}^m\mI_{(j)}\mA_{{\pi}}\mA_{{\pi}}^{\top}\mI_{(j)}\|_2\big)$. In particular, we run experiments on standard Gaussian data of size $(n, d)$. We fix the dimension $d = 500$, and vary the number of samples with $n = 100, 500, 1000, 2000$. In Figure~\ref{fig:gaussian-compare-tilde}, we plot the ratio $L / \Tilde{L}_{\pi}$ versus the batch size $b$ for $100$ different random permutations $\pi$, where the dotted lines represent the mean values and the filled regions indicate the standard deviation of permutations. We observe that the ratio $L / \Tilde{L}_{\pi}$ is concentrated around its empirical mean and exhibits $b^\alpha$ ($\alpha \in [0.74, 0.87]$) growth as the batch size $b$ increases. In particular, if we choose $b = \sqrt{n}$, the ratio can be $\gO(n^{0.4})$.

\begin{figure}[t!]
    \hspace*{\fill}
    \subfloat[$n = 100$]{\includegraphics[width=0.4\textwidth]{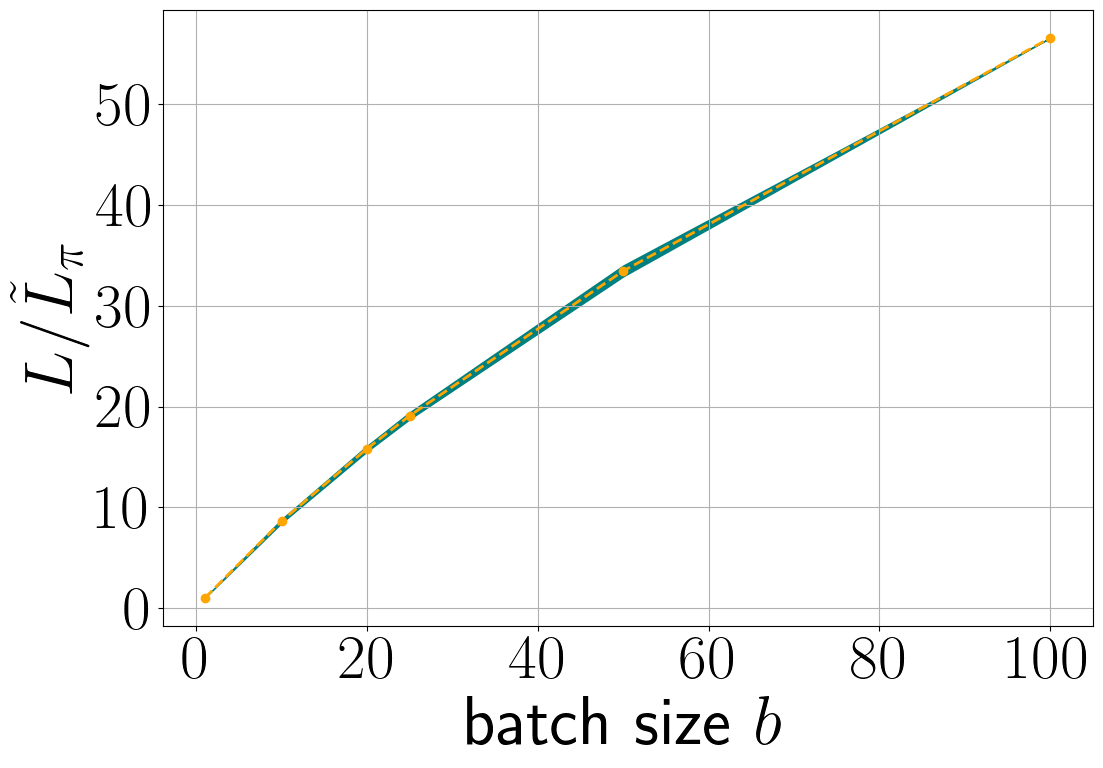}} \hfill
    \subfloat[$n = 500$]{\includegraphics[width=0.4\textwidth]{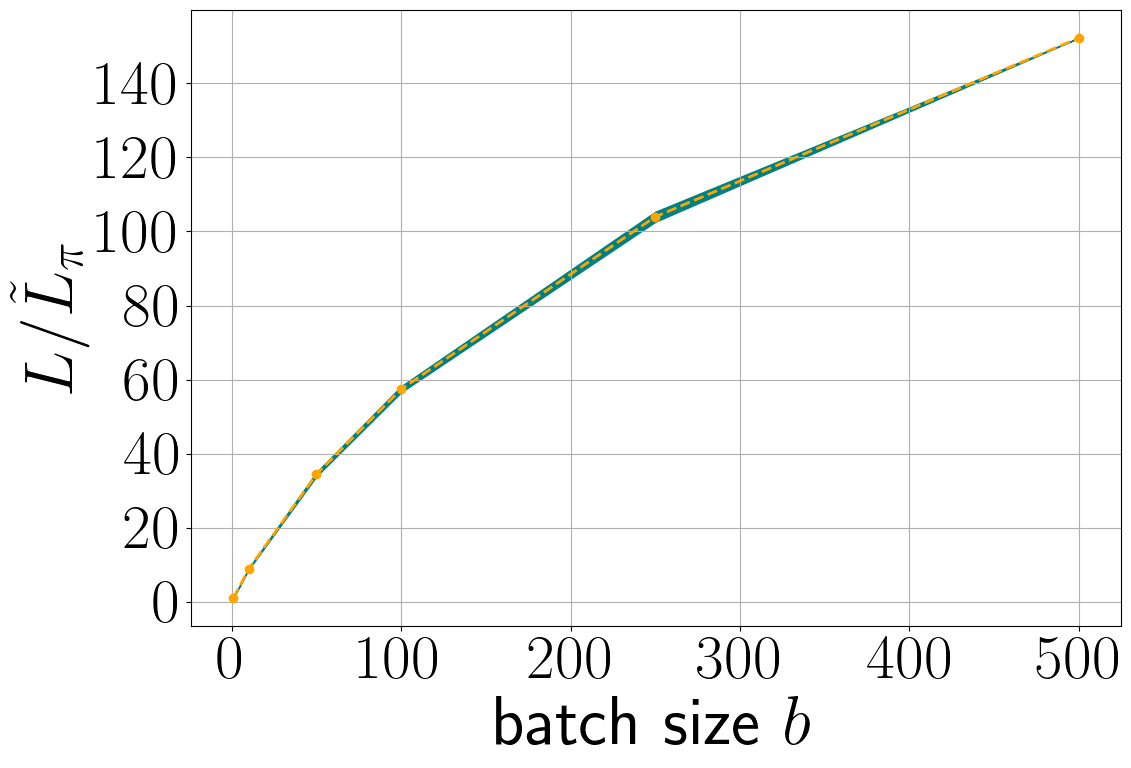}}\hspace*{\fill}\\
    \hspace*{\fill}
    \subfloat[$n = 1000$]{\includegraphics[width=0.4\textwidth]{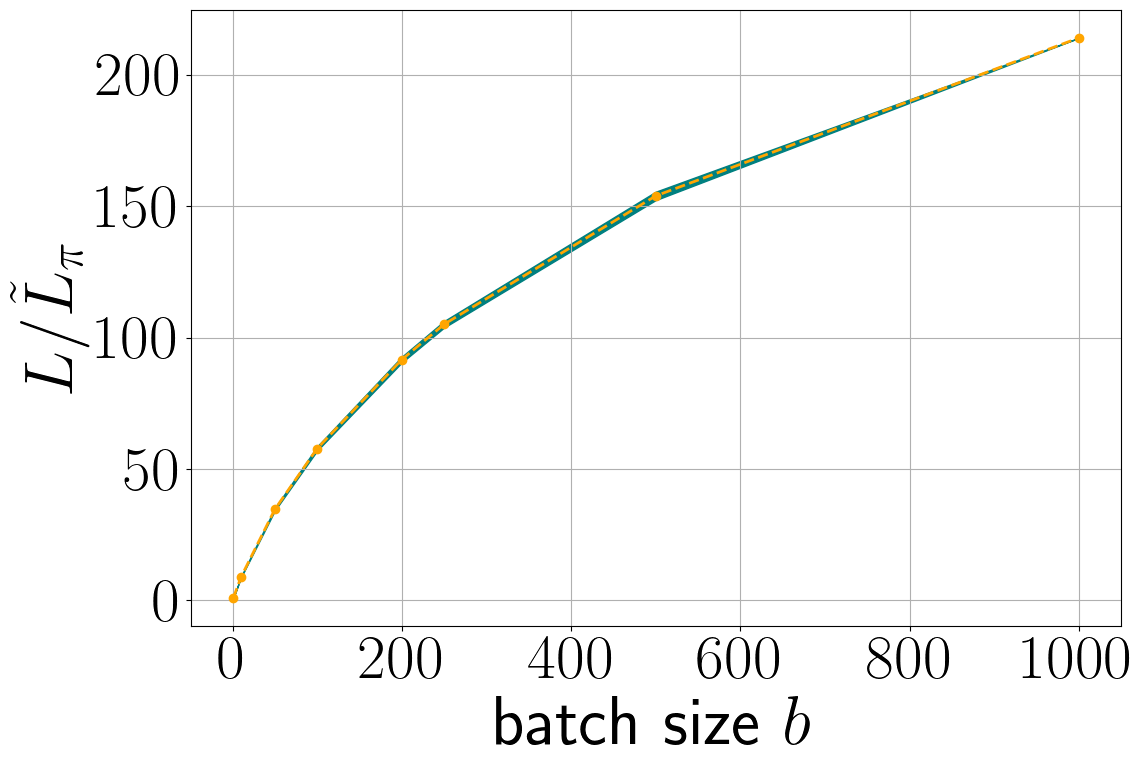}}\hfill
    \subfloat[$n = 2000$]{\includegraphics[width=0.4\textwidth]{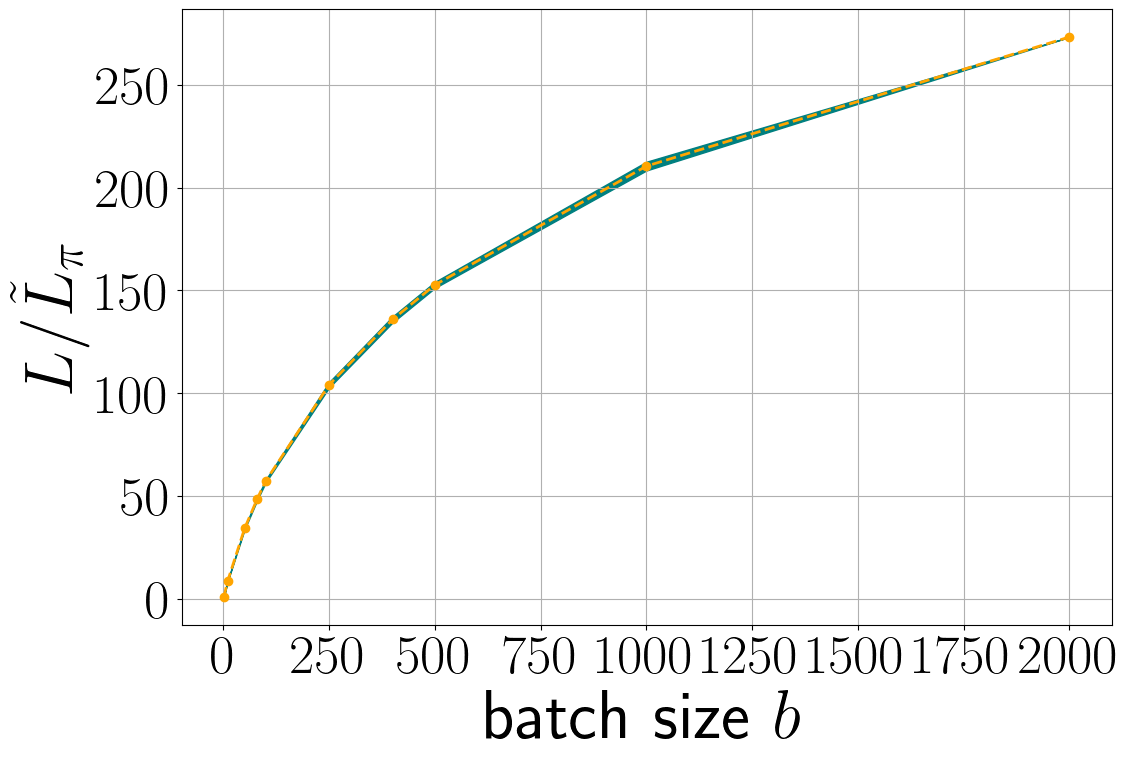}}\hspace*{\fill}
    \caption{Illustrations of $L / \Tilde{L}_\pi$ for different batch size $b$ on synthetic Gaussian data of size $(n, d)$.}
    \label{fig:gaussian-compare-tilde}
\end{figure}

\subsection{Distributions of $L / \hat{L}_\pi$}

In this subsection, we include histograms in Figure~\ref{fig:histograms-L} to illustrate the spread of $L / \hat{L}_\pi$ with respect to random permutations, for completeness. We observe that in all the examples $L / \hat{L}_\pi$ is concentrated around its empirical mean. The following plots are normalized, with y-axis representing the empirical probability density. The x-axis represents $L / \hat{L}_\pi$.

\begin{figure*}[ht!]
\captionsetup[subfigure]{labelformat=empty}
    \centering
    {\subfloat[\texttt{a1a}]{\includegraphics[width=0.31\textwidth]{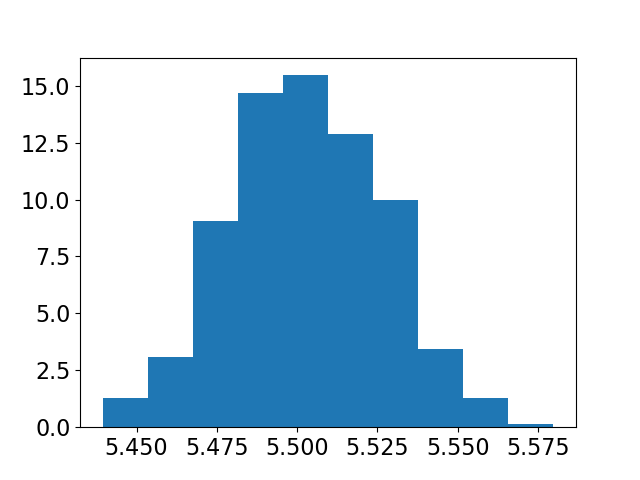}}}
    \hspace{\fill}
    {\subfloat[\texttt{a9a}]{\includegraphics[width=0.31\textwidth]{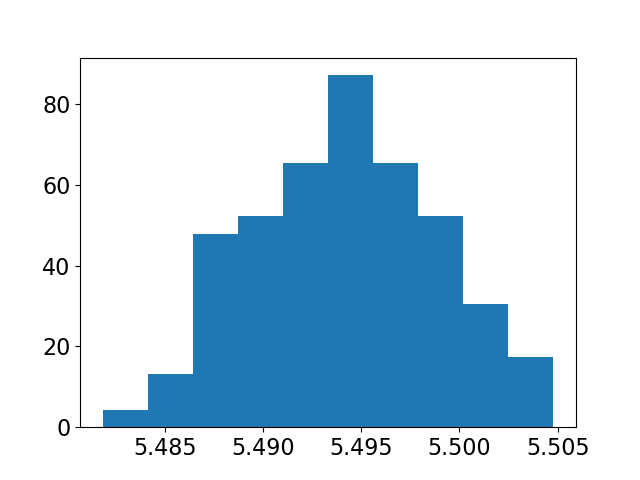}}}
    \hspace*{\fill}
    {\subfloat[\texttt{BBBC005}]{\includegraphics[width=0.31\textwidth]{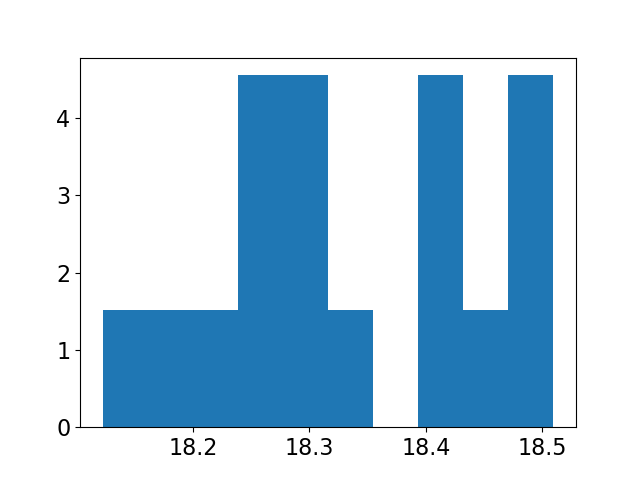}}}
    
    \hspace*{\fill}
    \vspace{-10pt}
    
    {\subfloat[\texttt{BBBC010}]{\includegraphics[width=0.31\textwidth]{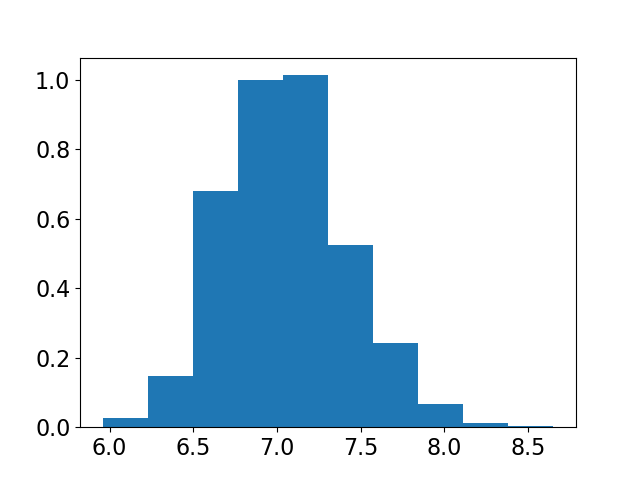}}}
    \hspace{\fill}
    {\subfloat[\texttt{CIFAR10}]{\includegraphics[width=0.31\textwidth]{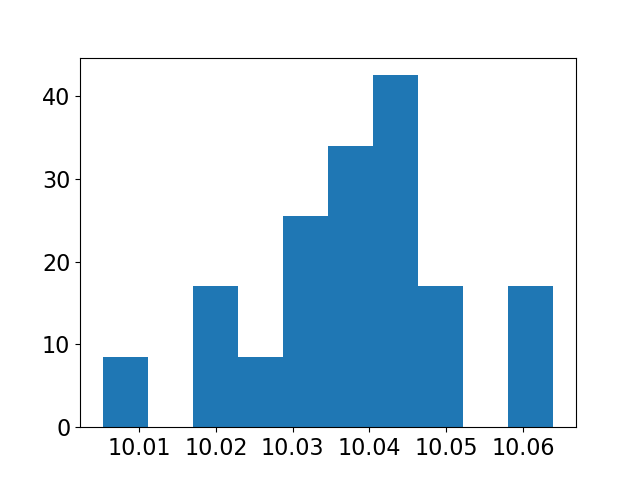}}}
    \hspace*{\fill}
    {\subfloat[\texttt{duke}]{\includegraphics[width=0.31\textwidth]{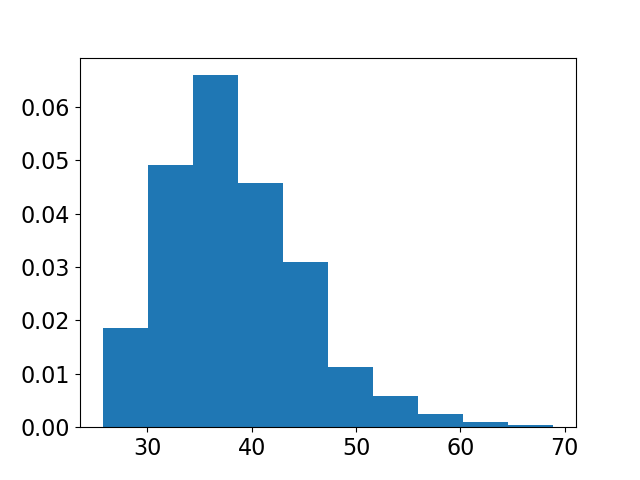}}}

    \hspace*{\fill}
    \vspace{-10pt}
    
    {\subfloat[\texttt{e2006train}]{\includegraphics[width=0.31\textwidth]{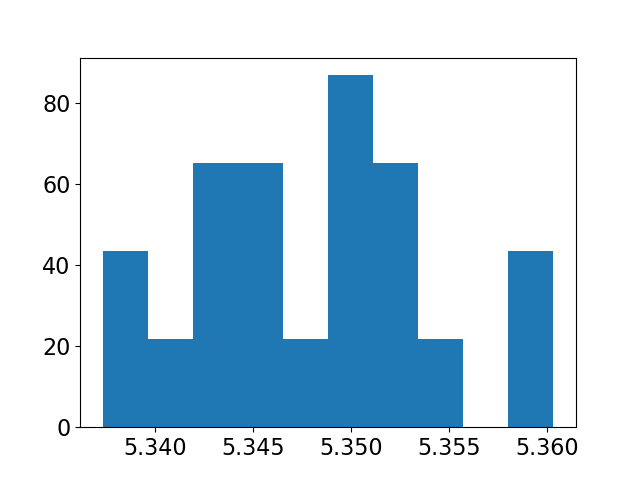}}}
    \hspace{\fill}
    {\subfloat[\texttt{gisette}]{\includegraphics[width=0.31\textwidth]{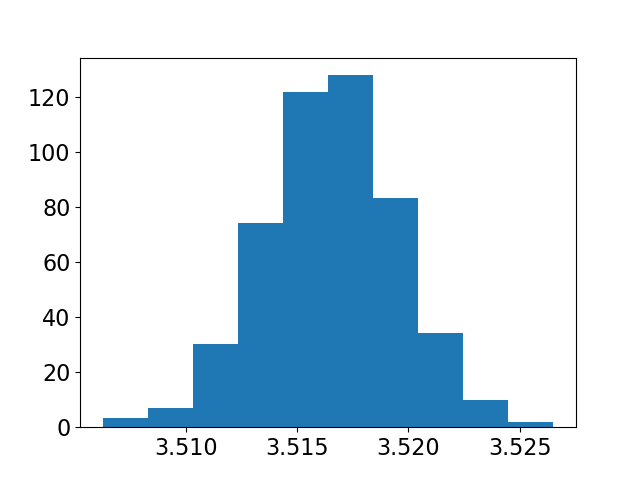}}}
    \hspace*{\fill}
    {\subfloat[\texttt{leu}]{\includegraphics[width=0.31\textwidth]{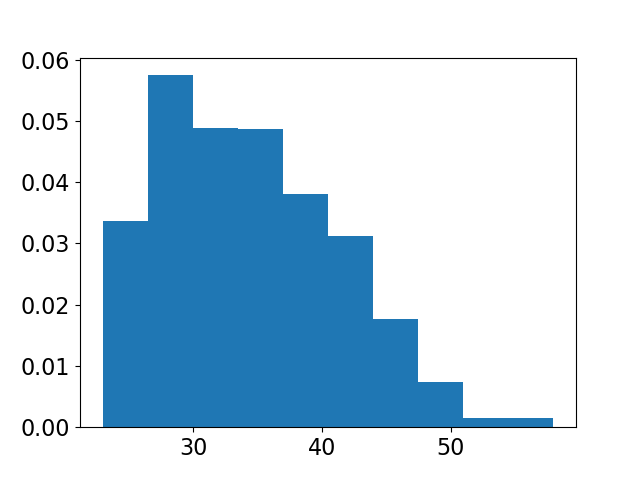}}}

    \hspace*{\fill}
    \vspace{-10pt}
    
    {\subfloat[\texttt{MNIST}]{\includegraphics[width=0.31\textwidth]{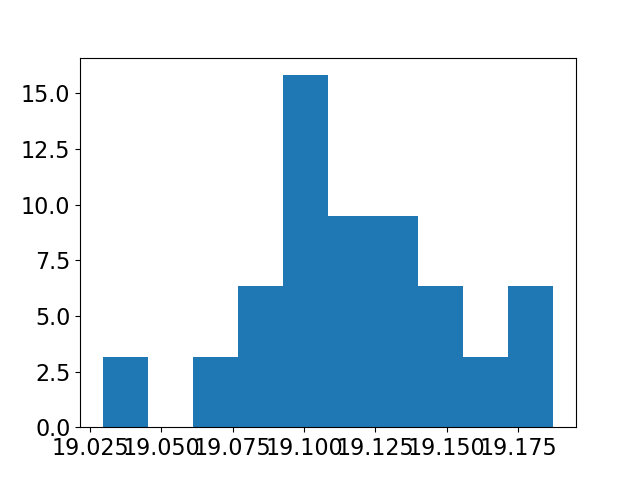}}}
    \hspace{\fill}
    {\subfloat[\texttt{news20}]{\includegraphics[width=0.31\textwidth]{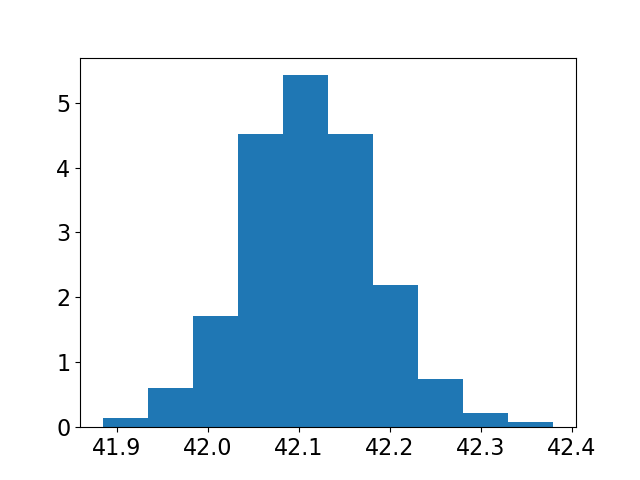}}}
    \hspace*{\fill}
    {\subfloat[\texttt{rcv1}]{\includegraphics[width=0.31\textwidth]{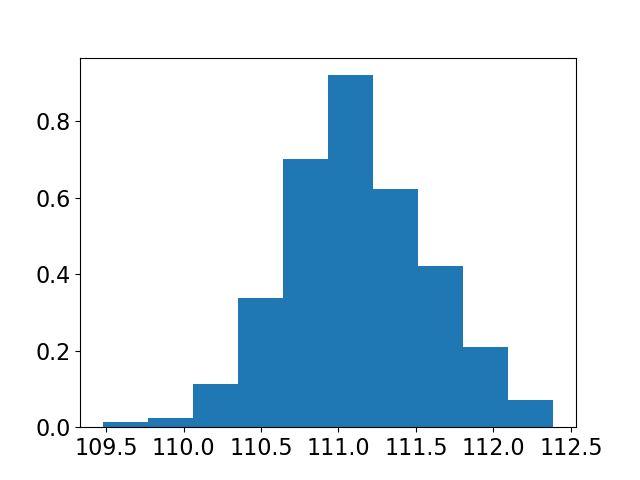}}}

    \hspace*{\fill}
    \vspace{-10pt}
    
    {\subfloat[\texttt{real-sim}]{\includegraphics[width=0.31\textwidth]{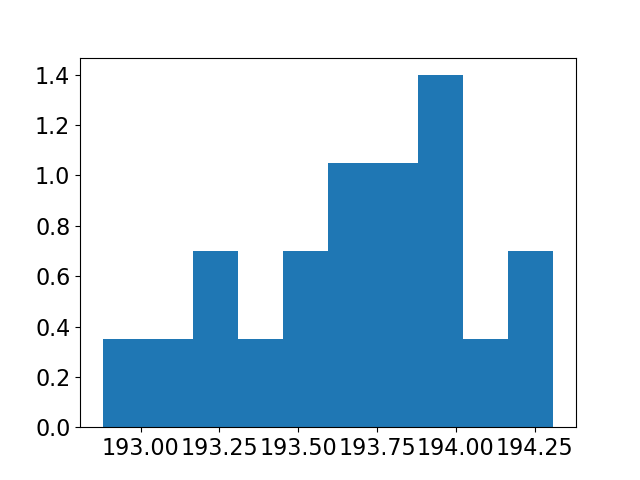}}}
    \hspace{\fill}
    {\subfloat[\texttt{sonar}]{\includegraphics[width=0.31\textwidth]{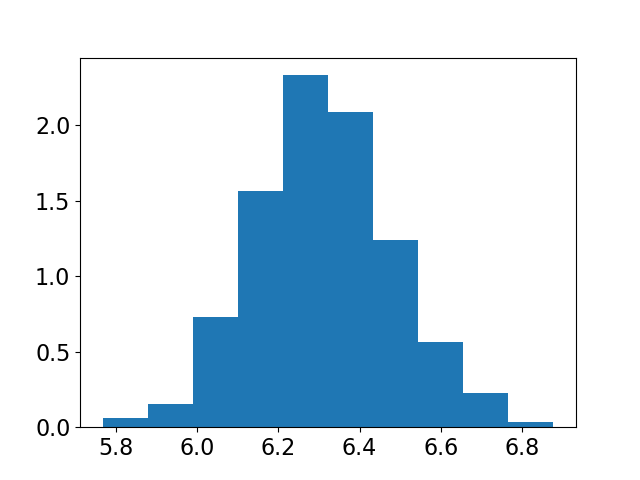}}}
    \hspace*{\fill}
    {\subfloat[\texttt{tmc2007}]{\includegraphics[width=0.31\textwidth]{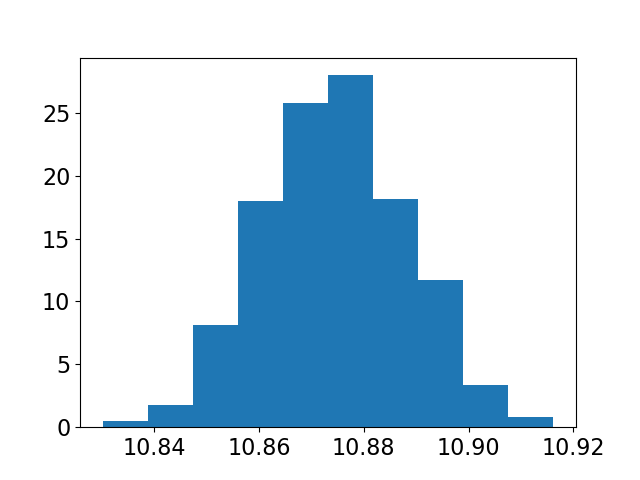}}}
    \caption{Visualization of the empirical distributions of $L / \hat{L}$ for $15$ large-scale datasets.}%
    \label{fig:histograms-L}
\end{figure*}

\end{document}